\documentclass[leqno,12pt]{amsart}

\usepackage{amssymb, amsthm, amscd, amsmath}
\usepackage{stmaryrd}
\usepackage{mathrsfs, mathtools}
\usepackage{epic,eepic}
\usepackage[all]{xy}
\usepackage{color}

\theoremstyle{plain}
\newtheorem{theorem}{Theorem}[section]
\newtheorem{lemma}[theorem]{Lemma}
\newtheorem{proposition}[theorem]{Proposition}

\newtheorem{problem}[theorem]{Problem}

\renewcommand{\theprob}

\renewcommand{\thequest}

\newtheorem{corollary}[theorem]{Corollary}
\theoremstyle{definition}
\newtheorem{definition}[theorem]{Definition}
\newtheorem{condition}[theorem]{Condition}
\newtheorem{example}[theorem]{Example}
\newtheorem{remarkA}{Remark}

\newtheorem{remarkB}{Remark}

\newcommand{\be}{\begin{enumerate}}
\newcommand{\ee}{\end{enumerate}}
\newcommand{\bcd}{\[\begin{CD}}
\newcommand{\ecd}{\end{CD}\]}
\newcommand{\bit}{\begin{itemize}}
\newcommand{\eit}{\end{itemize}}
\newcommand{\bq}{\begin{quote}}
\newcommand{\eq}{\end{quote}}
\newcommand{\bpf}{\begin{proof}}
\newcommand{\epf}{\end{proof}}

\newcommand{\mcB}{\mathcal{B}}

\newcommand{\mcE}{\mathcal{E}}
\newcommand{\mcF}{\mathcal{F}}

\newcommand{\mcI}{\mathcal{I}}

\newcommand{\mcL}{\mathcal{L}}
\newcommand{\mcM}{\mathcal{M}}

\newcommand{\mcO}{\mathcal{O}}

\newcommand{\mcX}{\mathcal{X}}
\newcommand{\mcY}{\mathcal{Y}}

\newcommand{\mbF}{\mathbb{F}}

\newcommand{\mbN}{\mathbb{N}}

\newcommand{\mbP}{\mathbb{P}}
\newcommand{\mbQ}{\mathbb{Q}}

\newcommand{\mbZ}{\mathbb{Z}}

\newcommand{\mfP}{\mathfrak{P}}

\newcommand{\msC}{\mathscr{C}}

\newcommand{\msN}{\mathscr{N}}

\newcommand{\msP}{\mathscr{P}}
\newcommand{\msQ}{\mathscr{Q}}

\newcommand{\msT}{\mathscr{T}}

\newcommand{\msV}{\mathscr{V}}

\newcommand{\SR}{\stackrel}
\newcommand{\mr}{\mathrm}

\newcommand{\defeq}{\SR{\mr{def}}{=}}

\newcommand{\spec}{\mathrm{Spec} \, }

\newcommand{\migi}{\rightarrow}

\newcommand{\isom}{\stackrel{\sim}{\migi}}

\newcommand{\migiinje}{\hookrightarrow}

\newcommand{\migisurj}{\twoheadrightarrow}

\newcommand{\invlim}{\varprojlim}
\newcommand{\chr}{ {\text{\rm char}} }

\newcommand{\et}{\text{\'et}}

\newcommand{\bp}{\bullet}
\newcommand{\muge}{\infty}

\newcommand{\tch}{\textcolor{black}}

\begin{document}


\title[Maximums of Generalized Hasse-Witt Invariants]{Maximums of generalized Hasse-Witt invariants and their applications to anabelian geometry}
\author{\sc Yu Yang}
\thanks{{\sc E-mail:} yuyang@kurims.kyoto-u.ac.jp \\ {\sc Address:} Research Institute for Mathematical Sciences, Kyoto University, Kyoto 606-8502, Japan}
\date{}
\maketitle



\begin{abstract}

Let $X^{\bp}=(X, D_{X})$ be a pointed stable curve of topological type $(g_{X}, n_{X})$ over an algebraically closed field of characteristic $p>0$. 

Under certain assumptions, we prove that, if $X^{\bp}$ is {\it component-generic}, then the first generalized Hasse-Witt invariant of {\it every} prime-to-$p$ cyclic admissible coverings of $X^{\bp}$ attains maximum. This result generalizes a result of S. Nakajima concerning the ordinariness of prime-to-$p$ cyclic \'etale coverings of smooth projective generic curves to the case of (possibly ramified) admissible coverings of (possibly singular) pointed stable curves. 

Moreover, we prove that, if $X^{\bp}$ is an {\it arbitrary} pointed stable curve, then there {\it exists} a prime-to-$p$ cyclic admissible covering of $X^{\bp}$ whose first generalized Hasse-Witt invariant attains maximum. This result generalizes a result of M. Raynaud concerning the new-ordinariness of prime-to-$p$ cyclic \'etale coverings of smooth projective curves to the case of (possibly ramified) admissible coverings of (possibly singular) pointed stable curves.  

As applications, we obtain {\it an anabelian formula for $(g_{X}, n_{X})$}, and prove that the field structures associated to inertia subgroups of marked points can be reconstructed group-theoretically from open continuous homomorphisms of admissible fundamental groups. Those results generalize A. Tamagawa's results concerning an anabelian formula for topological types and reconstructions of field structures associated to inertia subgroups of marked points of smooth pointed stable curves to the case of {\it arbitrary} pointed stable curves.

Keywords: pointed stable curve, admissible covering, generalized Hasse-Witt invariant, Raynaud-Tamagawa theta divisor, admissible fundamental group, anabelian geometry, positive characteristic.

Mathematics Subject Classification: Primary 14H30; Secondary 14G32.
\end{abstract}


\tableofcontents


\section{Introduction}\label{sec-1}

Let $X^{\bp}=(X, D_{X})$ be a pointed stable curve over an algebraically closed field $k$, where $X$ denotes the underlying curve, and $D_{X}$ denotes the (finite) set of marked points. Write $g_{X}$ for the genus of $X$ and $n_{X}$ for the cardinality $\#(D_{X})$ of $D_{X}$. We call $(g_{X}, n_{X})$ the topological type (or type for short) of $X^{\bp}$. By choosing a suitable base point of $X^{\bp}$, we have the admissible fundamental group $\Pi_{X^{\bp}}$ of $X^{\bp}$ (\ref{admfg}). The admissible fundamental groups of pointed stable curves are natural generalizations of the tame fundamental groups of  smooth pointed stable curves. In particular, $\Pi_{X^{\bp}}$ is isomorphic to the tame fundamental group of $X^{\bp}$ if $X^{\bp}$ is smooth over $k$.

Suppose that the characteristic $\chr(k)$ of $k$ is $0$ (resp. $p>0$). Then the structure of $\Pi_{X^{\bp}}$ (resp. the maximal prime-to-$p$ quotient $\Pi_{X^{\bp}}^{p'}$ of $\Pi_{X^{\bp}}$) is well-known, \tch{and} is isomorphic to the profinite completion (resp. prime-to-$p$ completion) of the topological fundamental group of a Riemann surface of type $(g_{X}, n_{X})$ (\ref{stradm}). In particular, $\Pi_{X^{\bp}}$ (resp. $\Pi_{X^{\bp}}^{p'}$) is a free profinite group with $2g_{X}+n_{X}-1$ generators if $n_{X}>0$. We see that the type $(g_{X}, n_{X})$ cannot be determined group-theoretically from the isomorphism class (as a profinite group) of $\Pi_{X^{\bp}}$ if $\chr(k)=0$ (resp. $\Pi_{X^{\bp}}^{p'}$ if $\chr(k)=p$).

\subsection{Fundamental groups of curves in positive characteristic}

\subsubsection{}
If $\chr(k)=p>0$, the situation is quite different from that in characteristic $0$, and the structure of $\Pi_{X^{\bp}}$ is no longer known. In the remainder of the introduction, we assume $\chr(k)=p>0$. In the case of positive characteristic, the admissible fundamental group $\Pi_{X^{\bp}}$ is very mysterious. Some developments of F. Pop-M. Sa$\rm \ddot{\i}$di (\cite{PS}), M. Raynaud (\cite{R2}), A. Tamagawa (\cite{T1}, \cite{T2}, \cite{T3}), and the author of the present paper (\cite{Y1}, \cite{Y2}, \cite{Y3}, \cite{Y5}, \cite{Y6}) showed evidence for very strong {\it anabelian} phenomena for curves over {\it algebraically closed fields of characteristic $p>0$.} In this situation, the Galois group of the base field is trivial, and the arithmetic fundamental group coincides with the geometric fundamental group, thus \tch{there is} a total absence of a Galois action of the base field. This kinds of anabelian \tch{phenomena} go beyond Grothendieck's anabelian geometry (since no Galois actions), and show that the admissible (or tame) fundamental group of a pointed stable curve over an algebraically closed field of characteristic $p$ must encode``{\it moduli}'' of the curve. This is the reason
why we do not have an explicit description of the admissible (or tame) fundamental group of
any pointed stable curve in positive characteristic.

\subsubsection{}
Furthermore, the theory developed in \cite{T2} and \cite{Y2} implies that the isomorphism class of $X^{\bp}$ (as a scheme) can possibly be determined by not only the isomorphism class of $\Pi_{X^{\bp}}$ but also the isomorphism class of the maximal pro-solvable quotient of $\Pi_{X^{\bp}}$. On the other hand, since all the admissible coverings (see Definition \ref{def-1-2}) of $X^{\bp}$ can be lifted to characteristic $0$ (\cite[Th\'eor\`eme 2.2 (c)]{V}), we obtain that $\Pi_{X^{\bp}}$ is topologically finitely generated. This implies that the isomorphism class of $\Pi_{X^{\bp}}$ is determined by the set of finite quotients of $\Pi_{X^{\bp}}$ (\cite[Proposition 16.10.6]{FJ}).  Then to understand the anabelian phenomena of curves in positive characteristic, we may ask the following question: {\it Which finite solvable \tch{groups} can appear as a quotient of $\Pi_{X^{\bp}}$?}

\subsubsection{}
Let $H \subseteq \Pi_{X^{\bp}}$ be an arbitrary open normal subgroup and $X^{\bp}_{H}=(X_{H}, D_{X_{H}})$ the pointed stable curve of type $(g_{X_{H}}, n_{X_{H}})$ over $k$ corresponding to $H$. We have an important invariant $\sigma_{X_{H}}$ associated to $X^{\bp}_H$ (or $H$) which is called {\it$p$-rank} (or {\it Hasse-Witt invariant}, see \ref{def-1-3}). Roughly speaking, $\sigma_{X_{H}}$ controls the finite quotients of $\Pi_{X^{\bp}}$ which are extensions of the group $\Pi_{X^{\bp}}/H$ by $p$-groups. 

Since the structures of maximal prime-to-$p$ quotients of admissible fundamental groups have been known, to find all the solvable quotients of $\Pi_{X^{\bp}}$, we need to compute the $p$-rank $\sigma_{X_{H}}$ when $\Pi_{X^{\bp}}/H$ is {\it abelian}. If $\Pi_{X^{\bp}}/H$ is a $p$-group, then $\sigma_{X_{H}}$ can be computed by using the Deuring-Shafarevich formula (\cite{C}) and the orders of the inertia groups of marked points of the covering $X^{\bp}_{H} \migi X^{\bp}$. If $\Pi_{X^{\bp}}/H$ is not a $p$-group, the situation of $\sigma_{X_{H}}$ is very complicated. Moreover, the Deuring-Shafarevich formula implies that, to compute $\sigma_{X_{H}}$, we only need to assume that $\Pi_{X^{\bp}}/H$ is a prime-to-$p$ group (i.e., the order of $\Pi_{X^{\bp}}/H$ is prime to $p$).

\subsection{Generalized Hasse-Witt invariants for generic curves}

\subsubsection{Nakajima, Zhang, and Ozman-Pries' results}
Firstly, let us consider the case of generic curves. Suppose that $n_{X}=0$, and that $X^{\bp}$ is {\it smooth} over $k$. If $X^{\bp}$ is a curve corresponding to a geometric generic point of the moduli space (i.e., a geometric generic curve), S. Nakajima (\cite[Proposition 4]{N}) proved that, if $\Pi_{X^{\bp}}/H$ is a prime-to-$p$ cyclic group, then $\sigma_{X_{H}}$ attains the maximum $g_{X_{H}}$ (i.e., $X_{H}^{\bp}$ is {\it ordinary}). Moreover, B. Zhang (\cite{Z}) generalized Nakajima's result to the case where $\Pi_{X^{\bp}}/H$ is an arbitrary abelian group. Recently, E. Ozman and R. Pries (\cite{OP}) generalized Nakajima's result to the case where $\Pi_{X^{\bp}}/H$ is a cyclic group with a prime order distinct from $p$, and where $X^{\bp}$ is a curve corresponding to a geometric point over a generic point of $p$-rank stratas of the moduli space. 

Let $m \in \mbN$ be an arbitrary positive natural number prime to $p$. In other words, the results of Nakajima, Zhang, and Ozman-Pries say that the {\it ``first" generalized Hasse-Witt invariant} (see \cite{N} or \ref{ghw} of the present paper) of {\it every} Galois {\it \'etale} covering of $X^{\bp}$ with Galois group $\mbZ/m\mbZ$ attains the maximum $g_{X}-1$. 

\subsubsection{}
The first main result of the present paper is as follows (see Theorem \ref{main-them-1} for a more precise statement):

\begin{theorem}\label{them-1-0}
Let $X^{\bp}$ be a component-generic pointed stable curve (see \ref{defcurves} for the definition) over an algebraically closed field $k$ of characteristic $p>0$. Then the ``first" generalized Hasse-Witt invariant of every prime-to-$p$ cyclic admissible covering of $X^{\bp}$ attains maximum under certain assumptions.
\end{theorem} 
\noindent
If $n_{X}=0$ and $X^{\bp}$ is smooth over $k$, then Theorem \ref{them-1-0} is equivalent to \cite[Proposition 4]{N}. Thus, Theorem \ref{them-1-0} generalizes Nakajima's result to the case of (possibly ramified) admissible coverings of (possibly singular) pointed stable curves. Moreover, by applying Theorem \ref{them-1-0},  we generalize \cite[Theorem 2]{N} to the case of tame coverings (see Proposition \ref{coro-3-13} (i)).

\subsection{Generalized Hasse-Witt invariants for arbitrary curves}
Next, let us consider the case where $X^{\bp}$ is an {\it arbitrary} pointed stable curve. Let $m$ be a positive natural number prime to $p$, and let $f^{\bp}: Y^{\bp} \migi X^{\bp}$ be a Galois admissible  covering over $k$ with Galois group $\mbZ/m\mbZ$ and $D$ the ramification divisor associated to $f^{\bp}$. Note that the degree $\text{deg}(D)$ of $D$ is divisible by $n$, and that $0\leq \text{deg}(D) \leq (n_{X}-1)m$ if $n_{X}\neq 0$.

\subsubsection{Raynaud and Tamagawa's results}\label{1.3.1}

If $X^{\bp}$ is not geometric generic, the computation of $\sigma_{X_{H}}$ is very difficult in general. Suppose that $X^{\bp}$ is {\it smooth} over $k$, and that $n_{X}=0$. Raynaud (\cite{R1}) developed his theory of theta divisors and proved that, if $m >>0$ is a natural number prime to $p$, then there {\it exists} a Galois {\it \'etale} \tch{covering} $f^{\bp}$ of $X^{\bp}$ with Galois group $\mbZ/m\mbZ$ whose ``first" generalized Hasse-Witt invariant attains the maximum $g_{X}-1$ (\cite[Th\'eor\`eme 4.3.1]{R1}). Moreover, as a consequence, Raynaud obtained that $\Pi_{X^{\bp}}$
is not a prime-to-$p$ profinite group. This is the first deep result concerning the global
structures of \'etale fundamental groups of curves over algebraically closed fields of characteristic $p>0$.

Suppose that $X^{\bp}$ is {\it smooth} over $k$, and that $n_{X}\geq 0$. The computations of generalized Hasse-invariants of admissible coverings of $X^{\bp}$ (i.e., tame coverings of $X^{\bp}$) are much more difficult than the case of $n_{X}=0$. Tamagawa observed that Raynaud's theory of theta divisors can be generalized to the case of tame coverings, and he established a tamely ramified version of the theory of Raynaud's theta divisors. 
By applying the theory of theta divisors, Tamagawa proved that, if $n_{X}\neq 0$ and $m>>0$, then there {\it exists} a  Galois admissible covering (i.e., Galois tame covering) $f^{\bp}$ of $X^{\bp}$ with Galois group $\mbZ/m\mbZ$ such that $\text{deg}(D)=m$, and that the ``first" generalized Hasse-Witt invariants of $f^{\bp}$ is as large as possible, namely equal to $g_{X}$.

\subsubsection{}
In the present paper, we study maximum generalized Hasse-Witt invariants for {\it arbitrary} (possibly singular) pointed stable curves. The second main theorem of the present paper is as follows (see Theorem \ref{main-them-2} for \tch{a more precise statement}):

\begin{theorem}\label{them-1-3}
We maintain the notation introduced above. Let $X^{\bp}$ be an arbitrary pointed stable curve over an algebraically closed field $k$ of characteristic $p>0$. Suppose that $m>>0$. Then there exists a Galois admissible covering $f^{\bp}$ of $X^{\bp}$ with Galois group $\mbZ/m\mbZ$ such that ${\rm deg}(D)=(n_{X}-1)m$ if $n_{X}\neq 0$, and that the ``first" generalized Hasse-Witt invariant (see \ref{ghw} for the definition) attains the maximum (see Definition \ref{def-3-2} for $\gamma_{X^{\bp}}^{\rm max}$)
\begin{eqnarray*}
\gamma^{\rm max}_{X^{\bp}}=\left\{ \begin{array}{ll}
g_{X}-1, & \text{if} \ n_{X} =0,
\\
g_{X}+n_{X}-2, & \text{if} \ n_{X} \neq 0.
\end{array} \right.
\end{eqnarray*} 
\end{theorem}

\begin{remarkA}
If $n_{X}=3$, we prove a strong version of Theorem \ref{them-1-3} (see Theorem \ref{main-them-2-3}) which is the key tool for reconstructing field structures associated to marked points from $\Pi_{X^{\bp}}$ (see \ref{1.5} below).
\end{remarkA}

\noindent
If $n_{X}=0$ and $X^{\bp}$ is smooth over $k$, then Theorem \ref{them-1-3} is equivalent to \cite[Th\'eor\`eme 4.3.1]{R1}. Thus, Theorem \ref{them-1-3} generalizes Raynaud's result to the case of (possibly ramified) admissible coverings of arbitrary (possibly singular) pointed stable curves. On the other hand, Theorem \ref{them-1-3} can be regarded as an analogue of Tamagawa's result in the case of arbitrary (possibly singular) pointed stable curves when $\text{deg}(D)$ attains maximum, namely equal to $(n_{X}-1)m$.


\subsection{An anabelian formula for topological types}
As we mentioned above, the generalized Hasse-Witt invariants can help us to analyze the structures of admissible fundamental groups in positive characteristic, and moreover, to understand the anabelian phenomena of curves in positive characteristic.

\subsubsection{}
By applying the result explained in \ref{1.3.1}, Tamagawa obtained a group-theoretical formula for the type $(g_{X}, n_{X})$  by using the tame fundamental group $\Pi_{X^{\bp}}$ when $X^{\bp}$ is {\it smooth} over $k$ (see \ref{rem-5-4-1}). In particular, $g_{X}$ and $n_{X}$ are group-theoretical invariants associated to $\Pi_{X^{\bp}}$ (\cite[Theorem 0.1]{T2}). This result is the main goal of \cite{T2}, which plays a key role in the theory of tame anabelian geometry of curves over algebraically closed fields of characteristic $p>0$ (e.g. \cite{T2}, \cite{Y2}).

On the other hand, we mention that Tamagawa's method for finding a group-theoretical formula for types {\it is difficult to be generalized} to the case of arbitrary pointed stable curves (see \ref{singav}). If $X^{\bp}$ is an {\it arbitrary} pointed stable curve over $k$, we ask the following question: 
\begin{quote}
{\it Does there exist a group-theoretical formula for $(g_{X}, n_{X})$ when $X^{\bp}$ is an arbitrary pointed stable curve over $k$?} 
\end{quote}
This problem can be completely solved by applying Theorem \ref{them-1-3}.



\subsubsection{}

Let $\overline \mbF_{p}$ be an arbitrary algebraic closure of $\mbF_{p}$, $\Pi$ an abstract profinite group such that $\Pi\cong \Pi_{X^{\bp}}$ as profinite groups, $\chi \in \text{Hom}(\Pi, \overline \mbF_{p}^{\times})$ such that $\chi \neq 1$, and $\Pi_{\chi} \subseteq \Pi$ the kernel of $\chi$. The profinite group $\Pi_{\chi}$ admits a natural action of $\Pi$ via the conjugation action. We put
$$N_{\chi}\defeq\{\pi \in H_{\chi, p_{\Pi}} \defeq \text{Hom}(\Pi_{\chi}, \mbZ/p\mbZ)\otimes_{\mbF_{p}} \overline \mbF_{p} \ | \ \tau\cdot\pi=\chi(\tau)\pi \ \text{for all} \ \tau \in \Pi\},$$ $$\gamma_{N_\chi} \defeq \text{dim}_{\overline \mbF_{p}}(N_{\chi}),$$ where $(\tau\cdot\pi)(x)\defeq \pi(\tau^{-1}\cdot x)$ for all $x \in \Pi_{\chi}$. Moreover, we put $$\gamma_{\Pi}^{\rm max} \defeq \text{max}\{ \gamma_{N_\chi} \ | \ \chi \in \text{Hom}(\Pi, \overline \mbF_{p}^{\times}) \ \text{such that} \ \chi \neq 1\}.$$ Since the prime number $p$ is a group-theoretical invariant associated to $\Pi$ (Lemma \ref{lem-5-2} (ii)), we see that $\gamma_{\Pi}^{\rm max}$ is also a group-theoretical invariant associated to $\Pi$. Moreover, we have that  $\gamma_{\Pi}^{\rm max}=\gamma_{X^{\bp}}^{\rm max}$ (Lemma \ref{lem-5-3}). Then we obtain the following formula (see also Theorem \ref{formula}):

\begin{theorem}\label{them-1-2}
Let $X^{\bp}$ be an arbitrary pointed stable curve of type $(g_{X}, n_{X})$ over an algebraically closed field $k$ of characteristic $p>0$ and $\Pi$ an abstract profinite group such that $\Pi\cong \Pi_{X^{\bp}}$ as profinite groups. Then we have (see \ref{5.1.2} for the definitions of group-theoretical invariants $b_{\Pi}^{1}$ and $b_{\Pi}^{2}$  associated to $\Pi$) $$g_{X}=b^{1}_{\Pi}-\gamma_{\Pi}^{\rm max}-1, \ n_{X}=2\gamma^{\rm max}_{\Pi}-b_{\Pi}^{1}-b_{\Pi}^{2}+3.$$ In particular, $g_{X}$ and $n_{X}$ are group-theoretical invariants associated to $\Pi$.
\end{theorem}

\begin{remarkA}
If $W^{\bp}$ is a pointed stable curve of type $(g_{W}, n_{W})$ over an {\it arithmetic} field (e.g. a number field, a $p$-adic field, a finite field), then a group-theoretical formula for $(g_{W}, n_{W})$ can be deduced immediately by computing ``weight" (e.g. by applying the weight monodromy conjecture or $p$-adic Hodge theory).
\end{remarkA}

\subsection{Some further applications to anabelian geometry}\label{1.5}
Next, let us explain some further applications of Theorem \ref{them-1-3} (and Theorem \ref{main-them-2-3}) that motivated the theory developed in the present paper. 

\subsubsection{}\label{1.5.1}
Theorem \ref{them-1-3} and Theorem \ref{main-them-2-3} are the main ingredients in the proof of the following important anabelian result:

\begin{quote}
Let $X_{i}^{\bp}$, $i \in \{1, 2\}$, be a pointed stable curve of type $(g_{X_{i}}, n_{X_{i}})$ over an algebraically closed field $k_{i}$ of characteristic $p>0$, $\Pi_{X^{\bp}_{i}}$ the admissible fundamental group of $X^{\bp}_{i}$, and $I_{i}\subseteq \Pi_{X^{\bp}_{i}}$ an inertia subgroup associated to a marked point of $X^{\bp}_{i}$. Suppose that $(g, n)\defeq (g_{X_{1}}, n_{X_{1}})=(g_{X_{2}}, n_{X_{2}})$. Let $$\phi: \Pi_{X^{\bp}_{1}} \migisurj \Pi_{X_{2}^{\bp}}$$ be an open continuous homomorphism. Then the following statements hold:

(i) $\phi(I_{1}) \subseteq \Pi_{X^{\bp}_{2}}$ is an inertia subgroup associated to a marked point of $X^{\bp}_{2}$, and there exists an inertia subgroup $I' \subseteq \Pi_{X^{\bp}_{1}}$ associated to a marked point of $X^{\bp}_{1}$ such that $\phi(I')=I_{2}$ (\cite[Theorem 4.11]{Y5}). 

\tch{(ii) The field structures associated to inertia subgroups of marked points can be reconstructed group-theoretically from $\Pi_{X_{i}^{\bp}}$, and that $\phi$ induces a field isomorphism between the fields associated to $I_{1}$ and $\phi(I_{1})$ group-theoretically (\cite[Theorem 4.13]{Y5}, Theorem \ref{fieldstr} and Remark \ref{rem-fstr} of the present paper).}
\end{quote}
\tch{This result generalizes \cite[Theorem 5.2 and Proposition 5.3]{T2} and \cite[Theorem 5.6 and Proposition 6.1]{Y2} to the case of arbitrary pointed stable curves. \cite[Theorem 5.2 and Proposition 5.3]{T2} and \cite[Theorem 5.6 and Proposition 6.1]{Y2} play key roles in the proofs of the weak Isom-version of the Grothendieck conjecture of curves over algebraically closed fields of characteristic $p>0$ (\cite[Theorem 0.2]{T2}) and the weak Hom-version of the Grothendieck conjecture of curves over algebraically closed fields of characteristic $p>0$ (\cite[Theorem 1.2]{Y2}), respectively.} 

\subsubsection{}
The result mentioned in \ref{1.5.1} is a critical step towards proving the main theorems of \cite{Y5}, \cite{Y6}. We explain briefly as follows. Let $\overline \mcM_{g, n, \mbZ}$ be the moduli stack of type $(g, n)$ over $\spec \mbZ$ and $\overline M_{g, n}$ the coarse moduli space of $\overline \mcM_{g, n, \mbZ} \times_{\mbZ} \overline \mbF_{p}$. Moreover, we may define an equivalence relation $\sim_{fe}$ on $\overline M_{g, n}$ induced by Frobenius actions (roughly speaking, for any $q_{1}, q_{2} \in \overline M_{g, n}$, $q_{1}\sim_{fe}q_{2}$ if the curve corresponding to a geometric point over $q_{1}$ is a Frobenius twist of the curve corresponding to a geometric point over $q_{2}$).

In \cite{Y5}, the author introduced a topological space $\overline \Pi_{g, n}$ whose points are isomorphism classes (as profinite groups) of admissible fundamental groups of pointed stable curves of type $(g, n)$ over algebraically closed fields of characteristic $p>0$, which we call {\it the moduli spaces of admissible fundamental groups of type $(g, n)$}. Moreover, the author proved that there is a natural continuous surjective homomorphism $$ \pi_{g, n}^{\rm adm}: \overline M_{g, n}/\sim_{fe} \migisurj \overline \Pi_{g, n},$$ and posed the so-called {\it Homeomorphism Conjecture} which says that $\pi_{g, n}^{\rm adm}$ is a homeomorphism. The Homeomorphism Conjecture shows a new kind of anabelian phenomenon which {\it cannot be explained by using Grothendieck's original anabelian philosophy}, and which means that the moduli spaces of pointed stable curves in positive characteristic {\it can be reconstructed group-theoretically as  topological spaces} from admissible fundamental groups of curves (see \cite[Introduction]{Y5}).  

The main theorems of \cite{Y5}, \cite{Y6} say that the Homeomorphism Conjecture holds when $\text{dim}(\overline M_{g, n})\leq 1$ (i.e., $(g, n)=(0, 3)$, $(0, 4)$, $(1,1)$).

\subsection{Structure of the present paper}
The present paper is organized as follows. 

In Section \ref{sec-2}, we recall some definitions and properties of admissible coverings, admissible fundamental groups, generalized Hasse-Witt invariants, and Raynaud-Tamagawa theta divisors. 

In Section \ref{sec-mghw}, we study the relation of generalized Hasse-Witt invariants  between a pointed stable curve and the smooth pointed stable curves associated to its irreducible components. Moreover, we introduce maximum generalized Hasse-Witt invariants and prove our main results in the case of $n_{X}\leq 1$.

In Section \ref{sec-generic}, we study maximum generalized Hasse-Witt invariants when $X^{\bp}$ is a component-generic pointed stable curve and prove the first main result.

In Section \ref{sec-arb}, we study maximum generalized Hasse-Witt invariants when $X^{\bp}$ is an arbitrary pointed stable curve and prove the second main result. 

In Section \ref{sec-appag}, by applying the results obtained in previous sections, we prove an anabelian formula for types of arbitrary pointed stable curves, and prove that the field structures associated to inertia subgroups of marked points can be group-theoretically reconstructed from admissible fundamental groups.

\subsection{Acknowledgements} The author would like to thank the referees very much for carefully reading the manuscript and for giving me many comments
which substantially helped improving the quality of the paper.
This work was supported by the Research Institute for Mathematical
Sciences (RIMS), an International Joint Usage/Research Center located in Kyoto
University.

\section{Preliminaries}\label{sec-2}

\subsection{Admissible coverings and admissible fundamental groups}\label{sec-2-1}

In this subsection, we recall some definitions and results concerning admissible coverings and admissible fundamental groups. 

\subsubsection{}\label{graph} We recall some notation concerning semi-graphs (\cite[Section 1]{M3}). Let $$\mathbf{G}\defeq(v(\mathbf{G}), e(\mathbf{G}), \{\zeta^{\mathbf{G}}_{e}\}_{e\in e(\mathbf{G})})$$ be a semi-graph. Here, $v(\mathbf{G})$, $e(\mathbf{G})$, and $\{\zeta^{\mathbf{G}}_{e}\}_{e\in e(\mathbf{G})}$ denote the set of vertices of $\mathbf{G}$, the set of edges of $\mathbf{G}$, and the set of coincidence maps of $\mathbf{G}$, respectively. 

Let $e \in e(\mathbf{G})$ be an edge. Then $e\defeq \{b_{e}^{1}, b_{e}^{2}\}$ is a set of cardinality $2$ for each $e \in e(\mathbf{G})$. The set $e(\mathbf{G})$ consists of closed edges and open edges. If $e$ is a closed edge, then the coincidence map $\zeta_{e}^{\mathbf{G}}$ is a map from $e$ to the set of vertices to which $e$ abuts. If $e$ is an open edge, then the coincidence map $\zeta_{e}^{\mathbf{G}}$ is a map from $e$ to the set which consists of the unique vertex to which $e$ abuts and the one element set $\{v(\mathbf{G})\}$ (i.e., either $\zeta_{e}^{\mathbf{G}}(b_{e}^{1})$ or $\zeta_{e}^{\mathbf{G}}(b^{2}_{e})$ is not contained in $v(\mathbf{G})$).

We shall write $e^{\rm op}(\mathbf{G}) \subseteq e(\mathbf{G})$ for the set of open edges of $\mathbf{G}$ and $e^{\rm cl}(\mathbf{G}) \subseteq e(\mathbf{G})$ for the set of closed edges of $\mathbf{G}$. Note that we have $$e(\mathbf{G})=e^{\rm op}(\mathbf{G})\cup e^{\rm cl}(\mathbf{G}).$$ Moreover, we denote by $e^{\rm lp}(\mathbf{G}) \subseteq e^{\rm cl}(\mathbf{G})$ the subset of closed edges such that $\#(\zeta_{e}^{\mathbf{G}}(e))=1$ for each $e \in e^{\rm lp}(\mathbf{G})$ (i.e., a closed edge which abuts to a unique vertex of $\mathbf{G}$), where ``lp" means ``loop". For each $e \in e(\mathbf{G})$, we denote by $v^{\mathbf{G}}(e) \subseteq v(\mathbf{G})$ the set of vertices of $\mathbf{G}$ to which $e$ abuts. For each $v \in v(\mathbf{G})$, we denote by $e^{\mathbf{G}}(v) \subseteq e(\mathbf{G})$ the set of edges of $\mathbf{G}$ to which $v$ is abutted.

We shall say that $\mathbf{G}$ is a {\it tree} if the Betti number $\text{dim}_{\mbQ}(H^{1}(\mathbf{G}, \mbQ))$ of $\mathbf{G}$ is equal to $0$, where $\mbQ$ denotes the rational number field. 

\begin{example}
Let us give an example of semi-graph to explain the above definitions. We use the notation ``$\bp$" and ``$\circ$ with a line segment" to denote a vertex and an open edge, respectively. 

Let $\mathbf{G}$ be a semi-graph as follows:

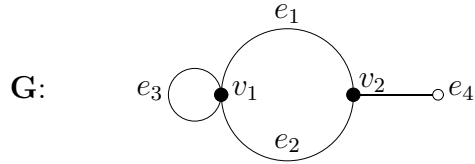
\begin{picture}(300,100)
\put(170,60){\circle*{5} $v_{1}$}
\put(195,60){\circle{50}}
\put(190,89){$e_{1}$}
\put(190,40){$e_{2}$}
\put(160,60){\circle{20}}
\put(138,60){$e_{3}$}
\put(220,60){\circle*{5}}
\put(222,63){$v_{2}$}
\put(220,60){\line(5,0){30}}
\put(252,60){\circle{4} $e_{4}$}
\put(90,59){$\mathbf{G}$:}
\end{picture}

\noindent
Then we see that $v(\mathbf{G})=\{v_{1}, v_{2}\}$,  $e^{\rm cl}(\mathbf{G})=\{e_{1}, e_{2}, e_{3}\}$, $e^{\rm op}(\mathbf{G})=\{e_{4}\}$, $\zeta_{e_{1}}^{\mathbf{G}}(e_{1})=\{v_{1}, v_{2}\}$, $\zeta_{e_{2}}^{\mathbf{G}}(e_{2})=\{v_{1}, v_{2}\}$, $\zeta_{e_{3}}^{\mathbf{G}}(e_{3})=\{v_{1}\}$, and $\zeta_{e_{4}}^{\mathbf{G}}(e_{4})=\{v_{2}, \{v(\mathbf{G})\}\}$. Moreover, we have $e^{\rm lp}(\mathbf{G})=\{e_{3}\}$, $v^{\mathbf{G}}(e_{1})=\{v_{1}, v_{2}\}$, $v^{\mathbf{G}}(e_{2})=\{v_{1}, v_{2}\}$, $v^{\mathbf{G}}(e_{3})=\{v_{1}\}$, $v^{\mathbf{G}}(e_{4})=\{v_{2}\}$, $e^{\mathbf{G}}(v_{1})=\{e_{1}, e_{2}, e_{3}\}$, and $e^{\mathbf{G}}(v_{2})=\{e_{1}, e_{2}, e_{4}\}$.
\end{example}

\subsubsection{}\label{defcurves}
Let $p$ be a prime number, and let $$X^{\bp}=(X, D_{X})$$ be a pointed stable curve over an algebraically closed field $k$ of characteristic $p$, where $X$ denotes the underlying curve, $D_{X}$ denotes a finite set of marked points satisfying \cite[Definition 1.1 (iv)]{K}. Write $g_{X}$ for the genus of $X$ and $n_{X}$ for the cardinality $\#(D_{X})$ of $D_{X}$. We call the pair $(g_{X}, n_{X})$ the {\it topological type} (or type for short) of $X^{\bp}$. 

Recall that the {\it dual semi-graph} $\Gamma_{X^{\bp}}\defeq(v(\Gamma_{X^{\bp}}), e(\Gamma_{X^{\bp}}), \{\zeta^{\Gamma_{X^{\bp}}}_{e}\}_{e\in e(\Gamma_{X^{\bp}})})$ of $X^{\bp}$ is a semi-graph associated to $X^{\bp}$ defined as follows: (i) $v(\Gamma_{X^{\bp}})$ is the set of irreducible components of $X$; (ii) $e^{\rm op}(\Gamma_{X^{\bp}})$ is the set of marked points $D_{X}$; (iii) $e^{\rm cl}(\Gamma_{X^{\bp}})$ is the set of nodes of $X$; (iv) $\zeta_{e}^{\Gamma_{X^{\bp}}}(e)$, $e \in e^{\rm op}(\Gamma_{X^{\bp}})$, consists of the unique irreducible component containing $e$ and the set $\{v(\Gamma_{X^{\bp}})\}$; (v) $\zeta_{e}^{\Gamma_{X^{\bp}}}(e)$, $e \in e^{\rm cl}(\Gamma_{X^{\bp}})$, consists of the irreducible components containing $e$. 

Moreover, we write $r_{X}\defeq\text{dim}_{\mbQ}(H^{1}(\Gamma_{X^{\bp}}, \mbQ))$ for the Betti number of the semi-graph $\Gamma_{X^{\bp}}$. 

\begin{example}
We give an example to explain dual semi-graphs of pointed stable curves. Let $X^{\bp}$ be a pointed stable curve over $k$ whose irreducible components are $X_{v_{1}}$ and $X_{v_{2}}$, whose node is $x_{e_{1}}$, and whose marked point is $x_{e_{2}} \in X_{v_{2}}$. We use the notation ``$\bp$" and ``$\circ$" to denote a node and a marked point, respectively. Then $X^{\bp}$ is as follows:

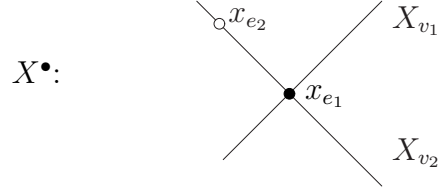
\begin{picture}(300,100)
\put(170,30){\line(5, 5){60} $X_{v_{2}}$}
\put(170,80){\line(5, -5){60} $X_{v_{1}}$}
\put(168.5,81.5){\circle{4}}
\put(172,82){$x_{e_{2}}$}
\put(160,90){\line(5,-5){7.2}}
\put(195, 55){\circle*{4}}
\put(201, 53){$x_{e_{1}}$}

\put(90,59){$X^{\bp}$:}
\end{picture}

We write $v_{1}$ and $v_{2}$ for the vertices of $\Gamma_{X^{\bp}}$ corresponding to $X_{v_{1}}$ and $X_{v_{2}}$, respectively, $e_{1}$ for the closed edge corresponding to $x_{e_{1}}$, and $e_{2}$ for the open edge corresponding to $x_{e_{2}}$. Moreover, we use the notation ``$\bp$" and ``$\circ$ with a line segment" to denote a vertex and an open edge, respectively. Then the dual semi-graph $\Gamma_{X^{\bp}}$ of $X^{\bp}$ is as follows:

\begin{picture}(300,100)
\put(155,60){$v_{1}$}
\put(170,60){\circle*{5}}
\put(170,60){\line(9,0){50}}
\put(190,65){$e_{1}$}
\put(220,60){\circle*{5}}
\put(222,65){$v_{2}$}
\put(220,60){\line(5,0){30}}
\put(252,60){\circle{4} $e_{2}$}
\put(90,59){$\Gamma_{X^{\bp}}$:}
\end{picture}
\end{example}

\subsubsection{}\label{smoothpointed}
Let $v \in v(\Gamma_{X^{\bp}})$ and $e \in e(\Gamma_{X^{\bp}})$. We write $X_{v}$ for the irreducible component of $X$ corresponding to $v$, $x_{e}$ for the node of $X$ corresponding to $e$ if $e \in e^{\rm cl}(\Gamma_{X^{\bp}})$, and $x_{e}$ for the marked point of $X$ corresponding to $e$ if $e \in e^{\rm op}(\Gamma_{X^{\bp}})$. \tch{Note that $X^{\bp}$ is allowed to have components with self-intersections in general (i.e., $e^{\rm lp}(\Gamma_{X^{\bp}}) \neq \emptyset$).} Moreover, write $\widetilde X_{v}$ for the {\it smooth} compactification of $U_{X_{v}}\defeq X_{v} \setminus X_{v}^{\rm sing}$, where $(-)^{\rm sing}$ denotes the singular locus of $(-)$. We call $$\widetilde X_{v}^{\bp}=(\widetilde X_{v}, D_{\widetilde X_{v}}\defeq(\widetilde X_{v} \setminus U_{X_{v}})\cup (D_{X}\cap X_{v}))$$ {\it the smooth pointed stable curve of type $(g_{v}, n_{v})$ over $k$ associated to $v$} (or the smooth pointed stable curve associated to $v$ for short). Note that $\widetilde X_{v}$ is the normalization of $X_{v}$.
\subsubsection{}\label{comgeneric}
Let $\overline \mcM_{g, n, \mbZ}$ be the moduli stack parameterizing pointed stable curves of type $(g, n)$ over $\spec \mbZ$, $\overline \mbF_{p}$ the algebraically closure of $\mbF_{p}$ in $k$, $\overline \mcM_{g, n} \defeq \overline\mcM_{g, n, \mbZ} \times_{\mbZ} \overline \mbF_{p}$, and $\overline M_{g, n}$ the coarse moduli space of $\overline \mcM_{g, n}$. Then $X^{\bp} \migi \spec k$ determines a morphism $c_{X}: \spec k \migi \overline \mcM_{g_{X}, n_{X}}$ and $\widetilde X_{v}^{\bp} \migi \spec k$, $v \in v(\Gamma_{X^{\bp}})$, determines a morphism $c_{v}: \spec k \migi \overline \mcM_{g_{v}, n_{v}}$. Moreover, we have a clutching morphism of moduli stacks (\cite[Definition 3.8]{K}) $$c: \prod_{v\in v(\Gamma_{X^{\bp}})} \overline \mcM_{g_{v}, n_{v}} \migi \overline \mcM_{g_{X}, n_{X}}$$ such that $c\circ (\prod_{v\in v(\Gamma_{X^{\bp}})} c_{v})=c_{X}$. We shall say that $X^{\bp}$ is a {\it component-generic} pointed stable curve over $k$ if the image of $$\prod_{v\in v(\Gamma_{X^{\bp}})} c_{v}: \spec k \migi \prod_{v\in v(\Gamma_{X^{\bp}})} \overline \mcM_{g_{v}, n_{v}}$$ is a generic point in $\prod_{v\in v(\Gamma_{X^{\bp}})} \overline M_{g_{v}, n_{v}}$.


\subsubsection{}\label{admcoverings}

We recall the definition of admissible coverings of pointed stable curves. Let $Y^{\bp}=(Y, D_{Y})$ be a pointed stable curve over $k$ and $\Gamma_{Y^{\bp}}$ the dual semi-graph of $Y^{\bp}$. Let $$f^{\bp}: Y^{\bp} \migi X^{\bp}$$ be a {\it surjective, generically \'etale, finite} morphism of pointed stable curves over $k$ such that $f(y)$ is a smooth (resp. singular) point of $X$ if $y$ is a smooth (resp. singular) point of $Y$. Write $f: Y \migi X$ for the morphism of underlying curves induced by $f^{\bp}$ and $f^{\rm sg}: \Gamma_{Y^{\bp}} \migi \Gamma_{X^{\bp}}$ for the map of dual semi-graphs induced by $f^{\bp}$. Let $v \in v(\Gamma_{X^{\bp}})$ and $w\in (f^{\rm sg})^{-1}(v) \subseteq v(\Gamma_{Y^{\bp}})$. Then $f^{\bp}$ induces a morphism of smooth pointed stable curves $$\widetilde f^{\bp}_{w, v}: \widetilde Y^{\bp}_{w}\migi \widetilde X_{v}^{\bp}$$ associated to $w$ and $v$ (\ref{smoothpointed}).

\begin{definition}\label{def-1-2}
We shall say that $f^{\bp}: Y^{\bp} \migi X^{\bp}$ is a {\it Galois admissible covering} over $k$ with Galois group $G$ if the following conditions are satisfied: (i) There exists a finite group $G\subseteq \text{Aut}_{k}(Y^{\bp})$ such that $Y^{\bp}/G=X^{\bp}$, and $f^{\bp}$ is equal to the quotient morphism $Y^{\bp} \migi Y^{\bp}/G$. (ii) $\widetilde f^{\bp}_{w, v}$ is a tame covering over $k$ for each $v\in v(\Gamma_{X^{\bp}})$ and each $w\in (f^{\rm sg})^{-1}(v)$. (iii) For each $y\in Y^{\rm sing}$, we write $D_{y}\subseteq G$ for the decomposition group of $y$ and $\tau$ a generator of $D_{y}$. Then the local morphism between singular points induced by $f$ is
\[
\begin{array}{ccc}
 \widehat \mcO_{X, f(y)} \cong k[[u,v]]/uv & \migi & \widehat \mcO_{Y,y}\cong k[[s,t]]/st
\\
u & \mapsto & s^{\#(D_{y})}
\\
v & \mapsto & t^{\#(D_{y})},
\end{array}
\]
and that $\tau(s)=\zeta_{\#(D_{y})}s$ and $\tau(t)=\zeta^{-1}_{\#(D_{y})}t$, where $\zeta_{\#(D_{y})}$ is a primitive $\#(D_{y})$th root of unity.

Moreover, we shall say that $f^{\bp}$ is an {\it admissible covering} if there exists a morphism of pointed stable curves $h^{\bp}: W^{\bp} \migi Y^{\bp}$ over $k$ such that the composite morphism $f^{\bp}\circ h^{\bp}: W^{\bp} \migi X^{\bp}$ is a Galois admissible covering over $k$. 

Let $Z^{\bp}$ be a disjoint union of finitely many pointed stable curves over $k$. We shall say that a morphism $f_{Z}^{\bp}: Z^{\bp} \migi X^{\bp}$ over $k$ is a {\it multi-admissible covering} if the restriction of $f_{Z}^{\bp}$ to each connected component of $Z^{\bp}$ is admissible. Moreover, we shall say that $f_{Z}^{\bp}$ is {\it \'etale} if the underlying morphism of curves $f_{Z}$ induced by $f_{Z}^{\bp}$ is an \'etale morphism.

\end{definition}

\subsubsection{}\label{admfg}
We recall some facts concerning admissible fundamental groups. Let $\mcM_{g_{X}, n_{X}}$ be the open substack of the moduli stack $\overline \mcM_{g_{X}, n_{X}}$ parameterizing smooth pointed stable curves. Write $\overline \mcM_{g_{X}, n_{X}} ^{\log}$ for the log stack obtained by equipping $\overline \mcM_{g_{X}, n_{X}}$ with the natural log structure associated to the divisor with normal crossings $\overline \mcM_{g_{X}, n_{X}}\setminus\mcM_{g_{X}, n_{X}} \subset \overline \mcM_{g_{X}, n_{X}}$. Then we obtain a morphism $s\defeq \spec k \migi \overline \mcM_{g_{X}, n_{X}}$ determined by $X^{\bp} \migi s$. Write $s_{X}^{\log}$ for the log scheme whose underlying scheme is $\spec k$, and whose log structure is the pulling back log structure induced by the morphism $s \migi \overline \mcM_{g_{X}, n_{X}}$. We obtain a natural morphism $s_{X}^{\log} \migi \overline \mcM_{g_{X}, n_{X}}^{\log}$ induced by the morphism $s \migi \overline \mcM_{g_{X}, n_{X}}$ and a stable log curve $$X^{\log}\defeq s_{X}^{\log}\times_{\overline \mcM_{g_{X}, n_{X}}^{\log}}\overline \mcM_{g_{X}, n_{X}+1}^{\log}$$ over $s_{X}^{\log}$ whose underlying scheme is $X$.

Let $\widetilde x^{\log} \migi X^{\log}$ be a log geometric point over a smooth point of $X$. Write $x \migi X$ for the geometric point induced by the log geometric point. Then we have a surjective homomorphism of log \'etale fundamental groups $\pi_{1}(X^{\log}, \widetilde x^{\log}) \migisurj \pi_{1}(s_{X}^{\log}, \widetilde x^{\log})$. We call $$\pi_{1}^{\rm adm}(X^{\bp}, x)\defeq \text{ker}(\pi_{1}(X^{\log}, \widetilde x^{\log}) \migisurj \pi_{1}(s_{X}^{\log}, \widetilde x^{\log}))$$ the {\it admissible fundamental group} of $X^{\bp}$ (i.e., the geometric log \'etale fundamental group of $X^{\log}$). It is well known that $\pi_{1}^{\rm adm}(X^{\bp}, x)$ independents the log structures of $X^{\log}$, and that there is a bijection between the set of open (resp. open normal) subgroups of $\pi_{1}^{\rm adm}(X^{\bp}, x)$ and the set of admissible (resp. Galois admissible) coverings of $X^{\bp}$ over $k$.

\subsubsection{}\label{stradm}
Since we only focus on the isomorphism class of $\pi_{1}^{\rm adm}(X^{\bp}, x)$ in the present paper, for simplicity of notation, we omit the base point and denote by $$\Pi_{X^{\bp}}$$ the admissible fundamental group $\pi_{1}^{\rm adm}(X^{\bp}, x)$. Note that, by the definition of admissible coverings, the admissible fundamental group of $X^{\bp}$ is naturally isomorphic to the tame fundamental group of $X^{\bp}$ when $X^{\bp}$ is smooth over $k$. \tch{The structure of the maximal prime-to-$p$ quotient of $\Pi_{X^{\bp}}$ is well-known, and is isomorphic to the prime-to-$p$ completion of the following group (\cite[Th\'eor\`eme 2.2 (c)]{V})  $$\langle a_{1}, \dots, a_{g_{X}}, b_{1}, \dots, b_{g_{X}}, c_{1}, \dots, c_{n_{X}} \ | \ \prod_{i=1}^{g_{X}}[a_{i}, b_{i}]\prod_{j=1}^{n_{X}}c_{j}=1\rangle.$$}

We denote by $\Pi^{\text{\'et}}_{X^{\bullet}}$ and $\Pi^{\rm top}_{X^{\bp}}$ the \'etale fundamental group of the underlying curve $X$ of $X^{\bp}$ and the profinite completion of the topological fundamental group of $\Gamma_{X^{\bp}}$, respectively. We have the following natural continuous surjective homomorphisms (for suitable choices of base points) $$\Pi_{X^{\bp}}\migisurj \Pi^{\text{\'et}}_{X^{\bullet}} \migisurj \Pi_{X^{\bp}}^{\rm top}.$$ Moreover, for each $v \in v(\Gamma_{X^{\bp}})$, we denote by $$\Pi_{\widetilde X_{v}^{\bp}}$$ the admissible fundamental group of $\widetilde X^{\bp}_{v}$ (i.e., the tame fundamental group of the smooth pointed stable curve associated to $v$). Then we have a natural (outer) injective homomorphism $\Pi_{\widetilde X_{v}^{\bp}} \migiinje \Pi_{X^{\bp}}$.

\subsection{Generalized Hasse-Witt invariants}\label{sec-2-2}

In this subsection, we recall some notation concerning generalized Hasse-Witt invariants of cyclic admissible coverings of arbitrary pointed stable curves. On the other hand, in the case of {\it smooth} pointed stable curves, the generalized Hasse-Witt invariants of cyclic tame coverings have been studied by I. Bouw (\cite[Section 2]{B}) and Tamagawa (\cite[Section 3]{T2}).

\subsubsection{\bf Settings}\label{settings2.2.1}
We maintain the notation introduced in \ref{defcurves}. Moreover, let $\Pi_{X^{\bp}}$ be the admissible fundamental group of $X^{\bp}$.

\subsubsection{}\label{def-1-3} 
We define the {\it $p$-rank} (or {\it Hasse-Witt invariant}) of $X^{\bp}$ to be $$\sigma_{X}\defeq\text{dim}_{\mbF_{p}}(H^{1}_{\text{\'et}}(X, \mbF_{p}))=\text{dim}_{\mbF_{p}}(\Pi_{X^{\bp}}^{\rm ab}\otimes \mbF_{p}),$$ where $\Pi_{X^{\bp}}^{\rm ab}$ denotes the abelianization of $\Pi_{X^{\bp}}$. We shall say that $X^{\bp}$ is {\it ordinary} if $g_{X}=\sigma_{X}$.  Moreover, we have the following:
$$\sigma_{X}=\sum_{v\in v(\Gamma_{X^{\bp}})} \sigma_{\widetilde X_{v}}+ r_{X},$$ where $r_{X}$ denotes the Betti number of $\Gamma_{X^{\bp}}$ (\ref{defcurves}).

\subsubsection{}\label{2.2.2}

Let $n$ be an arbitrary positive natural number prime to $p$ and $\mu_{n} \subseteq k^{\times}$ the group of $n$th roots of unity. Fix a primitive $n$th root $\zeta$, we may identify $\mu_{n}$ with $\mbZ/n\mbZ$ via the homomorphism $\zeta^{i} \mapsto i$. Let $\alpha \in \text{Hom}(\Pi_{X^{\bp}}^{\rm ab}, \mbZ/n\mbZ)$. We denote by $X^{\bp}_{\alpha}=(X_{\alpha}, D_{X_\alpha}) \migi X^{\bp}$ the Galois multi-admissible covering with Galois group $\mbZ/n\mbZ$ corresponding to $\alpha$.  Write $F_{X_{\alpha}}$ for the absolute Frobenius morphism on $X_{\alpha}$. Then there exists a decomposition (\cite[Section 9]{S}) $$H^{1}(X_{\alpha}, \mcO_{X_\alpha})=H^{1}(X_{\alpha}, \mcO_{X_\alpha})^{\rm st} \oplus H^{1}(X_{\alpha}, \mcO_{X_\alpha})^{\rm ni},$$ where $F_{X_{\alpha}}$ is a bijection on $H^{1}(X_{\alpha}, \mcO_{X_\alpha})^{\rm st}$ and is nilpotent on $H^{1}(X_{\alpha}, \mcO_{X_\alpha})^{\rm ni}$. Moreover, we have $H^{1}(X_{\alpha}, \mcO_{X_\alpha})^{\rm st}=H^{1}(X_{\alpha}, \mcO_{X_\alpha})^{F_{X_{\alpha}}}\otimes_{\mbF_{p}}k,$ where $H^{1}(X_{\alpha}, \mcO_{X_\alpha})^{F_{X_\alpha}}$ denotes the subspace of $H^{1}(X_{\alpha}, \mcO_{X_\alpha})$ on which  $F_{X_{\alpha}}$ acts trivially. Then Artin-Schreier theory implies that we may identify $$H_{\alpha}\defeq H^{1}_{\text{\'et}}(X_{\alpha}, \mbF_{p}) \otimes_{\mbF_{p}}k$$ with the largest subspace of $H^{1}(X_{\alpha}, \mcO_{X_\alpha})$ on which $F_{X_{\alpha}}$ is a bijection.

The finite dimensional $k$-linear space $H_{\alpha}$ is a finitely generated $k[\mu_{n}]$-module induced by the natural action of $\mu_{n}$ on $X_{\alpha}$. Then we have the following canonical decomposition $$H_{\alpha}=\bigoplus_{i\in \mbZ/n\mbZ} H_{\alpha, i},$$ where $\zeta \in \mu_{n}$ acts on $H_{\alpha, i}$ as the $\zeta^{i}$-multiplication. 

\subsubsection{}\label{ghw}
We call $$\gamma_{\alpha, i}\defeq\text{dim}_{k}(H_{\alpha, i}), \ i \in \mbZ/n\mbZ,$$ a {\it generalized Hasse-Witt invariant} (see \cite{B}, \cite{N}, \cite{T2} for the case of \'etale or tame coverings of smooth pointed stable curves) of the cyclic multi-admissible covering $X^{\bp}_{\alpha} \migi X^{\bp}$. In particular, we call $$\gamma_{\alpha, 1}$$ the {\it first} generalized Hasse-Witt invariant of the cyclic multi-admissible covering $X_{\alpha}^{\bp} \migi X^{\bp}$. Note that the above decomposition implies that $$\text{dim}_{k}(H_{\alpha})=\sum_{i \in \mbZ/n\mbZ}\gamma_{\alpha, i}.$$ In particular, if $X_{\alpha}$ is connected, then $\text{dim}_{k}(H_{\alpha})=\sigma_{X_{\alpha}}$.

\subsubsection{}\label{2.2.4}
We write $\mbZ[D_{X}]$ for the group of divisors whose supports are contained in $D_{X}$. Note that $\mbZ[D_{X}]$ is a free $\mbZ$-module with basis $D_{X}$. We define the following $$c'_{n}: \mbZ/n\mbZ[D_{X}]\defeq \mbZ[D_{X}] \otimes \mbZ/n\mbZ \migi \mbZ/n\mbZ, \ D \text{ mod} \ n \mapsto \text{deg}(D) \text{ mod } n.$$ Write $(\mbZ/n\mbZ)^{\sim}$ for the set $\{0, 1, \dots, n-1\}$ and $(\mbZ/n\mbZ)^{\sim}[D_{X}]$ for the subset of $\mbZ[D_{X}]$ consisting of the elements whose coefficients are contained in $(\mbZ/n\mbZ)^{\sim}$. Then we have a natural bijection $\iota_{n}: (\mbZ/n\mbZ)^{\sim}[D_{X}] \isom \mbZ/n\mbZ[D_{X}].$

We put $$(\mbZ/n\mbZ)^{\sim}[D_{X}]^{0}\defeq \iota^{-1}_{n}(\text{ker}(c_{n}')).$$ Note that we have $n  |  \text{deg}(D)$ for all $D \in (\mbZ/n\mbZ)^{\sim}[D_{X}]^{0}$. Moreover, we put $$s(D) \defeq \frac{\text{deg}(D)}{n} \in \mbZ_{\geq 0}.$$ Since every $D \in (\mbZ/n\mbZ)^{\sim}[D_{X}]^{0}$ can be regarded as a ramification divisor associated to some cyclic admissible covering, the structure of the maximal prime-to-$p$ quotient of $\Pi_{X^{\bp}}$ (\ref{stradm}) implies the following: 
\begin{eqnarray*}
0 \leq s(D)\leq \left\{ \begin{array}{ll}
0, & \text{if} \ n_{X}\leq 1,
\\
n_{X}-1, & \text{if} \ n_{X} \geq 2.
\end{array} \right.
\end{eqnarray*}



\subsubsection{}\label{universal}
We put $$\widehat X \defeq \invlim_{H \subseteq \Pi_{X^{\bp}}\ \text{open}} X_{H}, \ D_{\widehat X}\defeq \invlim_{H \subseteq \Pi_{X^{\bp}} \ \text{open}} D_{X_{H}}, \ \Gamma_{\widehat X^{\bp}} \defeq \invlim_{H \subseteq \Pi_{X^{\bp}}\ \text{open}} \Gamma_{X_{H}^{\bp}}.$$ We call $$\widehat X^{\bp} =(\widehat X, D_{\widehat X})$$ the universal admissible covering of $X^{\bp}$ corresponding to $\Pi_{X^{\bp}}$, and $\Gamma_{\widehat X^{\bp}}$ the dual semi-graph of $\widehat X^{\bp}$. Note that $\text{Aut}(\widehat X^{\bp} /X^{\bp})=\Pi_{X^{\bp}}$, and that $\Gamma_{\widehat X^{\bp}}$ admits a natural action of $\Pi_{X^{\bp}}$.

Let $\widehat X^{\bp}=(\widehat X, D_{\widehat X}) \migi X^{\bp}$ be a universal admissible covering corresponding to $\Pi_{X^{\bp}}$. For every $e \in e^{\rm op}(\Gamma_{X^{\bp}})$, write $\widehat e \in e^{\rm op}(\Gamma_{\widehat X^{\bp}})$ for an open edge over $e$ and  $x_{e}$ for the marked point corresponding to $e$.

We denote by $I_{\widehat e} \subseteq \Pi_{X^{\bp}}$ the stabilizer of $\widehat e$. The definition of admissible coverings (\ref{admcoverings}) implies that $I_{\widehat e}\cong \widehat \mbZ(1)^{p'}$ is isomorphic to the Galois group $\text{Gal}(K^{\rm t}_{x_{e}}/K_{x_{e}})$, where $K_{x_{e}}$ denotes the quotient field of $\mcO_{X, x_{e}}$,  $K^{\rm t}_{x_{e}}$ denotes a maximal tamely ramified extension, and $\widehat \mbZ(1)^{p'}$ denotes the maximal prime-to-$p$ quotient of $\widehat \mbZ(1)$. Suppose that $x_e$ is contained in $X_{v}$. Then we have an injection $$\phi_{\widehat e}: I_{\widehat e} \migiinje \Pi^{\rm ab}_{X^{\bp}}$$ which factors through $I_{\widehat e} \migiinje \Pi_{\widetilde X^{\bp}_v}^{\rm ab}$ induced by the composition of (outer) injective homomorphisms $I_{\widehat e} \migiinje \Pi_{\widetilde X^{\bp}_v} \migiinje \Pi_{X^{\bp}}$, where $\Pi_{\widetilde X^{\bp}_v}$ denotes the admissible fundamental group of the smooth pointed stable curve $\widetilde X^{\bp}_{v}$ associated to $v$ (\ref{smoothpointed}). Since the image of $\phi_{\widehat e}$ depends only on $e$, we may write $I_{e}$ for the image $\phi_{\widehat e}(I_{\widehat e})$. Moreover, the structure of maximal prime-to-$p$ quotients of admissible fundamental groups of pointed stable curves (\ref{stradm}) implies that the following holds: There exists a generator $[s_{e}]$ of $I_{e}$ for each $e \in e^{\rm op}(\Gamma_{X^{\bp}})$ such that $$\sum_{e\in e^{\rm op}(\Gamma_{X^{\bp}})}[s_{e}]=0$$ in $\Pi_{X^{\bp}}^{\rm ab}$. In the remainder of the present paper, we {\it fix} a set of generators $\{[s_{e}]\}_{e\in e^{\rm op}(\Gamma_{X^{\bp}})}$ of $I_{e}$ satisfying the above condition.

\begin{definition}\label{def-2-4}


We maintain the notation introduced above. 

(i) \tch{We put $$D_{\alpha}\defeq \sum_{e\in e^{\rm op}(\Gamma_{X^{\bp}})}\alpha([s_{e}])^{\sim}x_{e}, \ \alpha \in \text{Hom}(\Pi^{\rm ab}_{X^{\bp}}, \mbZ/n\mbZ),$$ where $\alpha([s_{e}])^{\sim}$ denotes the element of $(\mbZ/n\mbZ)^{\sim}$ corresponding to $\alpha([s_{e}])$ via the natural bijection $(\mbZ/n\mbZ)^{\sim} \isom \mbZ/n\mbZ$.} Note that we have $D_{\alpha} \in (\mbZ/n\mbZ)^{\sim}[D_{X}]^{0}$. On the other hand, for each $D \in (\mbZ/n\mbZ)^{\sim}[D_{X}]^{0}$, we put $$\text{Rev}_{D}^{\rm adm}(X^{\bp}) \defeq \{ \alpha \in \text{Hom}(\Pi_{X^{\bp}}^{\rm ab}, \mbZ/n\mbZ) \ | \ D_{\alpha}=D\}.$$  Moreover, we put $$\gamma_{(\alpha, D)}\defeq \gamma_{\alpha, 1} \ (\ref{ghw}).$$

(ii) Let $t \in \mbN$ be an arbitrary positive natural number, $n\defeq p^{t}-1$, and $$u=\sum_{j=0}^{t-1}u_{j}p^{j}, \ u \in \{0, \dots, n\},$$ the $p$-adic expansion with $u_{j} \in \{0, \dots, p-1\}$. We identify $\{0, \dots, t-1\}$ with $\mbZ/t\mbZ$ naturally. Then $\{0, \dots, t-1\}$ admits an additional structure induced by the natural additional structure of $\mbZ/t\mbZ$. We put $$u^{(i)}\defeq\sum_{j=0}^{t-1}u_{i+j}p^{j}, \ i\in\{0, \dots, t-1\}.$$ Let $D \in (\mbZ/n\mbZ)^{\sim}[D_{X}]^{0}$. We put $$D^{(i)}\defeq\sum_{x \in D_{X}} \bigl(\text{ord}_{x}(D)\bigl)^{(i)}x, \ i\in \{0, \dots, t-1\},$$ which is an effective divisor on $X$. 

Roughly speaking, $D^{(i)}$ is the ramification divisor associated to a suitable Frobenius action of a Galois admissible covering whose ramification divisor is $D$ (see the ``in particular" part of Remark \ref{trans} below for the relationship between the generalized Hasse-Witt invariants associated to $D$ and $D^{(i)}$).
\end{definition}

\begin{remarkA}\label{trans}
Let $l \in \{1, \dots, n-1\}$. We put $$D(l) \defeq lD-\bigl(\sum_{x \in D_{X}} n\cdot\left[\frac{l\cdot\text{ord}_{x}(D)}{n}\right]x\bigl),$$ where $[(-)]$ denotes the largest integer $\leq (-)$. Then we see that the following holds (see \ref{ghw} for the definition of $\gamma_{\alpha, l}$)  $$\gamma_{\alpha, l}=\gamma_{l\alpha, 1} =\gamma_{(l\alpha, D(l))}.$$
In particular, if $n\defeq p^{t}-1$, we have $\gamma_{(\alpha, D)}=\gamma_{\alpha, 1}=\gamma_{\alpha, p^{t-i}}=\gamma_{p^{t-i}\alpha, 1}=\gamma_{(p^{t-i}\alpha, D(p^{t-i}))}=\gamma_{(p^{t-i}\alpha, D^{(i)})}, \ i\in \{0, \dots, t-1\}.$

\end{remarkA}

\subsection{Generalized Hasse-Witt invariants via line bundles}\label{sec-2-3}

The generalized Hasse-Witt invariants can be also described in terms of line bundles and divisors.

\subsubsection{\bf Settings}\label{2.3.1} We maintain the settings introduced in \ref{settings2.2.1}. Moreover, in the present subsection, we suppose that $X^{\bp}$ is {\it smooth} over $k$.

\subsubsection{}\label{line}
Let $n \in \mbN$ be an arbitrary natural number prime to $p$. We denote by $\text{Pic}(X)$ the Picard group of $X$. Consider the following complex of abelian groups: $$\mbZ[D_{X}] \overset{a_{n}}{\migi} \text{Pic}(X)\oplus \mbZ[D_{X}] \overset{b_{n}}{\migi} \text{Pic}(X),$$ where $a_{n}(D)=([\mcO_{X}(-D)], nD)$, \tch{$b_{n}(([\mcL], D))=[\mcL^{n}\otimes \mcO_{X}(D)]$}. We denote by $$\msP_{X^{\bp}, n}\defeq\text{ker}(b_{n})/\text{Im}(a_{n})$$ the homology group of the complex. Moreover, we have the following exact sequence $$0\migi \text{Pic}(X)[n] \overset{a'_{n}}\migi \msP_{X^{\bp}, n} \overset{b'_{n}}\migi \mbZ/n\mbZ[D_{X}] \overset{c'_{n}}\migi\mbZ/n\mbZ,$$ where $\text{Pic}(X)[n]$ denotes the $n$-torsion subgroup of $\text{Pic}(X)$, and $$a'_{n}([\mcL])=([\mcL], 0)\ \text{mod Im}(a_{n}), \ b'_{n}(([\mcL], D)) \ \text{mod Im}(a_{n}))=D \ \text{mod}\ n,$$ $$c_{n}'(D \ \text{mod}\ n)=\text{deg}(D) \ \text{mod} \ n.$$ 
We shall define $$\widetilde \msP_{X^{\bp}, n}$$ to be the inverse image of $(\mbZ/n\mbZ)^{\sim}[D_{X}]^{0}\subseteq (\mbZ/n\mbZ)^{\sim}[D_{X}] \subseteq \mbZ[D_{X}]$ under the projection $\text{ker}(b_{n}) \migi \mbZ[D_{X}]$. It is easy to see that $\msP_{X^{\bp}, n}$ and $\widetilde \msP_{X^{\bp}, n}$ are free $\mbZ/n\mbZ$-groups with rank $2g_{X}+n_{X}-1$ if $n_{X}\neq 0$ and with rank $2g_{X}$ if $n_{X}=0$, and that there is a natural isomorphism $ \widetilde \msP_{X^{\bp}, n} \isom  \msP_{X^{\bp}, n}$. 

On the other hand, let $\alpha \in \text{Hom}(\Pi_{X^{\bp}}^{\rm ab}, \mbZ/n\mbZ)$ and $f_{\alpha}^{\bp}: X^{\bp}_{\alpha} \migi X^{\bp}$ the Galois multi-admissible covering over $k$ with Galois group $\mbZ/n\mbZ$ corresponding to $\alpha$. Fix a primitive $n$th root $\zeta$, we may identify $\mu_{n}$ with $\mbZ/n\mbZ$ via the map $\zeta^{i} \mapsto i$. Then we see  $$f_{\alpha, *}\mcO_{X_{\alpha}} \cong \bigoplus_{i\in \mbZ/n\mbZ}\mcL_{\alpha, i},$$ where locally $\mcL_{\alpha,i}$ is the eigenspace of the natural action of $i$ with eigenvalue $\zeta^{i}$. Moreover, we have the following natural isomorphism (\cite[Proposition 3.5]{T2}): $$\text{Hom}(\Pi_{X^{\bp}}^{\rm ab}, \mbZ/n\mbZ) \isom \widetilde \msP_{X^{\bp}, n}, \ \alpha \mapsto ([\mcL_{\alpha, 1}], D_{\alpha}).$$ Then every element of $\widetilde \msP_{X^{\bp}, n}$ induces a Galois multi-admissible covering of $X^{\bp}$ over $k$ with Galois group $\mbZ/n\mbZ$.

\subsubsection{\bf Assumptions}\label{assum-2.3.3}
In the remainder of the present paper, {\it we may assume that $$n\defeq p^{t}-1$$ for some positive natural number $t\in \mbN$ unless indicated otherwise}.

\subsubsection{}\label{ghwline}
Let $([\mcL], D) \in \widetilde \msP_{X^{\bp}, n}$. We fix an isomorphism $\mcL^{\otimes n}\cong \mcO_{X}(-D)$. Note that $D$ is an effective divisor on $X$. We have the following composition of morphisms of line bundles $$\mcL \overset{p^{t}}\migi \mcL^{\otimes p^{t}}=\mcL^{\otimes n}\otimes\mcL \isom \mcO_{X}(-D)\otimes\mcL \migiinje \mcL.$$ This composite morphism induces a homomorphism $\phi_{([\mcL], D)}: H^{1}(X, \mcL) \migi H^{1}(X, \mcL).$ We denote by $$\gamma_{([\mcL], D)}\defeq\text{dim}_{k}\bigl(\bigcap_{r\geq 1}\text{Im}(\phi_{([\mcL], D)}^{r})\bigl).$$ Write $\alpha_{\mcL} \in \text{Hom}(\Pi_{X^{\bp}}^{\rm ab}, \mbZ/n\mbZ)$ for the element corresponding to $([\mcL], D)$ and $F_{X}$ for the absolute Frobenius morphism on $X$. Then we see that $\gamma_{\alpha_{\mcL}, 1}$ (\ref{ghw}) is equal to the dimension over $k$ of the largest subspace of $H^{1}(X, \mcL)$ on which $F_{X}^{t}\defeq F_{X} \circ \dots \circ F_{X}$ is a bijection. Moreover, we have $$\gamma_{\alpha_{\mcL}, 1}=\text{dim}_{k}(H^{1}(X, \mcL)^{F_{X}^{t}}\otimes_{\mbF_{p}} k),$$ where $H^{1}(X, \mcL)^{F_{X}^{t}}$ denotes the subspace of $H^{1}(X, \mcL)$ on which  $F^{t}_{X}$ acts trivially. It is easy to see $$H^{1}(X, \mcL)^{F^{t}_{X}}\otimes_{\mbF_{p}} k=\bigcap_{r\geq 1}\text{Im}(\phi_{([\mcL], D)}^{r}).$$ Then we obtain  $\gamma_{([\mcL], D)}=\gamma_{\alpha_{\mcL}, 1}.$ Moreover, since $D_{\alpha_{\mcL}}=D$, we have $$\gamma_{([\mcL], D)}=\gamma_{(\alpha_{\mcL}, D)} \ (\defeq\gamma_{\alpha_{\mcL}, 1}).$$ We have the following lemma.

\begin{lemma}\label{lem-2-4}
We maintain the notation introduced above. Suppose that $X^{\bp}$ is smooth over $k$. Then we have  
\begin{eqnarray*}
\gamma_{(\alpha_{\mcL}, D)} \leq {\rm dim}_{k}(H^{1}(X, \mcL)) =\left\{ \begin{array}{ll}
g_{X}, & \text{if} \ ([\mcL], D)=([\mcO_{X}], 0),
\\
g_{X}-1, & \text{if} \ s(D) =0, \ [\mcL]\neq [\mcO_{X}],
\\
g_{X}+s(D)-1, & \text{if} \ s(D) \geq 1,
\end{array} \right.
\end{eqnarray*}
where $s(D)$ is the integer defined in \ref{2.2.4}.

\end{lemma}

\begin{proof}
The first inequality follows from the definition of generalized Hasse-Witt invariants. On the other hand, the Riemann-Roch theorem implies that $$\text{dim}_{k}(H^{1}(X, \mcL))=g_{X}-1-\text{deg}(\mcL)+\text{dim}_{k}(H^{0}(X, \mcL))$$$$=g_{X}-1+\frac{1}{n}\text{deg}(D)+\text{dim}_{k}(H^{0}(X, \mcL))=g_{X}-1+s(D)+\text{dim}_{k}(H^{0}(X, \mcL)).$$ This completes the proof of the lemma.
\end{proof}

\subsection{Raynaud-Tamagawa theta divisors}

In this subsection, we recall the theory of theta divisors which was introduced by Raynaud in the case of \'etale coverings (\cite{R1}), and which was generalized by Tamagawa in the case of tame coverings (\cite{T2}). 

\subsubsection{\bf Settings} We maintain the settings introduced in \ref{2.3.1}.

\subsubsection{}\label{thetabundle}
Let $F_{k}$ be the absolute Frobenius morphism on $\spec k$, $F_{X/k}$ the relative Frobenius morphism $X \migi X_{1}\defeq X\times_{k, F_{k}}k$ over $k$, and $F_{k}^{t}\defeq F_{k} \circ \dots \circ F_{k}$. We put $X_{t}\defeq X\times_{k, F^{t}_{k}}k,$ and define a morphism $$F^{t}_{X/k}: X \migi X_{t}$$ over $k$ to be $F_{X/k}^{t}\defeq F_{X_{t-1}/k}\circ \dots \circ F_{X_{1}/k} \circ F_{X/k}.$ 


Let $([\mcL], D) \in \widetilde \msP_{X^{\bp}, n}$, and let $\mcL_{t}$ be the pulling back of $\mcL$ by the natural morphism $X_{t} \migi X$. Note that $\mcL$ and $\mcL_{t}$ are line bundles of degree $-s(D)$ (\ref{2.2.4}). We put $$\mcB_{D}^{t}\defeq(F_{X/k}^{t})_{*}\bigl(\mcO_{X}(D)\bigl)/\mcO_{X_{t}}, \ \mcE_{D}\defeq\mcB_{D}^{t} \otimes \mcL_{t}.$$ Write $\text{rk}(\mcE_{D})$ for the rank of $\mcE_{D}$. Then we obtain $$\chi(\mcE_{D})=\text{deg}(\text{det}(\mcE_{D}))-(g_{X}-1)\text{rk}(\mcE_{D}).$$ Moreover, we have $\chi(\mcE_{D})=0$ (\cite[Lemma 2.3 (ii)]{T2}).

\subsubsection{}\label{divisortheta}
Let $J_{X_{t}}$ be the Jacobian variety of $X_{t}$ and $\mcL_{X_t}$ a universal line bundle on $X_{t} \times J_{X_{t}}$. Let $\text{pr}_{X_{t}}: X_{t}\times J_{X_{t}} \migi X_{t}$ and $\text{pr}_{J_{X_{t}}}: X_{t} \times J_{X_{t}} \migi J_{X_{t}}$ be the natural projections. We denote by $\mcF$ the coherent $\mcO_{X_{t}}$-module $\text{pr}_{X_{t}}^{*}(\mcE_{D}) \otimes \mcL_{X_t}$, and by $$\chi_{\mcF}\defeq\text{dim}_{k}(H^{0}(X_{t} \times_{k} k(y), \mcF\otimes k(y)))-\text{dim}_{k}(H^{1}(X_{t} \times_{k} k(y), \mcF\otimes k(y)))$$ for each $y \in J_{X_{t}}$, where $k(y)$ denotes the residue field of $y$. Note that since $\text{pr}_{J_{X_{t}}}$ is flat, $\chi_{\mcF}$ is independent of $y \in J_{X_{t}}$. Write $(-\chi_{\mcF})^{+}$ for $\text{max}\{0, -\chi_{\mcF}\}$. We denote by $$\Theta_{\mcE_{D}}\subseteq J_{X_{t}}$$ the closed subscheme of $J_{X_{t}}$ defined by the $(-\chi_{\mcF})^{+}$th Fitting ideal $\text{Fitt}_{(-\chi_{\mcF})^{+}}\bigl(R^{1}(\text{pr}_{J_{X_{t}}})_{*}(\mcF)\bigl).$ The definition of $\Theta_{\mcE_{D}}$ is independent of the choice of $\mcL_{t}$. Moreover, we have $\text{codim}(\Theta_{\mcE_{D}}) \leq 1$.

\subsubsection{}
In \cite{R1}, Raynaud investigated the following property of the vector bundle $\mcE_{D}$ on $X$.

\begin{condition}\label{condtheta}
We shall say that $\mcE_{D}$ satisfies $(\star)$ if there exists a line bundle $\mcL_{t}'$ of degree $0$ on $X_{t}$ such that $$0={\rm min}\{{\rm dim}_{k}({H}^{0}(X_{t}, \mcE_{D}\otimes \mcL_{t}')), {\rm dim}_{k}({H}^{1}(X_{t}, \mcE_{D}\otimes \mcL_{t}'))\}.$$
\end{condition}
Moreover,  \cite[Proposition 2.2 (i) (ii)]{T2} implies  that $[\mcL'] \not\in \Theta_{\mcE_{D}}$ if and only if $\mcE_{D}$ satisfies $(\star)$ for $\mcL'$, where $[\mcL']$ denotes the point of $J_{X_{t}}$ corresponding to $\mcL'$. Namely, $\Theta_{\mcE_{D}}$ is a {\it divisor} of $J_{X_{t}}$ when $\mcE_{D}$ satisfies $(\star)$. Then we have the following definition:

\begin{definition}\label{deftheta}
We shall say that $\Theta_{\mcE_{D}} \subseteq J_{X_{t}}$ is the {\it Raynaud-Tamagawa theta divisor} associated to $\mcE_{D}$ if $\mcE_{D}$ satisfies $(\star)$. 
\end{definition}

\begin{remarkA}
Suppose that $\mcE_{D}$ satisfies $(\star)$ (i.e., Condition \ref{condtheta}). \cite[Proposition 1.8.1]{R1} implies that $\Theta_{\mcE_{D}}$ is algebraically equivalent to $\text{rk}(\mcE_{D})\Theta$, where $\Theta$ is the classical theta divisor (i.e., the image of $X^{g_{X}-1}_{t}$ in $J_{X_{t}}$). 
\end{remarkA}

\begin{remarkB}\label{new-rem-RTd}
The definition of $\mcE_{D}$ implies that the following natural exact sequence $$0\migi \mcL_{t} \migi (F_{X/k}^{t})_{*}\bigl(\mcO_{X}(D)\bigl)\otimes \mcL_{t} \migi \mcE_{D} \migi 0.$$ Let $[\mcI] \in \text{Pic}(X)[n]$. Write $\mcI_{t}$ for the pulling back of $\mcI$ by the natural morphism $X_{t} \migi X$. We obtain the following exact sequence $$\dots \migi H^{0}(X_{t}, \mcE_{D} \otimes \mcI_{t}) \migi H^{1}(X_{t}, \mcL_{t} \otimes \mcI_{t})\overset{\phi_{\mcL_{t}\otimes \mcI_{t}}} \migi H^{1}(X_{t}, (F_{X/k}^{t})_{*}\bigl(\mcO_{X}(D)\bigl)\otimes \mcL_{t} \otimes \mcI_{t})$$$$\migi H^{1}(X_{t}, \mcE_{D} \otimes \mcI_{t}) \migi \dots.$$ Note that we have $$H^{1}(X_{t}, \mcL_{t} \otimes \mcI_{t}) \cong H^{1}(X, \mcL \otimes \mcI),$$ $$H^{1}(X_{t}, (F_{X/k}^{t})_{*}\bigl(\mcO_{X}(D)\bigl)\otimes \mcL_{t} \otimes \mcI_{t}) \cong H^{1}(X, \mcO_{X}(D)\otimes (F_{X/k}^{t})^{*}(\mcL_{t} \otimes \mcI_{t}))$$$$\cong H^{1}(X, \mcO_{X}(D)\otimes (\mcL \otimes \mcI)^{\otimes p^{t}}) \cong H^{1}(X, \mcL \otimes \mcI).$$ Moreover, it is easy to see that  the homomorphism $H^{1}(X, \mcL \otimes \mcI) \migi H^{1}(X, \mcL \otimes \mcI)$ induced by $\phi_{\mcL_{t}\otimes \mcI_{t}}$ coincides with $\phi_{([\mcL\otimes \mcI], D)}$. Thus, we obtain that if $$\gamma_{([\mcL\otimes \mcI], D)}=\text{dim}_{k}(H^{1}(X, \mcL \otimes \mcI))$$ for some $[\mcI] \in \text{Pic}(X)[n]$, then the Raynaud-Tamagawa theta divisor $\Theta_{\mcE_{D}}$ associated to $\mcE_{D}$ exists (i.e., $[\mcI_{t}] \not\in \Theta_{\mcE_{D}}$). 
\end{remarkB}

\subsubsection{}
Let $N$ be an arbitrary non-negative integer. We put 
\begin{eqnarray*}
C(N)\defeq  \left\{ \begin{array}{ll}
0, & \text{if} \ N=0,
\\
3^{N-1}N!, & \text{if} \ N\neq 0.
\end{array} \right.
\end{eqnarray*}
Then we have the following proposition.

\begin{proposition}\label{prop-theta}
We maintain the notation introduced above. Suppose that the Raynaud-Tamagawa theta divisor associated to $\mcE_{D}$ exists, and that $$n=p^{t}-1 > C(g_{X})+1.$$ Then there exists a line bundle $\mcI$ of degree $0$ on $X$ such that $[\mcI] \neq [\mcO_{X}]$, that $[\mcI^{\otimes n}]=[\mcO_{X}]$, and that $\gamma_{(([\mcL\otimes \mcI], D))}={\rm dim}_{k}(H^{1}(X, \mcL \otimes \mcI))$ (i.e., $[\mcI_{t}] \not\in \Theta_{\mcE_{D}}$).
\end{proposition}

\begin{proof}
By applying similar arguments to the arguments given in the proof of \cite[Corollary 3.10]{T2}, the proposition follows from Remark \ref{new-rem-RTd}.
\end{proof}
Namely, if the Raynaud-Tamagawa theta divisor associated to $\mcE_{D}$ exists, then there exists a line bundle $\mcI$ of degree $0$ with order $n>>0$ such that the following hold:
\begin{itemize}
\item $[\mcI_{t}] \not\in \Theta_{\mcE_{D}}$

\item the first generalized Hasse-Witt invariant (\ref{ghw}) of the Galois multi-admissible covering with Galois group $\mbZ/n\mbZ$ corresponding to $([\mcL\otimes \mcI], D)$ (\ref{line}) is as large as possible.
\end{itemize}

\subsubsection{}
The following fundamental theorem of theta divisors was proved by Raynaud when $s(D)=0$ (\cite[Th\'eor\`eme 4.1.1]{R1}), and by Tamagawa when $s(D)\leq 1$ (\cite[Theorem 2.5]{T2}).

\begin{theorem}\label{them-theta}
Suppose that $s(D)\in \{0, 1\}$ (\ref{2.2.4}). Then the Raynaud-Tamagawa theta divisor associated to $\mcE_{D}$ exists.
\end{theorem}

We may ask whether or not the Raynaud-Tamagawa theta divisor $\Theta_{\mcE_{D}}$ exists for an arbitrary $s(D)$ when $X^{\bp}$ is smooth over $k$. 
However, the following example shows that the Raynaud-Tamagawa theta divisor  {\it does not exist} in general. 

\begin{example}
Let  $X=\mbP_{k}^{1}$, $D_{X}=\{0, 1, \muge, \lambda \}$ such that $\lambda \not\in \{0, 1\}$, and $$D=\sum_{x \in D_{X}}\frac{p-1}{2}x$$ an effective divisor on $X$. Then we have $s(D)=2$. Let $([\mcL], D)$ be an arbitrary element of $\widetilde \msP_{X^{\bp}, n}$. We see that $\mcE_{D}$ satisfies $(\star)$ if and only if the elliptic curve defined by the equation $$y^{2}=x(x-1)(x-\lambda)$$ is ordinary. Thus, we cannot expect that $\Theta_{\mcE_{D}}$ exists in general. 
\end{example}

Note that since the existence of $\Theta_{\mcE_{D}}$ implies that $\mcE_{D}$ is a semi-stable bundle, we see that ${\rm deg}(D^{(i)}) \geq {\rm deg}(D)$ holds for each $i \in \{0, 1, \dots, t-1\}$ (\cite[Lemma 2.15]{T2}), where $D^{(i)}$ is the divisor defined in Definition \ref{def-2-4} (ii). We have the following open problem posed by Tamagawa (\cite[Question 2.20]{T2}).
\begin{problem}\label{pbl-3-2}
Suppose that $X^{\bp}$ is a smooth component-generic pointed stable curve (\ref{comgeneric}) of type $(g_{X}, n_{X})$ over $k$. Let $([\mcL], D)$ be an arbitrary element of $\widetilde \msP_{X^{\bp}, n}$. Moreover, suppose that ${\rm deg}(D^{(i)}) \geq {\rm deg}(D)$ holds for each $i \in \{0, 1, \dots, t-1\}$. Does the Raynaud-Tamagawa theta divisor $\Theta_{\mcE_{D}}$ associated to $\mcE_{D}$ exist?
\end{problem}
\noindent
In Section \ref{sec-generic}, we solve Problem \ref{pbl-3-2} under the assumption $s(D)=n_{X}-1$ (see Corollary \ref{coro-3-10} below).

\section{Maximum generalized Hasse-Witt invariants}\label{sec-mghw} 

In this section, we introduce the maximum generalized Hasse-Witt invariants for cyclic prime-to-$p$ admissible coverings. The main result of the present section is Theorem \ref{lem-3-1} which says that the first generalized Hasse-Witt invariant of a prime-to-$p$ cyclic admissible covering attains maximum {\it if and only if} the first generalized Hasse-Witt invariants of the induced admissible coverings of irreducible components attain maximum.

\subsection{Generalized Hasse-Witt invariants for coverings of dual semi-graphs}\label{ghwgraph}

In this subsection, we prove two technical propositions (Proposition \ref{ghwsg-1} and Proposition \ref{ghwsg-2}) concerning generalized Hasse-Witt invariants of coverings of dual semi-graphs induced by a cyclic prime-to-$p$ admissible coverings, which will be used in the proof of Theorem \ref{lem-3-1}. The readers who would like to start the proofs of Theorem \ref{lem-3-1} quickly may skip this subsection, after glancing at the statements of Proposition \ref{ghwsg-1} and Proposition \ref{ghwsg-2}.

\subsubsection{\bf Settings}
We maintain the notation introduced in \ref{defcurves}. Moreover, in this subsection, let $n$ be an {\it arbitrary positive natural number prime to $p$} and $\mu_{n} \subseteq k^{\times}$ the group of $n$th roots of unity. Fix a primitive $n$th root $\zeta$, we may identify $\mu_{n}$ with $\mbZ/n\mbZ$ via the homomorphism $\zeta^{i} \mapsto i$.  Let $$f^{\bp}: Y^{\bp} \migi X^{\bp}$$ be a Galois admissible covering over $k$ with Galois group $\mbZ/n\mbZ$, $f$ the underlying morphism of $f^{\bp}$, $\msN_{X}^{\rm et}$ the set of nodes of $X$ over which $f$ is \'etale, and $$f^{\rm sg}: \Gamma_{Y^{\bp}} \migi \Gamma_{X^{\bp}}$$ the map of dual semi-graphs of $Y^{\bp}$ and $X^{\bp}$ induced by $f^{\bp}$, where ``sg" means ``semi-graph".  Note that $Y^{\bp}$ is connected. 

We put $M_{Y^{\bp}}\defeq H^{1}_{\text{\rm  \'et}}(Y, \mbF_{p})\otimes_{\mbF_{p}} k$. Then $M_{Y^{\bp}}$ is a  $k[\mu_{n}]$-module and admits the following canonical decomposition $$M_{Y^{\bp}}=\bigoplus_{j\in \mbZ/n\mbZ}M_{Y^{\bp}}(j),$$ where $\zeta \in \mu_{n}$ (or $1\in \mbZ/n\mbZ$) acts on $M_{Y^{\bp}}(j)$ as the $\zeta^{j}$-multiplication. Note that $\text{dim}_{k}(M_{Y^{\bp}}(j))$, $j\in \mbZ/n\mbZ$, is a generalized Hasse-Witt invariant of the Galois admissible covering $f^{\bp}$ (\ref{ghw}).

\subsubsection{}
The dual semi-graph $\Gamma_{Y^{\bp}}$ admits a natural action of $\mbZ/n\mbZ$ induced by the Galois admissible covering $f^{\bp}$. We write $M_{\Gamma_{X^{\bp}}}$, $M_{\Gamma_{Y^{\bp}}}$, and $M_{\Gamma_{Y^{\bp}}}^{\vee}$ for $H^{1}(\Gamma_{X^{\bp}}, \mbF_{p})\otimes k$, $H^{1}(\Gamma_{Y^{\bp}}, \mbF_{p})\otimes k$, and the dual vector space $\pi_{1}^{\rm top}(\Gamma_{Y^{\bp}})^{\rm ab} \otimes_{\mbZ} \mbF_{p}\otimes_{\mbF_{p}}k$ of $M_{\Gamma_{Y^{\bp}}}$, respectively. Moreover, we {\it fix} a basis $\{\lambda_{q}\}_{q}$ for $M_{\Gamma_{Y^{\bp}}}$. Then we have a dual basis $\{\lambda_{q}^{\vee}\}_{q}$ for $M_{\Gamma_{Y^{\bp}}}^{\vee}$.

Let $l \subseteq \Gamma_{Y^{\bp}}$ be a loop and $\alpha_{l}^{\vee}\defeq \sum_{q}a_{q}\lambda_{q}^{\vee}$ the vector of $M_{\Gamma_{Y^{\bp}}}^{\vee}$ corresponding to $l$. We shall call $$\alpha_{l}\defeq \sum_{q}a_{q}\lambda_{q} \in M_{\Gamma_{Y^{\bp}}}$$ {\it the vector of $M_{\Gamma_{Y^{\bp}}}$ corresponding to $l$}. 

Moreover, we shall say that $l$ is {\it a minimal loop of $\Gamma_{Y^{\bp}}$} if, for any loop $l' \subseteq l \subseteq \Gamma_{Y^{\bp}}$ such that $l$ and $l'$ are homotopic, then $l = l'$.

\subsubsection{}
Let $r\in \mbZ/n\mbZ$. We denote by $r\cdot l$ and $r\cdot \alpha_{l}^{\vee}$ the natural actions of $r$ on $l$ and $\alpha_{l}^{\vee}$, respectively. Note that we have $r\cdot\alpha_{l}^{\vee}=\alpha_{r\cdot l}^{\vee}$. 

On the other hand, we denote by $r*\alpha_{l}$ the  action of $r \in \mbZ/n\mbZ$ on $\alpha_{l} \in M_{\Gamma_{Y^{\bp}}}$ induced by the action of $\mbZ/n\mbZ$ on $M_{\Gamma_{Y^{\bp}}}^{\vee}$. Then we have $(-r)* \alpha_{l}=\alpha_{r\cdot l}$. For convenience, we put $$r\cdot \alpha_{l} \defeq (-r)* \alpha_{l}.$$ Then we have $r\cdot \alpha_{l}=\alpha_{r\cdot l}$ and $r''\cdot (r'\cdot \alpha_{l})=(r'+r'')\cdot \alpha_{l}$ for all $r', r'' \in \mbZ/n\mbZ$. Moreover, $M_{\Gamma_{Y^{\bp}}}$ is a  $k[\mu_{n}]$-module and admits the following canonical decomposition $$M_{\Gamma_{Y^{\bp}}}=\bigoplus_{j\in \mbZ/n\mbZ}M_{\Gamma_{Y^{\bp}}}(j),$$ where $\zeta \in \mu_{n}$ (or $1\in \mbZ/n\mbZ$) acts on $M_{\Gamma_{Y^{\bp}}}(j)$ as the $\zeta^{j}$-multiplication. 


\subsubsection{} Firstly, we have the following lemma.

\begin{lemma}\label{topcov}
We maintain the notation introduced above. Then the following statements hold:

(i) Suppose that $f^{\bp}: X^{\bp} \migi Y^{\bp}$ is a Galois \'etale covering (i.e., $f$ is \'etale) with Galois group $\mbZ/n\mbZ$, and that $X^{\bp}$ and $Y^{\bp}$ are ordinary (\ref{def-1-3}). Then we have \begin{eqnarray*}
{\rm dim}_{k}(M_{Y^{\bp}}(j))=\left\{ \begin{array}{ll}
g_{X}, & \text{if} \ j=0,
\\
g_{X}-1, & \text{if} \ j \in \{1, \dots, n-1\},
\end{array} \right.
\end{eqnarray*}

(ii) Suppose that $f^{\rm sg}: \Gamma_{Y^{\bp}} \migi \Gamma_{X^{\bp}}$ is a Galois topological covering with Galois group $\mbZ/n\mbZ$. Then we have \begin{eqnarray*}
{\rm dim}_{k}(M_{\Gamma_{Y^{\bp}}}(j))=\left\{ \begin{array}{ll}
r_{X}, & \text{if} \ j=0,
\\
r_{X}-1, & \text{if} \ j \in \{1, \dots, n-1\},
\end{array} \right.
\end{eqnarray*}
where $r_{X}$ denotes the Betti number of $\Gamma_{X^{\bp}}$ (\ref{defcurves}).
\end{lemma}

\begin{proof}
(i) Since the underlying morphism $f$ is \'etale, there exists a line bundle $\mcL$ on $X$ whose degree is $0$, and whose order is $n$, such that $$f_{*}\mcO_{Y} \cong \bigoplus_{j \in \mbZ/n\mbZ}\mcL^{\otimes j}.$$ Then we have $\text{dim}_{k}(M_{Y^{\bp}}(j))=\text{dim}_{k}(H^{1}(X, \mcL))=g_{X}-1$ for all $j \in \{1, \dots, n-1\}$ and $\text{dim}_{k}(M_{Y^{\bp}}(0))=\text{dim}_{k}(H^{1}(X, \mcO_{X}))=g_{X}$. This completes the proof of (i).

(ii) Since this is a topological question, to verify (ii), by adding certain marked points, we may assume that $X^{\bp}$ is an {\it ordinary} pointed stable curve such that the normalization of every irreducible component is {\it isomorphic to} $\mbP_{k}^{1}$. Then $Y^{\bp}$ is also an {\it ordinary} pointed stable curve such that the normalizations of irreducible components are isomorphic to $\mbP_{k}^{1}$. Then (ii) follows immediately from (i). This completes the proof of the lemma.
\end{proof}

\subsubsection{}\label{etalegraph} 
We maintain the notation introduced above. Suppose that $f^{\bp}: Y^{\bp} \migi X^{\bp}$ is a Galois {\it \'etale} covering with Galois group $\mbZ/n\mbZ$, that the set of irreducible components of $X^{\bp}$ is $\{X_{1}, X_{2}\}$, and that $X_{1}$, $X_{2}$ are non-singular. Write $v_{1}$ and $v_{2}$ for the vertices of $\Gamma_{X^{\bp}}$ corresponding to $X_{1}$ and $X_{2}$, respectively. Let $f^{\bp}_{i}\defeq \widetilde f^{\bp}_{v_{i}}: Y_{i}^{\bp} \defeq \widetilde Y_{v_{i}}^{\bp} \migi X_{i}^{\bp} \defeq \widetilde X_{v_{i}}^{\bp}$, $i\in \{1, 2\}$, be the Galois multi-admissible covering over $k$ induced by $f^{\bp}$ (see \ref{smoothpointed} for $\widetilde X_{v_{i}}^{\bp}$). Moreover, we suppose that the following conditions are satisfied:

\begin{itemize}
\item $X_{1}^{\bp}$ and $X^{\bp}_{2}$ are ordinary pointed stable curves of type $(1, 1)$. This implies that the connected components of $Y_{i}$ are  non-singular curves of genus (resp. $p$-rank) $1$.

\item $Y_{1}$ is connected and $\#(\pi_{0}(Y_{2}))=t$, where $\pi_{0}(Y_{2})$ denotes the set of connected components of $Y_{2}$. Then the decomposition group of a connected component of $Y_{2}$ is a subgroup of $\mbZ/n\mbZ$ with order $s\defeq n/t$.
\end{itemize}
Note that the structures of maximal prime-to-$p$ quotients of admissible fundamental groups (\ref{stradm}) imply that the above Galois admissible covering exists.

Write $w_{1} \in v(\Gamma_{Y^{\bp}})$ for the vertex corresponding to $Y_{1}$ and $w_{2, b} \in v(\Gamma_{Y^{\bp}})$, $b\in \{1, \dots, t\}$ for the vertex corresponding to a connected component of $Y_{2}$. Then we see  $v(\Gamma_{Y^{\bp}})=\{w_{1}, w_{2, 1}, \dots, w_{2, t}\}$. Moreover, we write $\{e_{1, b}, \dots, e_{s, b}\}$ for the set of closed edges of $\Gamma_{Y^{\bp}}$ connecting $w_{1}$ and $w_{2, b}$. We define a minimal loop as following (for instance, see Example \ref{exa-ghwsg-2} below): $$l_{a, b}\defeq w_{1}e_{a, b}w_{2, b}e_{a+1 b}w_{1}, \ a\in \{1, \dots, s-1\}, \ b\in \{1, \dots, t\}.$$ Then we have $1\cdot \alpha_{l_{a, b}}=\alpha_{l_{a, b+1}}$ if $1\leq b \leq t-1$, and $1\cdot \alpha_{l_{a, t}}=\alpha_{l_{a+1,1}}$ if $1\leq a \leq s-2$, and $$1\cdot \alpha_{l_{s-1, t}}=-\sum_{a=1}^{s-1}\alpha_{l_{a, 1}}.$$ Note that the set of vectors $\{\alpha_{l_{a, b}}\}_{a\in \{1, \dots, s-1\}, b\in \{1, \dots, t\}}$ is to be linearly independent and is a basis for $M_{\Gamma_{Y^{\bp}}}$. 

\begin{lemma}\label{linearlem}
We maintain the notation introduced above. Then we have the following:
\begin{eqnarray*}
{\rm dim}_{k}(M_{\Gamma_{Y^{\bp}}}(j))=\left\{ \begin{array}{ll}
0, & \text{if} \ j\in \{0, s, 2s, \dots, (t-1)s \},
\\
1, & \text{otherwise}.
\end{array} \right.
\end{eqnarray*}
\end{lemma}

\begin{proof}
We put $$M_{1} \defeq H^{1}_{\text{\rm  \'et}}(Y_{1}, \mbF_{p})\otimes_{\mbF_{p}} k, \ M_{2}\defeq \bigoplus_{b=1}^{t} H^{1}_{\text{\rm  \'et}}(Y_{w_{2, b}}, \mbF_{p})\otimes_{\mbF_{p}} k.$$ Then since all the connected components of $Y_{1}$ and $Y_{2}$ are non-singular curves of genus (resp. $p$-rank) $1$, we obtain
\begin{eqnarray*}
{\rm dim}_{k}(M_{1}(j))=\left\{ \begin{array}{ll}
1, & \text{if} \ j=0,
\\
0, & \text{otherwise},
\end{array} \right.
\end{eqnarray*}
\begin{eqnarray*}
{\rm dim}_{k}(M_{2}(j))=\left\{ \begin{array}{ll}
1, & \text{if} \ j\in \{0, s, 2s, \dots, (t-1)s \},
\\
0, & \text{otherwise},
\end{array} \right.
\end{eqnarray*}
where $M_{1}(j) \defeq M_{\Gamma_{Y^{\bp}}}(j) \cap M_{1}$ and $M_{2}(j) \defeq M_{\Gamma_{Y^{\bp}}}(j) \cap M_{2}$. On the other hand, Lemma \ref{topcov} (i) implies $${\rm dim}_{k}(M_{Y^{\bp}}(j))={\rm dim}_{k}(M_{1}(1))+{\rm dim}_{k}(M_{2}(j))+{\rm dim}_{k}(M_{\Gamma_{Y^{\bp}}}(j))$$
\begin{eqnarray*}
=\left\{ \begin{array}{ll}
2, & \text{if} \ j=0,
\\
1, & \text{otherwise}.
\end{array} \right.
\end{eqnarray*}
Then we complete the proof of the lemma.
\end{proof}

\begin{example}\label{exa-ghwsg-2}
We maintain the notation introduced above. If $n=4$ and $b=2$, we have the following

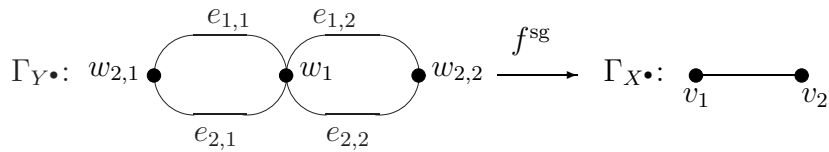
\begin{picture}(300,100)
\put(21,48){$\Gamma_{Y^{\bp}}$:}
\put(100,50){\oval(50,30)}
\put(150,50){\oval(50,30)}
\put(95,70){$e_{1, 1}$}
\put(135,70){$ e_{1, 2}$}
\put(90,25){$e_{2, 1}$}
\put(140,25){$e_{2, 2}$}
\put(75,50){\circle*{5}}
\put(50,50){$w_{2, 1}$}
\put(125,50){\circle*{5}}
\put(175,50){\circle*{5}}
\put(130,50){$w_{1}$}
\put(180,50){$w_{2, 2}$}
\put(210,60){$f^{\rm sg}$}
\put(245,48){$\Gamma_{X^{\bp}}$:}

\put(205,50){\vector(1,0){30}}
\put(280,50){\circle*{5}}
\put(275,40){$v_{1}$}
\put(320,50){\circle*{5}}
\put(320,40){$v_{2}$}
\put(320,50){\line(-1,0){40}}
\end{picture}
\end{example}

Moreover, Lemma \ref{linearlem} implies immediately the following result of linear algebras.

\begin{corollary}\label{ela}
Let $n=st \in \mbN$ be a positive natural number prime to $p$, $V_{b}$, $b\in \{1, \dots, t\}$, a $k$-linear space of dimension $s-1$, and $\{\alpha_{1, b}, \dots, \alpha_{s-1, b}\}$ a basis for $V_{b}$. We put $$V\defeq \bigoplus_{b=1}^{t}V_{b},$$ and let $1 \in \mbZ/n\mbZ$. We define an action of $\mbZ/n\mbZ$ on $V$ as follows: $1\cdot \alpha_{a, b}=\alpha_{a, b+1}$ if $1\leq a \leq s-1$ and $1\leq b \leq t-1$, $1\cdot \alpha_{a, t}=\alpha_{a+1, 1}$ if $1\leq a \leq s-2$, and $$1\cdot \alpha_{s-1, t}=-\sum_{a=1}^{s-1}\alpha_{a, 1}.$$ Then we have the following:
\begin{eqnarray*}
{\rm dim}_{k}(V(j))=\left\{ \begin{array}{ll}
0, & \text{if} \ j\in \{0, s, 2s, \dots, (t-1)s \},
\\
1, & \text{otherwise},
\end{array} \right.
\end{eqnarray*}
where $V(j) \subseteq V$ denotes the $k$-linear subspace on which  $1\in \mbZ/n\mbZ$ acts as the $\zeta^{j}$-multiplication. 
\end{corollary}

\subsubsection{} Next, we calculate generalized Hasse-Witt invariants of dual semi-graphs when $X^{\bp}$ is irreducible.

\begin{proposition}\label{ghwsg-1}
We maintain the notation introduced above. Suppose that  $v(\Gamma_{X^{\bp}})=\{v_{X}\}$. Then we have \begin{eqnarray*}
{\rm dim}_{k}(M_{\Gamma_{Y^{\bp}}}(1))=\left\{ \begin{array}{ll}
\#(\msN_{X}^{\rm et})-1, & \text{if} \ \#(v(\Gamma_{Y^{\bp}}))=n,
\\
\#(\msN_{X}^{\rm et}), & \text{if} \ \#(v(\Gamma_{Y^{\bp}}))\neq n.
\end{array} \right.
\end{eqnarray*}

\end{proposition}

\begin{proof}
Suppose that $\#(v(\Gamma_{Y^{\bp}}))=n$. Then we have that $f^{\rm sg}: \Gamma_{Y^{\bp}} \migi \Gamma_{X^{\bp}}$ is a Galois topological covering with Galois group $\mbZ/n\mbZ$. Then Lemma \ref{topcov} (ii) implies $$\text{dim}_{k}(M_{\Gamma_{Y^{\bp}}}(i))=r_{X}-1=\#(e^{\rm cl}(X^{\bp}))-\#(v(\Gamma_{X^{\bp}}))+1-1=\#(e^{\rm cl}(X^{\bp}))-1=\#(\msN_{X}^{\rm et})-1$$ for all $i\in \{1, \dots, n-1\}$ and $\text{dim}_{k}(M_{\Gamma_{Y^{\bp}}}(0))=r_{X}$. Thus, we obtain the proposition when $\#(v(\Gamma_{Y^{\bp}}))= n$. To verify the proposition, we may assume  $\#(v(\Gamma_{Y^{\bp}}))\neq  n$

Since $\#(v(\Gamma_{X^{\bp}}))=1$, we have that the linear space $M_{\Gamma_{X^{\bp}}}$ is spanned by the vectors corresponding to the closed edges of $\Gamma_{X^{\bp}}$, and that $\text{dim}_{k}(M_{\Gamma_{X^{\bp}}})=\#(e^{\rm cl}(\Gamma_{X^{\bp}}))=\#(e^{\rm lp}(\Gamma_{X^{\bp}}))$. Write $E^{\rm lp} \subseteq e^{\rm cl}(\Gamma_{X^{\bp}})$ for the subset of closed edges such that $(f^{\rm sg})^{-1}(E^{\rm lp})\subseteq e^{\rm lp}(\Gamma_{Y^{\bp}})$ and $E^{\rm tr} \defeq e^{\rm cl}(\Gamma_{X^{\bp}})\setminus E^{\rm lp}$, where ``tr" means ``tree". Note that for every $e_{X} \in E^{\rm tr}$,  every closed edge $e_{Y}\in (f^{\rm sg})^{-1}(e_{X})$ abuts exactly to two different vertices of $\Gamma_{Y^{\bp}}$. Moreover, we write $E^{\rm lp, et}$ (resp. $E^{\rm tr, et}$) for the subset of $E^{\rm lp}$ (resp. $E^{\rm tr}$) such that $f$ is \'etale over all the nodes corresponding to the elements of $E^{\rm lp, et}$ (resp. $E^{\rm tr, et}$). 

Let $e_{X} \in E^{\rm lp}$, $e_{Y} \in (f^{\rm sg})^{-1}(e_{X})$, and let $D_{e_{Y}}  \subseteq \mbZ/n\mbZ$ be the decomposition group of $e_{Y}$ and $m_{e_{Y}}$ the order of $D_{e_{Y}}$. Note that $e_{Y}$ is a minimal loop of $\Gamma_{Y^{\bp}}$, and that $D_{e_{Y}}$ does not depend on the choice of $e_{Y}$ (i.e., $D_{e'_{Y}}=D_{e''_{Y}}$ for all $e'_{Y}, e''_{Y} \in (f^{\rm sg})^{-1}(e_{X})$).  Write $n_{e_{Y}}$ for $n/m_{e_{Y}}$ and $M_{e_{X}}$ for the subspace of $M_{\Gamma_{Y^{\bp}}}$ spanned by $\{\alpha_{e_{Y}}\}_{e_{Y} \in (f^{\rm sg})^{-1}(e_{X})}$. Then we see that $M_{e_{X}}$ is a $k[(\mbZ/n\mbZ)/D_{e_{Y}}]$-module, that $$M_{e_{X}}\subseteq \bigoplus_{0\leq j\leq n_{e_{Y}}-1}M_{\Gamma_{Y^{\bp}}}(jm_{e_{Y}}),$$ and that $\text{dim}_{k}(M_{e_{X}})=n_{e_{Y}}.$ Write $M_{e_{X}}(1) \subseteq M_{e_{X}}$ for the subspace on which $\zeta \in \mu_{n}$ acts as the $\zeta$-multiplication. Thus, we obtain $\text{dim}_{k}(M_{e_{X}} \cap M_{\Gamma_{Y^{\bp}}}(1))=1$ if and only if $m_{e_{Y}}=1$ (i.e., $f$ is \'etale over the node of $X$ corresponding to $e_{X}$).

Suppose that $E^{\rm tr}=\emptyset $. We have  $$M_{\Gamma_{Y^{\bp}}}(1)=\bigoplus_{e_{X} \in E^{\rm lp, et}}M_{e_{X}}(1).$$ Then we obtain  $$\text{dim}_{k}(M_{\Gamma_{Y^{\bp}}}(1))=\#(E^{\rm lp, et})=\#(\msN_{X}^{\rm et})$$ if $E^{\rm tr}=\emptyset $ and $\#(v(\Gamma_{Y^{\bp}}))\neq n$.

Suppose  that $E^{\rm tr}\neq \emptyset$. Then we have $n\neq 2$. Let $v_{Y} \in (f^{\rm sg})^{-1}(v_{X})$, and let $D_{v_{Y}}$ be the decomposition group of $v_{Y}$ which does not depend on the choice of $v_{Y}$ (i.e., $D_{v'_{Y}}=D_{v''_{Y}}$ for all $v'_{Y}, v''_{Y} \in (f^{\rm sg})^{-1}(v_{X})$), $1<m_{v_{Y}}\defeq \#(D_{v_{Y}})$, and $n_{v_{Y}} \defeq n/m_{v_{Y}}$. We put $$(f^{\rm sg})^{-1}(v_{X}) \defeq \{v_{Y, 0}, \dots, v_{Y, n_{v_{Y}}-1}\},$$ which admits a natural action of $r\in \mbZ/n\mbZ$ such that $r\cdot v_{Y, 0}=v_{Y, \overline r}$, where $\overline r$ denotes the image of $\mbZ/n\mbZ\migisurj (\mbZ/n\mbZ)/D_{v_{Y}}\isom \mbZ/n_{v_{Y}}\mbZ$. Moreover, without loss of generality, we may assume $Y_{v_{Y, i}} \cap Y_{v_{Y, i+1}} \neq \emptyset$ for all $i\in \{0, \dots, n_{v_{Y}}-2\}$, $Y_{v_{Y, n_{Y}-1}} \cap Y_{v_{Y, 0}} \neq \emptyset$, and $Y_{v_{Y, j'}} \cap Y_{v_{Y, j''}}=\emptyset$ for all $j', j'' \in \{0, \dots, n_{v_{Y}}-1\}$.

Let  $$T_{Y, e_{X}} \defeq \{e_{Y, 0}, \dots, e_{Y, m_{v_{Y}-1}}\} \subseteq (f^{\rm sg})^{-1}(e_{X}) \subseteq e^{\rm cl}(\Gamma_{Y^{\bp}}), \ e_{X} \in E^{\rm tr, et},$$ be the subset of closed edges such that $v^{\Gamma_{Y^{\bp}}}(e_{Y, i})=\{v_{Y, 0}, v_{Y, 1}\}$ for all $i\in \{0, \dots, m_{v_{Y}}-1\}$ (see \ref{graph} for $v^{\Gamma_{Y^{\bp}}}(e_{Y, i})$). Then $$l_{e_{X}, i}\defeq v_{Y, 0}e_{Y, i}v_{Y, 1}e_{Y,i+1}v_{Y,0}, \ i \in \{0, \dots, m_{v_{Y}}-2\},$$ can be regarded as a minimal loop of $\Gamma_{Y^{\bp}}$ (for instance, see Example \ref{exa-ghwsg-1} (a) below). Moreover, the set of vectors $$\{j\cdot \alpha_{l_{e_{X}, i}}\}_{i\in \{0, \dots, m_{v_{Y}}-2\}, j\in \{0, \dots,  n_{v_{Y}}-1\}, e_{X} \in E^{\rm tr, et}} \subseteq M_{\Gamma_{Y^{\bp}}}$$ is to be linearly independent. We denote by $M_{T_{Y, e_{X}}}\subseteq M_{\Gamma_{Y^{\bp}}}$, $e_{X}\in E^{\rm tr, et}$, the subspace spanned by $$\{j\cdot \alpha_{l_{e_{X}, i}}\}_{i\in \{0, \dots, m_{v_{Y}}-2\}, j\in \{0, \dots,  n_{v_{Y}}-1\}}.$$  Then we see that $M_{T_{Y, e_{X}}}$, $e_{X}\in E^{\rm tr, et}$, is a $k[\mu_{n}]$-module. Moreover, we have $$1\cdot \bigl(j\cdot \alpha_{l_{e_{X}, i}})=(j+1)\cdot \alpha_{l_{e_{X}, i}}$$ if $0\leq j\leq n_{v_{Y}}-2$ and $0\leq i \leq m_{v_{Y}}-2$, $$1\cdot \bigl((n_{v_{Y}}-1)\cdot \alpha_{l_{e_{X}, i}}\bigl)=\alpha_{l_{e_{X}, i+1}}$$ if $0\leq i \leq m_{v_{Y}}-3$, and 
$$1\cdot \bigl((n_{v_Y}-1)\cdot \alpha_{l_{e_{Y}, m_{v_{Y}}-2}}\bigl)=-\sum_{i=0}^{m_{v_{Y}}-2}\alpha_{l_{e_{X}}, i}.$$ Then Corollary \ref{ela} implies $$\text{dim}_{k}(M_{T_{Y, e_{X}}}(1))=1,$$  where $M_{T_{Y, e_{X}}}(1)$ denotes the subspace on which $\zeta\in\mu_{n}$ acts as the $\zeta$-multiplication. 

Let $m\defeq \#(E^{\rm tr, et})$ and $E^{\rm tr, et}\defeq \{e_{X, 1}, \dots, e_{X, m}\}$. Let $e_{Y}^{i} \in T_{Y, e_{X, i}}$, $i\in \{1, \dots, m\}$, be an arbitrary element. Note that $v^{\Gamma_{Y^{\bp}}}(e_{Y}^{i})=\{v_{Y, 0}, v_{Y, 1}\}$. Then $$l_{i} \defeq v_{Y, 0}e_{Y}^{i}v_{Y, 1}e^{i+1}_{Y}v_{Y, 0}, \ i\in \{1, \dots, m-1\},$$ can be regarded as a minimal loop of $\Gamma_{Y^{\bp}}$ (for instance, see Example \ref{exa-ghwsg-1} (b) below). Note that $e^{i+1}_{Y} \in T_{Y, e_{X, i+1}}$. Let $$\alpha_{i} \defeq \sum_{j=0}^{m_{v_{Y}}-1}jn_{v_{Y}}\cdot \alpha_{l_{i}}.$$ Note that the decomposition group of $\alpha_{i}$ is $D_{v_{Y}}$. Then the set of vectors $$\{j\cdot \alpha_{i}\}_{i\in \{1, \dots, m-1\}, j\in \{0, \dots, n_{v_{Y}}-1\}} \subseteq M_{\Gamma_{Y^{\bp}}}$$ is to be linearly independent. We denote by $M_{i} \subseteq M_{\Gamma_{Y^{\bp}}}$, $i \in \{1, \dots, m-1\}$, the subspace spanned by $\{j\cdot \alpha_{i}\}_{j\in \{0, \dots, n_{v_{Y}}-1\}}.$ Then $M_{i}$ is a $k[(\mbZ/n\mbZ)/D_{v_{Y}}]$-module. Thus, we obtain that $$M_{i} \subseteq \bigoplus_{0\leq j <n_{v_{Y}}-1}M_{\Gamma_{Y^{\bp}}}(jm_{v_Y}), \ i\in \{1, \dots, m-1\}.$$ Moreover, we see that the set of vectors $$\{j\cdot \alpha_{i}\}_{i\in \{1, \dots, m-1\}, j\in \{0, \dots, n_{v_{Y}}-1\}}\cup \{j\cdot \alpha_{l_{e_{X}, i}}\}_{i\in \{0, \dots, m_{v_{Y}}-2\}, j\in \{0, \dots,  n_{v_{Y}}-1\}, e_{X} \in E^{\rm tr, et}} \subseteq M_{\Gamma_{Y^{\bp}}}$$ is to be linearly independent.

Let $e_{Y}\in T_{Y, e_{X}}$ for some $e_{X}\in E^{\rm tr, et}$. Then we see that $j\cdot e_{Y}$, $j\in \{0, \dots, n_{v_{Y}}-2\}$, abuts to $v_{Y, j}$ and $v_{Y, j+1}$ (i.e., $v^{\Gamma_{Y^{\bp}}}(j\cdot e_{Y})=\{v_{Y, j}, v_{Y, j+1}\}$), and that $(n_{v_{Y}}-1)\cdot e_{Y}$ abuts to $v_{Y, n_{Y}-1}$ and $v_{Y, 0}$ (i.e., $v^{\Gamma_{Y^{\bp}}}((n_{v_{Y}}-1)\cdot e_{Y})=\{v_{Y, n_{v_{Y}}-1}, v_{Y, 0}\}$). Thus, we have that $$l_{Y} \defeq v_{Y, 0}(0\cdot e_{Y})v_{Y, 1}\cdots v_{Y, n_{Y}-1}((n_{v_{Y}}-1)\cdot e_{Y})v_{Y,0}$$ can be regarded as a minimal loop of $\Gamma_{Y^{\bp}}$ (for instance, see Example \ref{exa-ghwsg-1} (c) below). We put $$\alpha_{\pi} \defeq \sum_{j=0}^{m_{v_{Y}}-1}jn_{v_Y}\cdot\alpha_{l_{Y}}  \in M_{\Gamma_{Y^{\bp}}}.$$ Let us prove that the decomposition group of $\alpha_{\pi}$ is $\mbZ/n\mbZ$. Note that $\alpha_{\pi}$ corresponds to the loop $$\pi \defeq l_{Y}(n_{v_Y}\cdot l_{Y})\cdots((m_{v_{Y}}-1)n_{v_Y}\cdot l_{Y}).$$ Let $e$ be an arbitrary closed edge which is  contained in $\pi$. Then there exists $r \in \mbZ/n_{v_{Y}}\mbZ$ such that $e$ abuts to $v_{Y, r}$ and $v_{Y, r+1}$. Thus, $e$ can be regraded as an oriented edge induced by the oriented loop $\pi$. The starting of $e$ is $v_{Y, r}$ (resp. $v_{Y, n_{v_{Y}}-1}$), and the ending of $e$ is $v_{Y, r+1}$ if $0\leq r\leq n_{Y}-2$ (resp. $v_{Y, 1}$ if $r=n_{Y}-1$). Consider the action of $1\in \mbZ/n\mbZ$ on $e$. We see that $1\cdot e$ is an oriented edge which is contained in $\pi$. Moreover, the starting of $1\cdot e$ is $v_{Y, r+1}$, and the ending of $1\cdot e$ is $v_{Y, r+2}$, where $r+1, r+2 \in \mbZ/n_{v_{Y}}\mbZ$. Namely, we have $1\cdot \pi=\pi$. Thus, the stabilizer of $\alpha_{\pi}$ is $\mbZ/n\mbZ$.  Write $M_{\pi}\subseteq M_{\Gamma_{Y^{\bp}}}$ for the subspace spanned by $\alpha_{\pi}$. Then we have  that  $$M_{\pi} \subseteq M_{\Gamma_{Y^{\bp}}}(0),$$ and that $$\{j\cdot \alpha_{i}\}_{i\in \{1, \dots, m-1\}, j\in \{0, \dots, n_{v_{Y}}-1\}}\cup \{j\cdot \alpha_{l_{e_{X}, i}}\}_{i\in \{1, \dots, m_{v_{Y}}-1\}, j\in \{0, \dots,  n_{v_{Y}}-1\}, e_{X} \in E^{\rm tr, et}} \cup \{\alpha_{\pi}\}$$ is to be linearly independent. 

Let $l \subseteq \Gamma_{Y^{\bp}}$ be an arbitrary minimal loop of $\Gamma_{Y^{\bp}}$ and $V_{l}\subseteq M_{\Gamma_{Y^{\bp}}}$ the subspace spanned by $\{j\cdot \alpha_{l}\}_{j\in \mbZ/n\mbZ}$. Then we have $$V_{l} \subseteq \bigl(\bigoplus_{e_{X} \in E^{\rm tr, et}} M_{T_{Y, e_{X}}}\bigl)\oplus \bigl(\bigoplus_{e_{X}\in E^{\rm lp, et}} M_{e_{X}}\bigl)\oplus \bigl(\bigoplus_{i=1}^{m}M_{i}\bigl) \oplus M_{\pi}.$$ Moreover, we see that $$V_{l}\subseteq \bigoplus_{0\leq j \leq n/m}M_{\Gamma_{Y^{\bp}}}(jm)$$   if all the decomposition groups of the elements of $e^{\rm cl}(\Gamma_{Y^{\bp}}) \cap l$ are not trivial, where $m>1$ is a natural number such that $m  |  n$. Thus, we obtain $$M_{\Gamma_{Y^{\bp}}}(1)=\bigl(\bigoplus_{e_{X} \in E^{\rm tr, et}} M_{T_{Y, e_{X}}}(1)\bigl)\oplus \bigl(\bigoplus_{e_{X} \in E^{\rm lp, et}} M_{e_{X}}(1)\bigl).$$ This implies $$\text{dim}_{k}(M_{\Gamma_{Y^{\bp}}}(1))=\#(E^{\rm tr, et})+\#(E^{\rm lp, et})=\#(\msN_{X}^{\rm et})$$ if $E^{\rm tr}\neq \emptyset $ and $\#(v(\Gamma_{Y^{\bp}}))\neq  n$. We complete the proof of the proposition.
\end{proof}

\begin{example}\label{exa-ghwsg-1}
We maintain the notation introduced in the proof of Proposition \ref{ghwsg-1}.

(a) If $n=4$, $m_{v_{Y}}=2$ and $\#(E^{\rm tr, et})=2$, we have the following:

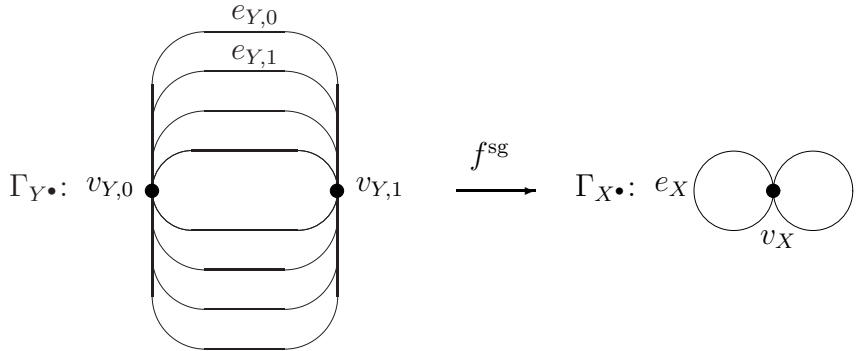
\begin{picture}(300,130)
\put(21,48){$\Gamma_{Y^{\bp}}$:}
\put(110,50){\oval(70,120)}
\put(105,115){$e_{Y, 0}$}
\put(110,50){\oval(70,90)}
\put(105,100){$e_{Y, 1}$}
\put(110,50){\oval(70,60)}

\put(110,50){\oval(70,30)}

\put(110,50){\oval(70,30)}
\put(75,50){\circle*{5}}
\put(50,50){$v_{Y, 0}$}
\put(145,50){\circle*{5}}
\put(152,50){$v_{Y, 1}$}
\put(195,60){$f^{\rm sg}$}
\put(235, 48){$\Gamma_{X^{\bp}}$:}
\put(190,50){\vector(1,0){30}}
\put(310,50){\circle*{5}}
\put(305,30){$v_{X}$}
\put(295,50){\circle{30}}
\put(265, 50){$e_{X}$}
\put(325,50){\circle{30}}
\end{picture}

\bigskip
\noindent
where $v_{Y, 1}=1\cdot v_{Y, 0}$. \\

(b) For instance, if $n=4$, $m_{v_{Y}}=2$ and $\#(E^{\rm tr, et})=2$, we have the following:

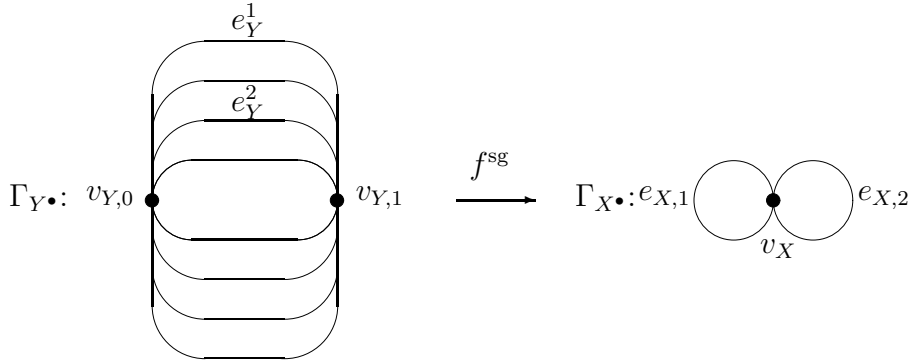
\begin{picture}(300,130)
\put(21,48){$\Gamma_{Y^{\bp}}$:}
\put(110,50){\oval(70,120)}
\put(110,50){\oval(70,90)}
\put(110,50){\oval(70,60)}
\put(110,50){\oval(70,30)}
\put(105,115){$e^{1}_{Y}$}
\put(105,84){$e^{2}_{Y}$}
\put(110,50){\oval(70,30)}
\put(75,50){\circle*{5}}
\put(50,50){$v_{Y, 0}$}
\put(145,50){\circle*{5}}
\put(152,50){$v_{Y, 1}$}
\put(195,60){$f^{\rm sg}$}
\put(235, 48){$\Gamma_{X^{\bp}}$:}
\put(190,50){\vector(1,0){30}}
\put(310,50){\circle*{5}}
\put(305,30){$v_{X}$}
\put(295,50){\circle{30}}
\put(259, 50){$e_{X, 1}$}
\put(342, 50){$e_{X, 2}$}
\put(325,50){\circle{30}}
\end{picture}

\bigskip
\noindent
where $v_{Y, 1}=1\cdot v_{Y, 0}$. \\

(c) If $n=4$, $m_{v_{Y}}=2$ and $\#(E^{\rm tr, et})=2$, we have the following:

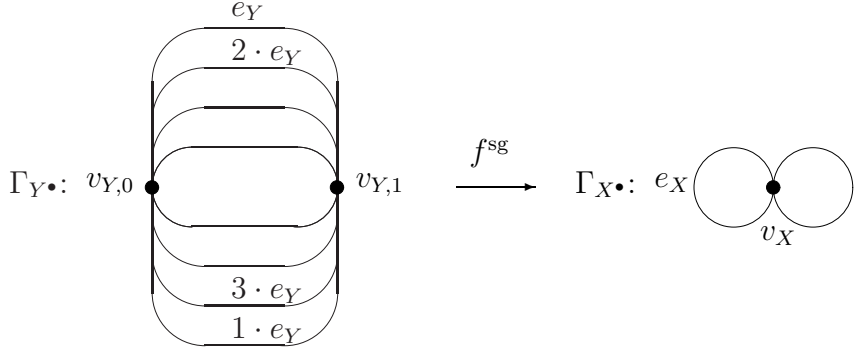
\begin{picture}(300,130)
\put(21,48){$\Gamma_{Y^{\bp}}$:}
\put(110,50){\oval(70,120)}
\put(105,115){$e_{Y}$}
\put(110,50){\oval(70,90)}
\put(105,-7){$1\cdot e_{Y}$}
\put(110,50){\oval(70,60)}
\put(105,98){$2\cdot e_{Y}$}
\put(110,50){\oval(70,30)}
\put(105,8){$3\cdot e_{Y}$}
\put(110,50){\oval(70,30)}
\put(75,50){\circle*{5}}
\put(50,50){$v_{Y, 0}$}
\put(145,50){\circle*{5}}
\put(152,50){$v_{Y, 1}$}
\put(195,60){$f^{\rm sg}$}
\put(235, 48){$\Gamma_{X^{\bp}}$:}
\put(190,50){\vector(1,0){30}}
\put(310,50){\circle*{5}}
\put(305,30){$v_{X}$}
\put(295,50){\circle{30}}
\put(265, 50){$e_{X}$}
\put(325,50){\circle{30}}
\end{picture}

\bigskip
\noindent
where $v_{Y, 1}=1\cdot v_{Y, 0}$.

\end{example}

\subsubsection{} We calculate generalized Hasse-Witt invariants of dual semi-graphs when $X^{\bp}$ is not irreducible.

\begin{lemma}\label{lem-ghwsg-2}
We maintain the notation introduced above. Suppose that the set of irreducible components of $X^{\bp}$ is $\{X_{1}, X_{2}\}$, and that $X_{1}$, $X_{2}$ are non-singular. Write $v_{1}$ and $v_{2}$ for the vertices of $\Gamma_{X^{\bp}}$ corresponding to $X_{1}$ and $X_{2}$, respectively. Let $f^{\bp}_{i}\defeq \widetilde f^{\bp}_{v_{i}}: Y_{i}^{\bp} \defeq \widetilde Y_{v_{i}}^{\bp} \migi X_{i}^{\bp} \defeq \widetilde X_{v_{i}}^{\bp}$, $i\in \{1, 2\}$, be the Galois multi-admissible covering over $k$ induced by $f^{\bp}$.  Then the following statements hold:

(i) Suppose that there exists $i\in\{1, 2\}$ such that $\#(v(\Gamma_{Y_{i}^{\bp}}))=n$. Then we have $${\rm dim}_{k}(M_{\Gamma_{Y^{\bp}}}(1))=\#(\msN_{X}^{\rm et})-1.$$

(ii) Suppose that $\#(X_{1} \cap X_{2})=1$ (i.e., $\#(e^{\rm cl}(\Gamma_{X^{\bp}}))=1$). Then we have $${\rm dim}_{k}(M_{\Gamma_{Y^{\bp}}}(1))=$$
\begin{eqnarray*}
\left\{ \begin{array}{ll}
\#(\msN_{X}^{\rm et})-1, & \text{if there exists}\ i\in\{1, 2\} \ \text{such that} \ \#(v(\Gamma_{Y_{i}^{\bp}}))=n,
\\
\#(\msN_{X}^{\rm et}), & \text{if for each} \ i\in \{1, 2\}, \ \#(v(\Gamma_{Y_{i}^{\bp}}))\neq n.
\end{array} \right.
\end{eqnarray*}

\end{lemma}

\begin{proof}
(i) Suppose that either $\#(v(\Gamma_{Y_{1}^{\bp}}))=n$ and $\#(v(\Gamma_{Y_{2}^{\bp}}))=\#(v(\Gamma_{X_{2}^{\bp}}))=1$ or $\#(v(\Gamma_{Y_{1}^{\bp}}))=\#(v(\Gamma_{X_{1}^{\bp}}))=1$ and $\#(v(\Gamma_{Y_{2}^{\bp}}))=n$ holds. Then either $f_{1}^{\bp}$ or $f_{2}^{\bp}$ is a trivial Galois multi-admissible  covering with Galois group $\mbZ/n\mbZ$. Then we see $$\text{dim}_{k}(M_{\Gamma_{Y^{\bp}}}(1))=r_{X}-r_{X_{1}}-r_{X_{2}}=r_{X}=\#(X_{1} \cap X_{2})-1=\#(\msN_{X}^{\rm et})-1,$$ where $r_{X_{i}}$, $i\in \{1, 2\}$, denotes the Betti number of the dual semi-graph of $X^{\bp}_{i}$. 

Suppose that $\#(v(\Gamma_{Y_{1}^{\bp}}))=n$ and $\#(v(\Gamma_{Y_{2}^{\bp}}))=n$. Then $f^{\rm sg}$ is a Galois topological covering of $\Gamma_{X^{\bp}}$ with Galois group $\mbZ/n\mbZ$. Then Lemma \ref{topcov} (ii) implies $$\text{dim}_{k}(M_{\Gamma_{Y^{\bp}}}(1))=r_{X}-1=\#(\msN_{X}^{\rm et})-1.$$ This completes the proof of (i).

(ii) Suppose that there exists $i\in\{1, 2\}$ such that $\#(v(\Gamma_{Y_{i}^{\bp}}))=n$. Then (ii) follows immediately from (i).

Suppose that $\#(v(\Gamma_{Y_{i}^{\bp}}))\neq n$ for all $i\in \{1, 2\}$. Let $w_{1} \in (f^{\rm sg})^{-1}(v_{1})$ and $w_{2}\in (f^{\rm sg})^{-1}(v_{2})$ be arbitrary vertices of $\Gamma_{Y^{\bp}}$. We denote by $D_{1}\subseteq \mbZ/n\mbZ$ and $D_{2} \subseteq \mbZ/n\mbZ$ the decomposition groups of $w_{1}$ and $w_{2}$, respectively, which do not depend on the choices of $w_{1}$ and $w_{2}$. Since $\#(X_{1} \cap X_{2})=1$ (i.e., $\Gamma_{X^{\bp}}$ is a tree), the Galois topological coverings of $\Gamma_{X^{\bp}}$ with cyclic Galois groups do not exist. Namely, either $D_{1}=\mbZ/n\mbZ$ or $D_{2}=\mbZ/n\mbZ$ holds. 
Without loss of generality, we may assume that $D_{1} =\mbZ/n\mbZ \supseteq D_{2}$. We put  $m_{2} \defeq \#(D_{2})$ and $n_{2} \defeq n/m_{2}$.

Let $e_{X}$ be the unique closed edge of $\Gamma_{X^{\bp}}$, $(f^{\rm sg})^{-1}(v_{1})\defeq \{w_{1}\}$, $w_{2, 0} \in (f^{\rm sg})^{-1}(v_{2})$, and $e_{Y, 0} \in (f^{\rm sg})^{-1}(e_{X})$ such that $e_{Y, 0}$ abuts to $w_{1}$ and $w_{2, 0}$ (i.e., $v^{\Gamma_{Y^{\bp}}}(e_{Y, 0})=\{w_{1}, w_{2, 0}\}$). Let $D_{e_{Y, 0}}$ be the decomposition group of $e_{Y, 0}$, $m_{e_{Y, 0}}=\#(D_{e_{Y, 0}})$ and $n_{e_{Y, 0}}\defeq n/m_{e_{Y, 0}}$. Then we see  that $$M_{\Gamma_{Y^{\bp}}}=\bigoplus_{0\leq j \leq n_{e_{Y, 0}}-1}M_{\Gamma_{Y^{\bp}}}(jm_{e_{Y, 0}}).$$ Then we obtain that $\text{dim}_{k}(M_{\Gamma_{Y^{\bp}}}(1))=0$ if $\#(\msN_{X}^{\rm et})=0$.

Suppose that $\#(\msN_{X}^{\rm et})=1$. Then $f^{\rm sg}: \Gamma_{Y^{\bp}} \migi \Gamma_{X^{\bp}}$ is equal to the covering of dual semi-graphs induced by the Galois admissible covering constructed in \ref{etalegraph}. Thus, Lemma \ref{linearlem} implies $\text{dim}_{k}(M_{\Gamma_{Y^{\bp}}}(1))=1=\#(\msN_{X}^{\rm et})$. This completes the proof of the lemma.
\end{proof}


\begin{proposition}\label{ghwsg-2}
We maintain the notation introduced above. Suppose that the set of irreducible components of $X^{\bp}$ is $\{X_{1}, X_{2}\}$, and that $X_{1}$, $X_{2}$ are non-singular.  Then we have $${\rm dim}_{k}(M_{\Gamma_{Y^{\bp}}}(1))=$$
\begin{eqnarray*}
\left\{ \begin{array}{ll}
\#(\msN_{X}^{\rm et})-1, & \text{if there exists}\ i\in\{1, 2\} \ \text{such that} \ \#(v(\Gamma_{Y_{i}^{\bp}}))=n,
\\
\#(\msN_{X}^{\rm et}), & \text{if for each} \ i\in \{1, 2\}, \ \#(v(\Gamma_{Y_{i}^{\bp}}))\neq n.
\end{array} \right.
\end{eqnarray*}
\end{proposition}

\begin{proof}
Suppose that there exists $i\in\{1, 2\}$ such that $\#(v(\Gamma_{Y_{i}^{\bp}}))=n$. Then the proposition follows from Lemma \ref{lem-ghwsg-2} (i). To verify the proposition, we may assume that $\#(v(\Gamma_{Y_{i}^{\bp}}))\neq n$ for all $i\in \{1, 2\}$. Moreover, if $\#(X_{1} \cap X_{2})=1$, then the proposition follows from Lemma \ref{lem-ghwsg-2} (ii). Thus, we may assume that $\#(X_{1} \cap X_{2})\geq 2$.

First, let us construct two Galois admissible coverings associated to $f^{\bp}: Y^{\bp} \migi X^{\bp}$.  Let $R$ be a complete discrete valuation ring with residue field $k$, $K$ the quotient field of $R$, $\overline K$ an algebraic closure of $K$, $e \in e^{\rm cl}(\Gamma_{X^{\bp}})$ an arbitrary closed edge, and $x_{e}$ the node of $X$ corresponding to $e$. By deforming $X$ along $x_{e}$, we obtain a pointed stable curve $\mcX$ over $R$ whose special fiber is $X^{\bp}$, and whose generic fiber $X_{K}^{\bp}$ is an irreducible pointed stable curve over $K$ such that $\#(e^{\rm cl}(\Gamma_{X_{K}^{\bp}}))=\#(e^{\rm lp}(\Gamma_{X_{K}^{\bp}}))=\#(e^{\rm cl}(\Gamma_{X^{\bp}}))-1$, where $\Gamma_{X_{K}^{\bp}}$ denotes the dual semi-graph  of $X_{K}^{\bp}$. Moreover, since the specialization homomorphism of admissible fundamental groups is a surjection, by replacing $R$ by a finite extension of $R$, $f^{\bp}$ can be lifted to a finite morphism $f_{K}^{\bp}: Y^{\bp}_{K} \migi X^{\bp}_{K}$ over $K$ such that $$f_{\setminus e}^{\bp}\defeq f^{\bp}_{K} \times_{K}\overline K: Y_{\setminus e}^{\bp}\defeq Y^{\bp}_{K} \times_{K}\overline K \migi X_{\setminus e}^{\bp}\defeq  X^{\bp}_{K}\times_{K}\overline K$$ is a Galois admissible covering over $\overline K$ with Galois group $\mbZ/n\mbZ$. Note that $X^{\bp}_{\setminus e}$ is irreducible. We write $\Gamma_{Y_{\setminus e}^{\bp}}$ for the dual semi-graph of $Y_{\setminus e}^{\bp}$ and denote by $M_{\Gamma_{Y_{\setminus e}^{\bp}}} \defeq H^{1}(\Gamma_{Y_{\setminus e}^{\bp}}, \mbF_{p})\otimes k.$ Then $M_{\Gamma_{Y_{\setminus e}^{\bp}}}$ is a $k[\mu_{n}]$-module and admits the following canonical decomposition $$M_{\Gamma_{Y_{\setminus e}^{\bp}}}=\bigoplus_{j\in \mbZ/n\mbZ}M_{\Gamma_{Y_{\setminus e}^{\bp}}}(j),$$ where $\zeta \in \mu_{n}$ acts on $M_{\Gamma_{Y_{\setminus e}^{\bp}}}(j)$ as the $\zeta^{j}$-multiplication. 

On the other hand, let $\text{norm}_{e}: X_{e} \migi X$ be the normalization morphism of $X$ over the nodes corresponding to the closed edges which are contained in $e^{\rm cl}(\Gamma_{X^{\bp}}) \setminus \{e\}$. Then we obtain a pointed stable curve $$X_{e}^{\bp}= (X_{e}, D_{X_{e}} \defeq \text{norm}_{e}^{-1}(D_{X}) \cup \{\text{norm}^{-1}(x_{e'})\}_{e' \in e^{\rm cl}(\Gamma_{X^{\bp}}) \setminus \{e\}})$$ over $k$. Note that $X_{e}$ has two non-singular irreducible components $X_{1}$ and $X_{2}$, and that $X_{1}\cap X_{2}=\{x_{e}\}$ in $X_{e}$. Then $f^{\bp}$ induces a Galois multi-admissible covering $$f_{e}^{\bp}: Y_{e}^{\bp} \migi X_{e}^{\bp}$$ over $k$ with Galois group $\mbZ/n\mbZ$.  We write $\Gamma_{Y_{e}^{\bp}}$ for the dual semi-graph of $Y_{e}^{\bp}$ and denote by $M_{\Gamma_{Y_{ e}^{\bp}}} \defeq H^{1}(\Gamma_{Y_{ e}^{\bp}}, \mbF_{p})\otimes k.$ Then $M_{\Gamma_{Y_{ e}^{\bp}}}$ is a $k[\mu_{n}]$-module and admits the following canonical decomposition $$M_{\Gamma_{Y_{e}^{\bp}}}=\bigoplus_{j\in \mbZ/n\mbZ}M_{\Gamma_{Y_{ e}^{\bp}}}(j),$$ where $\zeta \in \mu_{n}$ acts on $M_{\Gamma_{Y_{e}^{\bp}}}(j)$ as the $\zeta^{j}$-multiplication. 

Write $\msN^{\rm et}_{X_{\setminus e}}$ and $\msN^{\rm et}_{X_{e}}$ for the sets of nodes of $X_{\setminus e}$ and $X_{e}$ over which $f_{\setminus e}$ and $f_{e}$ are \'etale, respectively. Then we see immediately  $\#(\msN^{\rm et}_{X_{\setminus e}})+\#(\msN^{\rm et}_{X_{e}})=\#(\msN^{\rm et}_{X}).$ The constructions of $\Gamma_{Y_{\setminus e}^{\bp}}$ and $\Gamma_{Y_{e}^{\bp}}$ imply $M_{\Gamma_{Y^{\bp}}}=M_{\Gamma_{Y_{\setminus e}^{\bp}}}\oplus M_{\Gamma_{Y_{ e}^{\bp}}}$ as $k[\mu_{n}]$-modules. Then we obtain $$M_{\Gamma_{Y^{\bp}}}(1)=M_{\Gamma_{Y_{\setminus e}^{\bp}}}(1)\oplus M_{\Gamma_{Y_{ e}^{\bp}}}(1).$$ On the other hand, Proposition \ref{ghwsg-1} and Lemma \ref{lem-ghwsg-2} imply that $$\text{dim}_{k}(M_{\Gamma_{Y^{\bp}}}(1))=\text{dim}_{k}(M_{\Gamma_{Y_{\setminus e}^{\bp}}}(1))+\text{dim}_{k}(M_{\Gamma_{Y_{e}^{\bp}}}(1))$$ $$=\#(\msN^{\rm et}_{X_{\setminus e}})+\#(\msN^{\rm et}_{X_{e}})=\#(\msN^{\rm et}_{X}).$$  This completes the proof of the proposition.
\end{proof}

\subsection{Generalized Hasse-Witt invariants of curves and their irreducible components}

\subsubsection{\bf Settings} We maintain the notation introduced in \ref{defcurves} and the assumption introduced in \ref{assum-2.3.3} (i.e.,  $n\defeq p^{t}-1$).

\subsubsection{} The main result of the present section is as follows:

\begin{theorem}\label{lem-3-1}
Let $D\in (\mbZ/n\mbZ)^{\sim}[D_{X}]^{0}$ (\ref{2.2.4}) and $\alpha \in {\rm Rev}_{D}^{\rm adm}(X^{\bp}) \setminus \{0\}$ (Definition \ref{def-2-4} (i)). Let $\Pi_{X^{\bp}}$ be the admissible fundamental group of $X^{\bp}$ and $$f^{\bp}: Y^{\bp}=(Y, D_{Y}) \migi X^{\bp}$$ the Galois multi-admissible covering over $k$ with Galois group $\mbZ/n\mbZ$ induced by $\alpha$. For each $v \in v(\Gamma_{X^{\bp}})$, $f^{\bp}$ induces a Galois multi-admissible covering for the smooth pointed stable curve of type $(g_{v}, n_{v})$ associated to $v$ (\ref{smoothpointed}) $$\widetilde f^{\bp}_{v}: \widetilde Y^{\bp}_{v} \migi \widetilde X^{\bp}_{v}$$ over $k$ with Galois group $\mbZ/n\mbZ$. Let $\widetilde \alpha_{v} \in {\rm Hom}(\Pi^{\rm ab}_{\widetilde X^{\bp}_v}, \mbZ/n\mbZ)$ be the homomorphism $\Pi^{\rm ab}_{\widetilde X^{\bp}_v} \migi \Pi_{X^{\bp}}^{\rm ab} \overset{\alpha} \migi \mbZ/n\mbZ$, where $\Pi^{\rm ab}_{\widetilde X^{\bp}_v} \migi \Pi_{X^{\bp}}^{\rm ab}$ is the natural homomorphism induced by the natural (outer) injection $\Pi_{\widetilde X^{\bp}_v} \migiinje \Pi_{X^{\bp}}$. Then  we have (see Definition \ref{def-2-4} (i) for $\gamma_{(\alpha, D)}$ and \ref{2.2.4} for $s(D)$)
\begin{eqnarray*}
\gamma_{(\alpha, D)}= \left\{ \begin{array}{ll}
g_{X}-1, & \text{if} \ {\rm Supp}(D)= \emptyset,
\\
g_{X}+s(D)-1, & \text{if} \ {\rm Supp}(D) \neq \emptyset,
\end{array} \right.
\end{eqnarray*}
if and only if, for each $v\in v(\Gamma_{X^{\bp}})$, 
\begin{eqnarray*}
\gamma_{(\widetilde \alpha_{v}, D_{\widetilde \alpha_{v}})}=  \left\{ \begin{array}{ll}
g_{v}, & \text{if} \ \widetilde \alpha_{v}=0,
\\
g_{v}-1, & \text{if} \ \widetilde \alpha_{v}\neq 0,  \ {\rm Supp}(D_{\widetilde \alpha_{v}})= \emptyset,
\\
g_{v}+s(D_{\widetilde \alpha_{v}})-1, & \text{if} \ \widetilde \alpha_{v}\neq 0, \ {\rm Supp}(D_{\widetilde \alpha_{v}})\neq \emptyset,
\end{array} \right.
\end{eqnarray*}
where ${\rm Supp}(-)$ denotes the support of $(-)$. 

\end{theorem}

\begin{proof}
We prove the theorem by induction on the cardinality $\#(v(\Gamma_{X^{\bp}}))$ of $v(\Gamma_{X^{\bp}})$. Suppose that $\#(v(\Gamma_{X^{\bp}}))=1$ (i.e., $X$ is irreducible). Then we have $D_{\widetilde \alpha_{v}}|_{U_{X_{v}}}=D$ and $$g_{v}=g_{X}-\#(X^{\rm sing}),$$ where $U_{X_{v}}$ is the open subset of $X_{v}$ defined in \ref{smoothpointed}. Moreover, since $\widetilde X_{v}^{\bp}$ is smooth over $k$, we write $([\mcL_{\widetilde\alpha_{v}}], D_{\widetilde\alpha_{v}}) \in \widetilde \msP_{\widetilde X^{\bp}_{v}, n}$ for the pair induced by $\widetilde\alpha_{v}$ (\ref{line}). Write $\msN_{X}^{\rm ra}\subseteq X^{\rm sing}$ for the subset of nodes over which $f$ is ramified and $\msN_{X}^{\rm et} \subseteq X^{\rm sing}$ for the subset of nodes over which $f$ is \'etale. Then we have $s(D_{\widetilde \alpha_{v}})=s(D)+\#(\msN_{X}^{\rm ra})$.

On the other hand, \tch{by Lemma \ref{lem-2-4},} we have \begin{eqnarray*}
{\rm dim}_{k}(H^{1}(\widetilde X_{v}, \mcL_{\widetilde \alpha_v}))=  \left\{ \begin{array}{ll}
g_{v}, & \text{if} \ \widetilde \alpha_{v}=0,
\\
g_{v}-1, & \text{if} \ \widetilde \alpha_{v}\neq 0,  \ {\rm Supp}(D_{\widetilde \alpha_{v}})= \emptyset,
\\
g_{v}+s(D_{\widetilde \alpha_{v}})-1, & \text{if} \ \widetilde \alpha_{v}\neq 0, \ {\rm Supp}(D_{\widetilde \alpha_{v}})\neq \emptyset.
\end{array} \right.
\end{eqnarray*}
Write $\Gamma_{Y^{\bp}}$ for the dual semi-graph of $Y^{\bp}$. The natural $k[\mu_{n}]$-submodule $$H^{1}(\Gamma_{Y^{\bp}}, \mbF_{p})\otimes k \subseteq H^{1}_{\rm \text{\'et}}(Y, \mbF_{p})\otimes k$$ admits the following canonical decomposition $$H^{1}(\Gamma_{Y^{\bp}}, \mbF_{p})\otimes k =\bigoplus_{j \in \mbZ/n\mbZ}M_{\Gamma_{Y^{\bp}}}(j),$$ where $\zeta \in \mu_{n}$ acts on $M_{\Gamma_{Y^{\bp}}}(j)$ as the $\zeta^{j}$-multiplication. By  Proposition \ref{ghwsg-1}, we have 
\begin{eqnarray*}
{\rm dim}_{k}(M_{\Gamma_{Y^{\bp}}}(1))=\left\{ \begin{array}{ll}
\#(\msN_{X}^{\rm et})-1, & \text{if} \ \widetilde \alpha_{v}=0,
\\
\#(\msN_{X}^{\rm et}), & \text{if} \ \widetilde \alpha_{v}\neq 0.
\end{array} \right.
\end{eqnarray*}
Thus, we obtain 
\begin{eqnarray*}
{\rm dim}_{k}(H^{1}(\widetilde X_{v}, \mcL_{\widetilde \alpha_v}))+\text{dim}_{k}(M_{\Gamma_{Y^{\bp}}}(1))=\left\{ \begin{array}{ll}
g_{X}-1, & \text{if} \ \text{Supp}(D)=\emptyset,
\\
g_{X}+s(D)-1, & \text{if} \ \text{Supp}(D)\neq\emptyset.
\end{array} \right.
\end{eqnarray*}
Since $\gamma_{(\alpha, D)}=\gamma_{([\mcL_{\widetilde \alpha_v}], D_{\widetilde \alpha_{v}})}+\text{dim}_{k}(M_{\Gamma_{Y^{\bp}}}(1))$ and $\#(X^{\rm sing})=\#(\msN^{\rm ra}_{X})+\#(\msN^{\rm et}_{X})$, we have  \begin{eqnarray*}
\gamma_{(\alpha, D)}= \left\{ \begin{array}{ll}
g_{X}-1, & \text{if} \ \text{Supp}(D)=\emptyset,
\\
g_{X}+s(D)-1, & \text{if} \ \text{Supp}(D)\neq\emptyset
\end{array} \right.
\end{eqnarray*}
if and only if $\gamma_{(\widetilde \alpha_{v}, D_{\widetilde \alpha_{v}})}=\gamma_{([\mcL_{\widetilde \alpha_{v}}], D_{\widetilde \alpha_v})}={\rm dim}_{k}(H^{1}(\widetilde X_{v}, \mcL_{\widetilde \alpha_v})).$ This completes the proof of the proposition when $\#(v(\Gamma_{X^{\bp}}))=1$.

Suppose $m \defeq \#(v(\Gamma_{X^{\bp}})) \geq 2$. Let $v_{0} \in v(\Gamma_{X^{\bp}})$ be a vertex such that $\Gamma_{X^{\bp}} \setminus \{v_{0}, e^{\Gamma_{X^{\bp}}}(v_{0})\}$ (see \ref{graph} for $e^{\Gamma_{X^{\bp}}}(v_{0})$) is connected (note that $v_{0}$ always exists). Write $X_{1}$ for the topological closure of $X\setminus X_{v_{0}}$ in $X$ and $X_{2}$ for $X_{v_{0}}$. Note that $X_{1}$ is connected. We define a pointed stable curve $$X_{i}^{\bp}=(X_{i}, D_{X_{i}} \defeq (X_{i} \cap D_{X})\cup (X_{1} \cap X_{2})), \ i \in \{1, 2\},$$ of type $(g_{X_{i}}, n_{X_{i}})$ over $k$. Then $f^{\bp}$ induces a Galois multi-admissible covering $$f_{i}^{\bp}: Y_{i}^{\bp} \migi X_{i}^{\bp}, \ i\in \{1, 2\},$$ over $k$ with Galois group $\mbZ/n\mbZ$. Moreover, we denote by $\alpha_{i} \in \text{Hom}(\Pi^{\rm ab}_{X^{\bp}_{i}}, \mbZ/n\mbZ), \ i \in \{1, 2\},$ the composition of the natural homomorphisms $\Pi_{X_{i}^{\bp}}^{\rm ab} \migi \Pi^{\rm ab}_{X^{\bp}} \overset{\alpha}\migi \mbZ/n\mbZ$, where $\Pi_{X^{\bp}_{i}}$ denotes the admissible fundamental group of $X^{\bp}_{i}$. 

Write $\msN_{X_{1}\cap X_{2}}^{\rm ra} \subseteq X_{1} \cap X_{2}$ for the subset of nodes over which $f$ is ramified  and $\msN_{X_{1}\cap X_{2}}^{\rm et} \subseteq X_{1} \cap X_{2}$ for the subset of nodes over which $f$ is \'etale. Then we have $\#(\msN_{X_{1}\cap X_{2}}^{\rm ra})+\#(\msN_{X_{1}\cap X_{2}}^{\rm et})=\#(X_{1} \cap X_{2})$. Moreover, we have  \begin{eqnarray*}
\gamma_{(\alpha_{i}, D_{\alpha_{i}})}\leq \left\{ \begin{array}{ll}
g_{X_{i}}, & \text{if} \ \alpha_{i}=0,
\\
g_{X_{i}}-1, & \text{if} \ \alpha_{i}\neq 0, \ \text{Supp}(D_{\alpha_{i}}) =\emptyset,
\\
g_{X_{i}}+s(D_{\alpha_{i}})-1, & \text{if} \ \alpha_{i}\neq 0, \ \text{Supp}(D_{\alpha_{i}}) \neq\emptyset,
\end{array} \right.
\end{eqnarray*}
for all $i\in\{1, 2\}$.
Note that  the definition of admissible coverings implies that $s(D_{\alpha_{1}})+s(D_{\alpha_{2}})=s(D)+\#(\msN^{\rm ra}_{X_{1}\cap X_{2}})$.

On the other hand, the natural $k[\mu_{n}]$-modules $H^{1}(\Gamma_{Y^{\bp}}, \mbF_{p})\otimes k$ and $H^{1}(\Gamma_{Y_{i}^{\bp}}, \mbF_{p})\otimes k$, $i \in \{1, 2\}$, admit the following canonical decomposition $$H^{1}(\Gamma_{Y^{\bp}}, \mbF_{p})\otimes k =\bigoplus_{j \in \mbZ/n\mbZ}M_{\Gamma_{Y^{\bp}}}(j),$$ $$H^{1}(\Gamma_{Y^{\bp}_{i}}, \mbF_{p})\otimes k =\bigoplus_{j \in \mbZ/n\mbZ}M_{\Gamma_{Y^{\bp}_{i}}}(j), \ i\in \{1, 2\},$$ respectively, where $\Gamma_{Y^{\bp}_{i}}$ denotes the dual semi-graph of $Y_{i}^{\bp}$, and $\zeta \in \mu_{n}$ acts on $M_{\Gamma_{Y^{\bp}}}(j)$ and $M_{\Gamma_{Y^{\bp}_{i}}}(j)$, $i\in \{1, 2\}$, as the $\zeta^{j}$-multiplication, respectively. We put $$\text{dim}_{k}(M_{\Gamma_{Y_{1}\cap Y_{2}}}(1))\defeq \text{dim}_{k}(M_{\Gamma_{Y^{\bp}}}(1))-\text{dim}_{k}(M_{\Gamma_{Y_{1}^{\bp}}}(1))-\text{dim}_{k}(M_{\Gamma_{Y_{2}^{\bp}}}(1)).$$ Let us compute $\text{dim}_{k}(M_{\Gamma_{Y_{1}\cap Y_{2}}}(1))$. Without loss of generality, we may assume that $X_{1}$ and $X_{2}$ are non-singular. Then Proposition \ref{ghwsg-2} implies that the following holds: $${\rm dim}_{k}(M_{\Gamma_{Y_{1}\cap Y_{2}}}(1))=$$
\begin{eqnarray*}
\left\{ \begin{array}{ll}
\#(\msN_{X_{1} \cap X_{2}}^{\rm et})-1, & \text{if there exists}\ i\in\{1, 2\} \ \text{such that} \ \alpha_{i}=0,
\\
\#(\msN_{X_{1} \cap X_{2}}^{\rm et}), & \text{if for each} \ i\in \{1, 2\}, \ \alpha_{i}\neq 0.
\end{array} \right.
\end{eqnarray*}
Since $\alpha\neq 0$, we obtain $$\gamma_{(\alpha, D)}=\gamma_{(\alpha_{1}, D_{\alpha_{1}})}+\gamma_{(\alpha_{2}, D_{\alpha_2})}+\text{dim}_{k}(M_{\Gamma_{Y_{1}\cap Y_{2}}}(1))$$$$\leq g_{X_{1}}+s(D_{\alpha_{1}})+g_{X_{2}}+s(D_{\alpha_{2}})+\#(\msN_{X_{1}\cap X_{2}}^{\rm et})-2=g_{X}+s(D)-1$$
\begin{eqnarray*}
= \left\{ \begin{array}{ll}
g_{X}-1, & \text{if} \ \text{Supp}(D)=\emptyset,
\\
g_{X}+s(D)-1, & \text{if} \ \text{Supp}(D)\neq\emptyset.
\end{array} \right.
\end{eqnarray*}
Thus, we have 
\begin{eqnarray*}
\gamma_{(\alpha, D)}= \left\{ \begin{array}{ll}
g_{X}-1, & \text{if} \ \text{Supp}(D)=\emptyset,
\\
g_{X}+s(D)-1, & \text{if} \ \text{Supp}(D)\neq\emptyset
\end{array} \right.
\end{eqnarray*}
if and only if  \begin{eqnarray*}
\gamma_{(\alpha_{i}, D_{\alpha_i})}=\left\{ \begin{array}{ll}
g_{X_{i}}, & \text{if} \ \alpha_{i}=0,
\\
g_{X_{i}}-1, & \text{if} \ \alpha_{i}\neq 0, \ \text{Supp}(D_{\alpha_{i}})= \emptyset,
\\
g_{X_{i}}+s(D_{\alpha_i})-1, & \text{if} \ \alpha_{i}\neq 0, \ \text{Supp}(D_{\alpha_{i}})\neq  \emptyset,
\end{array} \right.
\end{eqnarray*}
for all $i\in \{1, 2\}$. By induction, the theorem follows from the theorem when $\#(v(\Gamma_{X^{\bp}}))=m-1$ and $\#(v(\Gamma_{X^{\bp}}))=1$. This completes the proof of the theorem.
\end{proof}

\subsubsection{} We define the maximum generalized Hasse-Witt invariants as follows:

\begin{definition}\label{def-3-2}
(i) We put $$\gamma^{\rm max}_{X^{\bp}}\defeq \text{max}_{t\in \mbN}\{\gamma_{(\alpha, D_{\alpha})} \ | \ \alpha \in \text{Hom}(\Pi^{\rm ab}_{X^{\bp}}, \mbZ/(p^{t}-1)\mbZ) \ \text{and} \ \alpha\neq 0\}$$$$=\text{max}_{m\in \mbN \ \text{s.t.}\ (m, p)=1}\{\gamma_{(\alpha, D_{\alpha})} \ | \ \alpha \in \text{Hom}(\Pi^{\rm ab}_{X^{\bp}}, \mbZ/m\mbZ) \ \text{and} \ \alpha\neq 0\},$$ and call $\gamma^{\rm max}_{X^{\bp}}$ {\it the maximum generalized Hasse-Witt invariant of prime-to-$p$ cyclic admissible coverings} of $X^{\bp}$. Note that Theorem \ref{lem-3-1} implies that 
\begin{eqnarray*}
\gamma^{\rm max}_{X^{\bp}} \leq \left\{ \begin{array}{ll}
g_{X}-1, & \text{if} \ n_{X} =0,
\\
g_{X}+n_{X}-2, & \text{if} \ n_{X} \neq 0.
\end{array} \right.
\end{eqnarray*}

(ii) Let $m$ be a natural number prime to $p$ and $\alpha \in \text{Hom}(\Pi^{\rm ab}_{X^{\bp}}, \mbZ/m\mbZ)$. We shall say that {\it $\gamma_{(\alpha, D_{\alpha})}$ attains maximum} if \begin{eqnarray*}
\gamma_{(\alpha, D_{\alpha})}= \left\{ \begin{array}{ll}
g_{X}-1, & \text{if} \ n_{X} =0,
\\
g_{X}+n_{X}-2, & \text{if} \ n_{X} \neq 0.
\end{array} \right.
\end{eqnarray*}
We shall say that {\it $\gamma^{\rm max}_{X^{\bp}}$ attains maximum} if there exist  a prime-to-$p$ natural number $m'$ and an element $\alpha' \in \text{Hom}(\Pi^{\rm ab}_{X^{\bp}}, \mbZ/m'\mbZ)$ such that $\gamma_{(\alpha', D_{\alpha'})}$ attains maximum.
\end{definition}

\subsection{An easy case of main results}

\subsubsection{\bf Settings} We maintain the notation introduced in \ref{defcurves}.

\subsubsection{} By applying Nakajima and Raynaud's results concerning ordinariness of cyclic \'etale coverings, we have the following result.

\begin{proposition}\label{maineasy}
(i) Let $n \in \mbN$ be {\it an arbitrary natural number prime to $p$}, $D \in (\mbZ/n\mbZ)^{\sim}[D_{X}]^{0}$, and $\alpha \in {\rm Rev}_{D}^{\rm adm}(X^{\bp}) \setminus \{0\}$. Suppose that $X^{\bp}$ is an irreducible component-generic pointed stable curve over $k$, and that $n_{X} \leq 1$.  Then  $\gamma_{(\alpha, D)}$ attains maximum. Namely, the following holds: 
$$\gamma_{(\alpha, D)}=g_{X}-1.$$

(ii) Let $X^{\bp}$ be an arbitrary pointed stable curve over $k$. Suppose that $n_{X}\leq 1$. Then $\gamma^{\rm max}_{X^{\bp}}$ attains maximum. Namely, the following holds:
$$\gamma^{\rm max}_{X^{\bp}}=g_{X}-1.$$
\end{proposition}

\begin{proof}
Since $n_{X}\leq 1$, the structures of maximal prime-to-$p$ quotients of admissible fundamental groups (\ref{stradm}) imply $D=0$. Namely, all of the prime-to-$p$ cyclic Galois admissible coverings of $X^{\bp}$ are \'etale over $D_{X}$ (in particular, are \'etale if $X^{\bp}$ is smooth over $k$). 

(i) This follows immediately from \cite[Proposition 4]{N} and Theorem \ref{lem-3-1}.

(ii) Let $v \in v(\Gamma_{X^{\bp}})$, and let $\widetilde X^{\bp}_{v}$ be the smooth pointed stable curve of type $(g_{v}, n_{v})$ over $k$ associated to $v$ (\ref{smoothpointed}). We denote by $$\msV \defeq \{v \in v(\Gamma_{X^{\bp}}) \ | \ g_{v}>0\}.$$ Suppose that $\msV=\emptyset$. Then $n_{X}\leq 1$ implies that $\Gamma_{X^{\bp}}$ is not a tree. Namely, $\Pi_{X^{\bp}}^{\rm top}$ is not trivial. Let $\alpha': \Pi_{X^{\bp}}^{\rm top, ab} \migisurj \mbZ/n\mbZ$ be a surjection and $\alpha: \Pi^{\rm ab}_{X^{\bp}} \migisurj \mbZ/n\mbZ$ the composite of the homomorphisms $\Pi^{\rm ab}_{X^{\bp}} \migisurj \Pi_{X^{\bp}}^{\rm top, ab} \overset{\alpha'}\migisurj \mbZ/n\mbZ$. Then the theorem follows from Lemma \ref{topcov} (ii). 

Suppose that $\msV\neq \emptyset$. Let $v \in \msV$. Then Proposition \ref{prop-theta} and Theorem \ref{them-theta} imply that there exists an element $\widetilde \alpha_{v} \in \text{Rev}_{0}^{\rm adm}(\widetilde X_{v}^{\bp})$ such that $\widetilde \alpha_{v}: \Pi^{\rm ab}_{\widetilde X^{\bp}_{v}} \migisurj \mbZ/n\mbZ$ is a surjective, and that $$\gamma_{(\widetilde \alpha_{v}, 0)}=g_{v}-1.$$ Write $\widetilde f^{\bp}_{v}: \widetilde Y^{\bp}_{v} \migi \widetilde X^{\bp}_{v}$ for the connected Galois \'etale covering with Galois group $\mbZ/n\mbZ$ induced by $\widetilde \alpha_{v}$. 

Let $\msC$ be the set of connected components of $\overline {X \setminus \bigcup_{v \in \msV}X_{v}}$ and $C \in \msC$, where $\overline {X \setminus \bigcup_{v \in \msV}X_{v}}$ denotes the topological closure of $X \setminus \bigcup_{v \in \msV}X_{v}$ in $X$. We define $$C^{\bp}=(C, D_{C}\defeq (C\cap \bigcup_{v \in \msV}X_{v}) \cup (D_{X} \cap C))$$ to be a pointed stable curve over $k$. Note that the normalization of each irreducible component of $C$ is isomorphic to $\mbP^{1}_{k}$.  We put $$Y^{\bp}_{C}\defeq \bigsqcup_{i \in \mbZ/n\mbZ}C_{i}^{\bp},$$ where $C_{i}^{\bp}$ is a copy of $C^{\bp}$. Then we obtain a Galois multi-admissible covering $f_{C}^{\bp}: Y^{\bp}_{C}\migi C^{\bp}$ over $k$ with Galois group $\mbZ/n\mbZ$, where the restriction morphism $f_{C}^{\bp}|_{C^{\bp}_{i}}$ is an identity, and the Galois action is $j(C_{i})=C_{i+j}$ for all $i, j \in \mbZ/n\mbZ$. By gluing $\{\widetilde Y_{v}^{\bp}\}_{v\in \msV} \ \text{and} \ \{Y^{\bp}_{C}\}_{C \in \msC}$ along $\{D_{\widetilde X_{v}}\}_{v\in\msV}$ and $\{D_{C}\}_{C \in \msC}$ \tch{in a way} that is compatible with the gluing of $\{\widetilde X^{\bp}_{v}\}_{v\in \msV} \cup \{C^{\bp}\}_{C \in \msC}$ that gives rise to $X^{\bp}$, we obtain a Galois (\'etale) admissible covering $$f^{\bp}: Y^{\bp} \migi X^{\bp}$$ over $k$ with Galois group $\mbZ/n\mbZ$. Write $\alpha \in \text{Rev}_{0}^{\rm adm}(X^{\bp})$ for an element induced by $f^{\bp}$ such that the composition  of the natural homomorphisms $\Pi^{\rm ab}_{\widetilde X^{\bp}_{v}} \migi \Pi^{\rm ab}_{X^{\bp}} \overset{\alpha}\migi \mbZ/n\mbZ$, $v\in \msV$, is equal to $\widetilde \alpha_{v}$. Then (ii) follows from Theorem \ref{lem-3-1}.
\end{proof}

\subsubsection{}
If $X^{\bp}$ is an arbitrary {\it component-generic pointed stable curve} (i.e., possibly singular) over $k$, we see that Proposition \ref{maineasy} (i) {\it does not hold in general}. For instance, we have the following example:

\begin{example}\label{rem-prop-3-9}

Let $X^{\bp}$ be a component-generic pointed stable curve of type $(g_{X}, 0)$ over $k$ with two smooth irreducible components $X_{1}$ and $X_{2}$ of genus $g_{1}$ and $g_{2}$, respectively. Moreover, suppose that  $X_{1} \cap X_{2}=\{x_{1}, x_{2}, x_{3}\}$. We put $$X_{j}^{\bp}=(X_{j}, D_{X_{j}} \defeq X_{1} \cap X_{2}), \ j \in \{1, 2\}.$$

Let $n\defeq p^{t}-1$, and let $D_{1} \in (\mbZ/n\mbZ)^{\sim}[D_{X_{1}}]^{0}$ be an effective divisor on $X_{1}$ such that $s(D_{1})=2$, and that $s(D_{1}^{(i)})=1<s(D_{1})=2$ for some $i \in \{1, \dots, t-1\}$ (see Definition \ref{def-2-4} (ii) for $D_{1}^{(i)}$). We put $$D_{2} \defeq (n-\text{ord}_{x_{1}}(D_{1}))x_{1}+(n-\text{ord}_{x_{2}}(D_{1}))x_{2}+(n-\text{ord}_{x_{3}}(D_{1}))x_{3}.$$ Then $D_{2} \in (\mbZ/n\mbZ)^{\sim}[D_{X_{2}}]^{0}$ is an effective divisor on $X_{2}$ with degree $n$ (i.e., $s(D_{2})=1$). 

Let $\alpha_{j} \in \text{Rev}_{D_{j}}^{\rm adm}(X_{j}^{\bp})$, $j \in \{1, 2\}$, and let $f_{j}^{\bp}: Y_{j}^{\bp} \migi X_{j}^{\bp}$ be the Galois multi-admissible covering induced by $\alpha_{j}$. Then by gluing $\{X^{\bp}_{j}\}_{j=1, 2}$ and $\{Y^{\bp}_{j}\}_{j=1, 2}$, we obtain a Galois multi-admissible covering $$f^{\bp}: Y^{\bp} \migi X^{\bp}$$ over $k$. Write $\alpha \in \text{Rev}_{0}^{\rm adm}(X^{\bp})$ for the element induced by $f^{\bp}$. Then Theorem \ref{lem-3-1} implies that $\gamma_{(\alpha, 0)}=g_{X}-1=g_{1}+g_{2}+1$ if and only if $\gamma_{(\alpha_{1}, D_{1})}=g_{1}+1$ and $\gamma_{(\alpha_{2}, D_{2})}=g_{2}$. On the other hand, since $s(D_{1}^{(i)})=1$ for some $i\in \{1, \dots, t-1\}$, we have that $\gamma_{(\alpha_{1}, D_{1})}=g_{1}$. Then we have $$\gamma_{(\alpha, 0)}\neq g_{X}-1.$$ Namely, $\gamma_{(\alpha, 0)}$ does not attain maximum.
\end{example} 

On the other hand, if $X^{\bp}$ {\it is not a component-generic pointed stable curve}, then Proposition \ref{maineasy} (i) does not hold in general. In fact, if $k$ is an algebraic closure of $\mbF_{p}$, then Proposition \ref{maineasy} (i) {\it does not hold for any pointed stable curves} over $k$ (see Proposition \ref{coro-3-13} (ii)).

\subsubsection{}
The main goals of the next two sections are to generalize Proposition \ref{maineasy} to the case of (possibly singular) pointed stable curves of an arbitrary type $(g_{X}, n_{X})$. More precisely, we will generalize Proposition \ref{maineasy} as follows:

\subsubsection*{The first main theorem: a generalization of Proposition \ref{maineasy} (i)}  Let $n$ be a natural number prime to $p$ and $D \in (\mbZ/n\mbZ)^{\sim}[D_{X}]^{0}$. Suppose that $D$ satisfies a certain condition (i.e., {\it a necessary condition} that $\gamma_{(\alpha', D)}$ attains maximum for some $\alpha' \in {\rm Rev}_{D}^{\rm adm}(X^{\bp})$). We will prove that $\gamma_{(\alpha, D)}$ attains maximum for {\it all} $\alpha \in {\rm Rev}_{D}^{\rm adm}(X^{\bp})$ when $X^{\bp}$ is a certain component-generic pointed stable curve (Theorem \ref{main-them-1} (i)). Moreover, we will prove that $\gamma_{(\beta, D)}$ attains maximum for {\it some} $\beta \in {\rm Rev}_{D}^{\rm adm}(X^{\bp})$ when $X^{\bp}$ is an {\it arbitrary} component-generic pointed stable curve (Theorem \ref{main-them-1} (ii)).


\subsubsection*{The second main theorem: a generalization of Proposition \ref{maineasy} (ii)} We will prove that $\gamma^{\rm max}_{X^{\bp}}$ attains maximum for {\it arbitrary} pointed stable curve of type $(g_{X}, n_{X})$ over $k$ (Theorem \ref{main-them-2}).

\section{Maximum generalized Hasse-Witt invariants for generic curves}\label{sec-generic}

In the present section, we discuss the maximum generalized Hasse-Witt invariants of cyclic admissible coverings for generic curves. The main result of this section is Theorem \ref{main-them-1}. 

\subsection{Idea} We briefly  explain the idea of our proof of Theorem \ref{main-them-1}.

\subsubsection{\bf Settings} We maintain the notation introduced in \ref{defcurves}. Moreover, we suppose that $X^{\bp}$ is a {\it component-generic} pointed stable curve over $k$ (\ref{comgeneric}).

\subsubsection{} Let $n$ be a natural number prime to $p$, $D \in (\mbZ/n\mbZ)^{\sim}[D_{X}]^{0}$ (\ref{2.2.4}),  $\alpha \in {\rm Rev}_{D}^{\rm adm}(X^{\bp})$ (Definition \ref{def-2-4} (i)), and $f_{\alpha}^{\bp}: X_{\alpha}^{\bp} \migi X^{\bp}$ the Galois multi-admissible covering corresponding to $\alpha$. Note that there exists a natural number $t \in \mbN$ such that $p^{t}=1$ in $(\mbZ/n\mbZ)^{\times}$. Write $m'$ for $(p^{t}-1)/n$. To compute $\gamma_{(\alpha, D)}$, by replacing $\alpha$ and $D$ by the composition of homomorphisms $\Pi_{X^{\bp}}^{\rm ab} \overset{\alpha}\migi \mbZ/n\mbZ \overset{m'}\migi \mbZ/(p^{t}-1)\mbZ$ and $m'D$, we may assume that $n\defeq p^{t}-1$.

{\it A necessary condition} (see \ref{neccond}) that $\gamma_{(\alpha, D)}$ attains maximum (Definition \ref{def-3-2} (ii)) is the following (see Definition \ref{def-2-4} (ii) for $D^{(i)}$): $$\text{deg}(D^{(i)})=\text{deg}(D)=(n_{X}-1)n, \ i \in \{0, \dots, t-1\}.$$
Then we may assume that $D$ satisfies the above condition.

\subsubsection{Irreducible case} Suppose that $X^{\bp}$ is irreducible. Moreover, by Proposition \ref{maineasy}, we may assume $n_{X} \geq 2$.  First, if $(g_{X}, n_{X})=(0, 3)$, then the first main result follows from a result of I. Bouw (Lemma \ref{lem-3-3}). Next, for the case of an arbitrary $(g_{X}, n_{X})$, since $X^{\bp}$ is component-generic, we introduce certain degeneration data (i.e., degenerations of $X^{\bp}$) concerning $X^{\bp}$ (see \ref{degdata}) such that, for instance, the irreducible components of the degeneration data are either type of $(0,3)$ or $(1, 0)$. Then by applying  Lemma \ref{lem-3-3}, Proposition \ref{maineasy} (i), and Theorem \ref{lem-3-1}, we see that the first generalized Hasse-Witt invariants (\ref{ghw}) of Galois multi-admissible coverings of the degenerations of $X^{\bp}$ induced by $f_{\alpha}^{\bp}$ attain maximum. Moreover, the specialization homomorphisms of admissible fundamental groups implies that $\gamma_{(\alpha, D)}$ attains maximum. This completes the proof of Theorem \ref{main-them-1} (i) when $X^{\bp}$ is irreducible (Proposition \ref{pro-3-6}).

\subsubsection{General case} We have been seen that $\gamma_{(\alpha, D)}$ does not attain maximum for all $\alpha \in {\rm Rev}_{D}^{\rm adm}(X^{\bp})$ in general if $X^{\bp}$ is an arbitrary component-generic pointed stable curve (e.g. Example \ref{rem-prop-3-9}). To avoid such situation, we introduce a kind of semi-graph associated to the sets of marked points of pointed stable curves, which we call {\it a minimal quasi-tree} (\ref{quasitree}). Roughly speaking, a minimal quasi-tree is a minimal tree-like ``sub-semi-graph" of the dual semi-graph of a pointed stable curve which contains all open edges. We see that, if the dual semi-graph $\Gamma_{X^{\bp}}$ is a minimal quasi-tree, then the ramifications over nodes of a Galois multi-admissible covering can be completely determined by the ramifications over the marked points $D_{X}$. Thus, by applying Theorem \ref{lem-3-1}, we can reduce Theorem \ref{main-them-1} (i)  to the case where $X^{\bp}$ is irreducible. This completes Theorem \ref{main-them-1} (i) (Proposition \ref{prop-3-9}). Moreover, by gluing certain Galois multi-admissible coverings, Theorem \ref{main-them-1} (ii) follows from Theorem \ref{lem-3-1}, Proposition \ref{maineasy} (i), and Theorem \ref{main-them-1} (i).

\subsection{A necessary condition}\label{neccond} 

\subsubsection{\bf Settings}
We maintain the notation introduced in \ref{defcurves} and the assumption introduced in \ref{assum-2.3.3} (i.e.,  $n\defeq p^{t}-1$). Moreover, let $D \in (\mbZ/n\mbZ)^{\sim}[D_{X}]^{0}$, and let $${\rm ord}_{x}(D)=\sum_{j=0}^{t-1}d_{x, j}p^{j}, \ x \in D_{X}$$ be the $p$-adic expansion. For any $x \in D_{X}$, we put $d_{x}\defeq {\rm ord}_{x}(D)$ and $d_{x}^{(i)}\defeq (\text{ord}_{x}(D))^{(i)}$ (Definition \ref{def-2-4} (ii)).

\subsubsection{}


First, we have the following necessary condition that the first generalized Hasse-Witt invariants attain maximum.

\begin{lemma}\label{lem-3-4}
We maintain the notation introduced above. Suppose that $n_{X} \geq 2$. Then the following statements hold:

(i) Suppose that $s(D)=n_{X}-1$. Then  ${\rm deg}(D^{(i)}) \geq {\rm deg}(D)$  for all $i \in \{0, 1, \dots, t-1\}$ if and only if $$\sum_{x \in D_{X}}d_{x, j}=(n_{X}-1)(p-1), \ j \in \{0, \dots, t-1\}.$$ Note that ${\rm deg}(D^{(i)}) \geq {\rm deg}(D)=(n_{X}-1)n$ is equivalent to ${\rm deg}(D^{(i)}) = {\rm deg}(D)=(n_{X}-1)n$.

(ii) Suppose that there exists an element $\alpha \in {\rm Rev}_{D}^{\rm adm}(X^{\bp})$ such that $\gamma_{(\alpha, D)}$ attains maximum. Then we have ${\rm deg}(D^{(i)}) = {\rm deg}(D)=(n_{X}-1)n$  for all $i \in \{0, 1, \dots, t-1\}$.

(iii) Suppose that $n > n_{X}-1$. Then there exists $D' \in (\mbZ/n\mbZ)^{\sim}[D_{X}]^{0}$ such that $s(D')=n_{X}-1$, and that $${\rm deg}((D')^{(i)}) = {\rm deg}(D')=(n_{X}-1)n, \ i \in \{0, 1, \dots, t-1\}.$$
\end{lemma}

\begin{proof}
(i) The ``if" part of (i) is trivial. We only prove the ``only if" part of (i). Since ${\rm deg}(D^{(i)}) \geq {\rm deg}(D)$ for all $i \in \{0, \dots, t-1\}$ and $n |  {\rm deg}(D^{(i)})$, we have $${\rm deg}(D^{(i)})=\sum_{x \in D_{X}}(\text{ord}_{x}(D))^{(i)}=\sum_{x \in D_{X}}d_{x}^{(i)}=(n_{X}-1)n.$$  Moreover, for each $i \in \{0, \dots, t-1\}$, we have $$d^{(i+1)}_{x}=d_{x, i}p^{t-1}+\frac{d_{x}^{(i)}-d_{x, i}}{p}=\frac{1}{p}d^{(i)}_{x}+\frac{p^{t}-1}{p}d_{x, i}=\frac{1}{p}d^{(i)}_{x}+\frac{n}{p}d_{x, i}.$$ \tch{Thus, we obtain  $$(n_{X}-1)n=\sum_{x \in D_{X}}d_{x}^{(i+1)}=\frac{1}{p}\sum_{x \in D_{X}}d_{x}^{(i)}+\frac{n}{p}\sum_{x \in D_{X}}d_{x, i}$$$$=\frac{1}{p}(n_{X}-1)n+\frac{n}{p}\sum_{x \in D_{X}}d_{x, i}.$$} Then we have  $$\sum_{x \in D_{X}}d_{x, i}=(n_{X}-1)(p-1), \ i \in \{0, \dots, t-1\}.$$ 

(ii) This follows from Definition \ref{def-3-2} (ii) and Remark \ref{trans}.

(iii) Note that there exists $D'' \in (\mbZ/n\mbZ)^{\sim}[D_{X}]^{0}$ such that $s(D'')=(n_{X}-1)n$ if $n>n_{X}-1$. Then (iii) follows from (i). We complete the proof of the lemma.
\end{proof}



\subsubsection{} The following lemma will be used in the constructions of Galois admissible coverings of degenerations of $X^{\bp}$.

\begin{lemma}\label{lem-3-5}
We maintain the notation introduced above. Suppose that $n_{X} \geq 3$, and that  ${\rm deg}(D^{(i)}) = {\rm deg}(D)=(n_{X}-1)n$  for all $i \in \{0, 1, \dots, t-1\}$.   Moreover, we put $D_{X}\defeq \{x_{1}, \dots, x_{n_{X}}\}$ and 
$$a_{l, l+1} \defeq [\sum_{r=l+1}^{n_{X}}d_{x_{r}}], \ b_{l, l+1}\defeq [\sum_{r=1}^{l}d_{x_{r}}], \ l \in \{2, \dots, n_{X}-2\},$$ $$$$ where $[(-)]$ denotes the integer which is equal to the image of $(-)$ in $\mbZ/n\mbZ$ when we identify $\{0, \dots, n-1\}$ with $\mbZ/n\mbZ$ naturally.  Then, for each $i \in \{0, \dots, t-1\}$, we have $$\ a_{l, l+1}^{(i)}+b_{l, l+1}^{(i)}=n, \ l \in \{2, \dots, n_{X}-2\},$$
$$d_{x_{1}}^{(i)}+d_{x_{2}}^{(i)}+a_{2, 3}^{(i)}=2n,$$ $$b_{n_{X}-2, n_{X}-1}^{(i)}+d_{x_{n_{X}-1}}^{(i)}+d_{x_{n_{X}}}^{(i)}=2n,$$ $$b_{l, l+1}^{(i)}+d_{x_{l+1}}^{(i)}+a_{l+1, l+2}^{(i)}=2n, \ l \in \{2, \dots, n_{X}-3\}.$$
\end{lemma}

\begin{proof}
First, let us treat the first equality. Let $l \in \{2, \dots, n_{X}-2\}$ and $i \in \{0, \dots, t-1\}$. Since $$\sum_{r=l+1}^{n_{X}}d_{x_{r}}+\sum_{r=1}^{l}d_{x_{r}}=\text{deg}(D)=(n_{X}-1)n,$$ we see $n |  (a_{l, l+1}+b_{l, l+1})$. Note that if $a_{l, l+1}+b_{l, l+1}=0$, then $\text{deg}(D)<(n_{X}-1)n$. Thus, $a_{l, l+1}+b_{l, l+1} \neq 0$. Moreover, since $a_{l, l+1}+b_{l, l+1} \leq n$, we obtain that $a_{l, l+1}+b_{l, l+1}=n$, and that $a^{(i)}_{l, l+1}+b^{(i)}_{l, l+1}$ is divided by $n$. On the other hand, since $0<a^{(i)}_{l, l+1}, b^{(i)}_{l, l+1}< n$,  we have $a_{l, l+1}^{(i)}+b_{l, l+1}^{(i)}=n$. This completes the proof of the first equality. 

\tch{Let $i \in \{0, \dots, t-1\}$. We denote by $$S^{(i)}_{1} \defeq d^{(i)}_{x_{1}}+d^{(i)}_{x_{2}}+a^{(i)}_{2, 3},$$ $$S^{(i)}_{l} \defeq b^{(i)}_{l, l+1}+d^{(i)}_{x_{l+1}}+a^{(i)}_{l+1, l+2}, \ l\in \{2, \dots, n_{X}-3\},$$ $$S^{(i)}_{n_{X}-2}\defeq b^{(i)}_{n_{X}-2, n_{X}-1}+d^{(i)}_{x_{n_{X}-1}}+d^{(i)}_{x_{n_{X}}}.$$ We have  $S_{l}^{(i)}\leq 2n$, $\ l\in \{1, \dots, n_{X}-2\}$. Moreover, if $i=0$, the definitions of $a_{l, l+1}$ and $b_{l, l+1}$ imply  $$S_{l}^{(0)}=2n, \ l\in \{1, \dots, n_{X}-2\}.$$ Then we have that $S_{l}^{(i)}$ is divided by $n$ for all $i\in \{0, \dots, t-1\}$.  Since $s(D^{(i)})=n_{X}-1$, the first equality implies that  $$\sum_{l=1}^{n_{X}-2}S_{l}^{(i)}=n(n_{X}-3)+(n_{X}-1)n=2n(n_{X}-2).$$ On the other hand, since $S_{l}^{(i)}\leq 2n$, we have $$S_{l}^{(i)}=2n, \ l\in \{1, \dots, n_{X}-2\}, \ i\in \{0, \dots, t-1\}.$$ This completes the proof of the lemma.}
\end{proof}



\subsubsection{} The following lemma follows immediately from \cite[Corollary 6.8]{B}.

\begin{lemma}\label{lem-3-3}
We maintain the notation introduced above. Suppose that $X^{\bp}=(X, D_{X})$ is a smooth {\it component-generic} pointed stable curve of type $(g_{X}, n_{X})=(0, 3)$ over $k$. Moreover, we suppose that ${\rm deg}(D^{(i)}) = {\rm deg}(D)=(n_{X}-1)n=2n$ for all $i \in \{0, 1, \dots, t-1\}.$ 
Then the Raynaud-Tamagawa theta divisor $\Theta_{\mcE_{D}}$ associated to $\mcE_{D}$ exists (see Definition \ref{deftheta}). Moreover, we have (see \ref{ghwline} for $\gamma_{([\mcL], D)}$) $$\gamma_{([\mcL], D)}={\rm dim}_{k}(H^{1}(X, \mcL)), \ ([\mcL], D) \in \widetilde \msP_{X^{\bp}, n}.$$ 
\end{lemma}

\begin{remarkA}
Note that, if $n_{X}=3$, then we have $s(D) \in \{0, 1, 2\}$. 

\end{remarkA}

\subsection{Degenerations}\label{degdata}
In this subsection, we introduce certain degenerations of $X^{\bp}$ and prove the first main result in the case of  {\it irreducible} component-generic pointed stable curves (Proposition \ref{pro-3-6}). 


\subsubsection{\bf Settings} We maintain the notation introduced in \ref{defcurves} and the assumption introduced in \ref{assum-2.3.3} (i.e.,  $n\defeq p^{t}-1$). Moreover, we assume that  $X^{\bp}=(X, D_{X}\defeq \{x_{1}, \dots, x_{n_{X}}\})$ is an {\it irreducible} component-generic pointed stable curve over $k$.

\subsubsection{Degeneration data}\label{degdata}
We introduce some degeneration data for $X^{\bp}$. Let $R$ be a discrete valuation ring with algebraically closed residue field $k_{R}$, $K_{R}$ the quotient field of $R$, and $\overline K_{R}$ an algebraic closure of $K_{R}$. Suppose that $k \subseteq K_{R}$. Let $\mcX^{\bp}=(\mcX, D_{\mcX}\defeq \{e_{1}, \dots, e_{n_{X}}\})$ be a pointed stable curve of type $(g_{X}, n_{X})$ over $R$. We put $$\mcX^{\bp}_{\eta}=(\mcX_{\eta}, D_{\mcX_{\eta}}\defeq \{e_{\eta, 1}, \dots, e_{\eta, n_{X}}\}) \defeq \mcX^{\bp} \times_{R} K_{R},$$ $$ \mcX^{\bp}_{\overline \eta}=(\mcX_{\overline \eta}, D_{\mcX_{\overline \eta}}\defeq \{e_{\overline \eta, 1}, \dots, e_{\overline \eta, n_{X}}\}) \defeq \mcX^{\bp} \times_{R} \overline K_{R},$$ $$\mcX^{\bp}_{s}=(\mcX_{s}, D_{\mcX_{s}}\defeq \{e_{s, 1}, \dots, e_{s, n_{X}}\}) \defeq \mcX^{\bp} \times_{R} k_{R}.$$ 

We shall say that $X^{\bp}$ {\it admits a (DEG)} if $(g_{X}, n_{X}) \neq (1, 1)$ and there exists $\mcX^{\bp}$ such that the following conditions hold, where ``(DEG)" means ``degeneration": 

(i) We have that $\mcX^{\bp}_{\overline \eta}$  is $\overline K_{R}$-isomorphic to $X^{\bp} \times_{k} \overline K_{R}$, and that $\mcX_{s}^{\bp}$ is a {\it component-generic} pointed stable curve over $k_{R}$. Then without loss of generality, we may identify $e_{\overline \eta, r}, \ r \in \{1, \dots, n_{X}\}$, with  $x_{r}\times_{k}\overline K_{R}$ via this isomorphism.

(iii) Let $\msT$ be a set with cardinality $\#(\msT)=\#(X^{\rm sing})$ which consists of irreducible singular projective semi-stable curves of genus $1$ (i.e., the normalizations are $\mbP^{1}_{k_R}$). Let $\msP$ be a set which consists of smooth semi-stable curves of genus $0$ (i.e., a set of $\mbP^{1}_{k_R}$). Let $C_{1}$ be either an empty set or a smooth semi-stable curve of genus $g_{X}-\#(X^{\rm sing})$. We have the set of irreducible components of $\mcX_{s}$ is $$\{T\}_{T \in \msT} \cup \{C_{1}\} \cup \{P\}_{P \in \msP}.$$ Moreover, one of the following conditions is satisfied (see Example \ref{examdeg} below for examples of (a) (b) (c)):

\begin{quote}
(a) Suppose that  $n_{X}\leq 1$ and $\#(X^{\rm sing})\geq 1$. Then the following conditions hold: 

\begin{itemize}
\item $C_{1}=\emptyset$ when $n_{X}=0$ and $\#(X^{\rm sing})=2$; otherwise, $C_{1}\neq \emptyset$. 

\item When $C_{1}=\emptyset$, we have $\msT\defeq\{T_{1}, T_{2}\}$ such that $\#(T_{1} \cap T_{2})=1$. 

\item When $C_{1}\neq \emptyset$, we have that $T' \cap T'' \not\neq \emptyset$ if and only if $T'=T''$ for all $T', T'' \in \msT$, that $\#(T \cap C_{1})=1$ for all $T \in \msT$, and that $D_{\mcX_{s}} \subseteq C_{1}$. 

\item $\msP=\emptyset$.
\end{itemize}

\noindent
(b) Suppose that $n_{X}= 2$. Then the following conditions hold:  
\begin{itemize}
\item $T' \cap T'' \not\neq \emptyset$ if and only if $T'=T''$ for all $T', T'' \in \msT$. 

\item $C_{1}=\emptyset$ when $g_{X}-\#(X^{\rm sing})=0$; otherwise, $C_{1} \neq \emptyset$ when $g_{X}-\#(X^{\rm sing}) \geq 1$. 

\item $\msP\defeq \{P\}$ such that $D_{\mcX_{s}} \subseteq P$. 

\item When $C_{1}=\emptyset$, we have $\#(P \cap T)=1$ for all $T \in \msT$. 

\item When $C_{1}\neq \emptyset$, we have that $\#(C_{1}\cap T)=1$, that $\#(C_{1} \cap P)=1$, and that $P \cap T =\emptyset$ for all $T \in \msT$.
\end{itemize}

\noindent
(c) Suppose that $n_{X} \geq 3$. Then the following conditions hold: 

\begin{itemize}

\item The first and the second conditions of (b) hold. 

\item $\msP\defeq \{P_{v}\}_{v \in \{2, \dots, n_{X}-1\}}$ such that $D_{\mcX_{s}} \subseteq \bigcup_{v \in \{2, \dots, n_{X}-1\}}P_{v}$ . 

\item When $C_{1}=\emptyset$, we have $\#(P_{2} \cap T)=1$ and $P_{v} \cap T =\emptyset$ for all $T \in \msT$ and all $v \neq 2$.

\item When $C_{1}\neq \emptyset$, we have $\#(C_{1} \cap T)=1$, $\#(C_{1} \cap P_{2})=1$, $C_{1} \cap P_{v}=\emptyset $, and $P_{v} \cap T =\emptyset$ for all $T \in \msT$ and all $v \neq 2$. 

\item  If $n_{X} \geq 4$, for each $v \in \{2, \dots, n_{X}-2\}$, we have $\#(P_{v} \cap P_{v+1})=1$ and $P_{v} \cap P_{v'}=\emptyset$ when $v' \not\in \{v-1, v, v+1\}$. 

\item If $n_{X}=3$, we have $D_{\mcX_{s}} \cap P_{2}=\{e_{s, 1}, e_{s, 2}, e_{s, 3}\}$. 

\item If $n_{X}=4$, we have $D_{\mcX_{s}} \cap P_{2}=\{e_{s, 1}, e_{s, 2}\}$ and $D_{\mcX_{s}} \cap P_{3}=\{e_{s, 3}, e_{s, 4}\}$. 

\item If $n_{X} \geq 5$, we have $D_{\mcX_{s}} \cap P_{2}=\{e_{s, 1}, e_{s, 2}\}$, $D_{\mcX_{s}} \cap P_{n_{X}-1}=\{e_{s, n_{X}-1}, e_{s, n_{X}}\}$, and $D_{\mcX_{s}} \cap P_{v}=\{e_{s, v}\}$, $v\in \{3, \dots, n_{X}-2\}$.
\end{itemize}
\end{quote}

Note that since generic curves admit all degeneration types, we have that $X^{\bp}$ admits a (DEG) when $X^{\bp}$ is a component-generic pointed stable curve of type $(g_{X}, n_{X}) \neq (1,1)$. 

\subsubsection{}
Next, we give some examples to explain the degeneration data introduced above. For simplicity, we assume that $\#(X^{\rm sing})=2$, $C_{1}\neq \emptyset $, and $n_{X}\neq 0$.

\begin{example}\label{examdeg}
We use the notation ``$\bp$" and ``$\circ$ with a line segment" to denote a vertex and an open edge, respectively. Moreover, we use $v_{(-)}$ to denote the vertex corresponding to the irreducible component $(-)$.

(a) If $n_{X}=1$, then the dual semi-graph $\Gamma_{\mcX_{s}^{\bp}}$ of $\mcX_{s}^{\bp}$ is as follows (see \ref{degdata} (iii)-(a)):

\begin{picture}(300,100)
\put(170,60){\circle*{5}}
\put(145,60){$v_{T_{1}}$}
\put(170,70){\circle{20}}

\put(170,40){\circle*{5}}
\put(170,60){\line(0,-10){20}}
\put(145,40){$v_{C_{1}}$}

\put(170,20){\circle*{5}}
\put(170,10){\circle{20}}
\put(170,20){\line(0,10){20}}
\put(145,20){$v_{T_{2}}$}

\put(170,40){\line(10,0){20}}
\put(192,40){\circle{4}}

\put(90,59){$\Gamma_{\mcX^{\bp}_{s}}$:}
\end{picture}

(b) If $n_{X}=2$, then the dual semi-graph $\Gamma_{\mcX_{s}^{\bp}}$ of $\mcX_{s}^{\bp}$ is as follows (see \ref{degdata} (iii)-(b)):

\begin{picture}(300,100)
\put(170,60){\circle*{5}}
\put(145,60){$v_{T_{1}}$}
\put(170,70){\circle{20}}

\put(170,40){\circle*{5}}
\put(170,60){\line(0,-10){20}}
\put(145,40){$v_{C_{1}}$}

\put(170,20){\circle*{5}}
\put(170,10){\circle{20}}
\put(170,20){\line(0,10){20}}
\put(145,20){$v_{T_{2}}$}

\put(170,40){\line(10,0){20}}
\put(192,40){\circle*{5}}

\put(192,40){\line(0,10){15}}
\put(192,40){\circle*{5}}
\put(192,40){\line(0,-10){15}}
\put(200,40){$v_{P}$}

\put(192,57){\circle{4}}
\put(192,23){\circle{4}}

\put(90,59){$\Gamma_{\mcX^{\bp}_{s}}$:}
\end{picture}

(c) If $n_{X}=5$, then the dual semi-graph $\Gamma_{\mcX_{s}^{\bp}}$ of $\mcX_{s}^{\bp}$ is as follows (see \ref{degdata} (iii)-(c)):

\begin{picture}(300,100)
\put(170,60){\circle*{5}}
\put(145,60){$v_{T_{1}}$}
\put(170,70){\circle{20}}

\put(170,40){\circle*{5}}
\put(170,60){\line(0,-10){20}}
\put(145,40){$v_{C_{1}}$}

\put(170,20){\circle*{5}}
\put(170,10){\circle{20}}
\put(170,20){\line(0,10){20}}
\put(145,20){$v_{T_{2}}$}

\put(170,40){\line(10,0){20}}
\put(192,40){\circle*{5}}

\put(192,40){\line(0,10){15}}
\put(192,40){\circle*{5}}
\put(192,40){\line(0,-10){15}}
\put(195,45){$v_{P_{2}}$}

\put(192,57){\circle{4}}
\put(192,23){\circle{4}}

\put(192,40){\line(10, 0){20}}
\put(212,40){\circle*{5}}
\put(212,40){\line(0,10){15}}
\put(215,45){$v_{P_{3}}$}
\put(212,57){\circle{4}}

\put(214,40){\line(10, 0){20}}
\put(232,40){\circle*{5}}
\put(232,40){\line(0,10){15}}
\put(232,40){\line(0,-10){15}}
\put(235,45){$v_{P_{4}}$}
\put(232,57){\circle{4}}
\put(232,23){\circle{4}}

\put(90,59){$\Gamma_{\mcX^{\bp}_{s}}$:}
\end{picture}

\end{example}

\subsubsection{}
The main result of the present subsection is as follows:

\begin{proposition}\label{pro-3-6}
Let $X^{\bp}=(X, D_{X}\defeq \{x_{1}, \dots, x_{n_{X}}\})$ be an irreducible component-generic pointed stable curve over $k$, $D \in (\mbZ/n\mbZ)^{\sim}[D_{X}]^{0}$, and $\alpha \in {\rm Rev}_{D}^{\rm adm}(X^{\bp}) \setminus \{0\}$. Suppose that ${\rm deg}(D)=(n_{X}-1)n$ (i.e., $s(D)=n_{X}-1$) if $n_{X}\neq 0$, and that ${\rm deg}(D^{(i)}) = {\rm deg}(D)=(n_{X}-1)n$ for all $i \in \{0, 1, \dots, t-1\}.$ Then $\gamma_{(\alpha, D)}$ attains maximum. Namely, the following holds:
\begin{eqnarray*}
\gamma_{(\alpha, D)}=\gamma^{\rm max}_{X^{\bp}}=\left\{ \begin{array}{ll}
g_{X}-1, & \text{if} \ n_{X} =0,
\\
g_{X}+n_{X}-2, & \text{if} \ n_{X} \neq 0.
\end{array} \right.
\end{eqnarray*}
\end{proposition}

\begin{proof}
Let $f^{\bp}: Y^{\bp}=(Y, D_{Y}) \migi X^{\bp}$ be the Galois multi-admissible covering over $k$ with Galois group $\mbZ/n\mbZ$ induced by $\alpha$. Firstly, we note that, to verify the proposition, we may assume that $Y^{\bp}$ is connected.

Suppose that $X^{\bp}$ is smooth over $k$, and that $n_{X}\leq 1$. Then the proposition follows immediately from Proposition \ref{maineasy} (i). Thus, to verify the proposition, it is sufficient to assume that one of the following conditions holds: (1) $\#(X^{\rm sing}) \geq 1$ and $n_{X}\leq 1$; (2) $n_{X}\geq 2$.

Suppose that $X^{\bp}$ is a singular curve of type $(1,1)$ (i.e., $\#(X^{\rm sing})=1$). Since $f^{\bp}$ is \'etale, the proposition follows immediately from Lemma \ref{topcov} (i). Thus, to verify the proposition, we may assume  $(g_{X}, n_{X})\neq (1, 1)$. Now, we can use the degeneration data introduced in \ref{degdata}.

Since $X^{\bp}$ is a component-generic pointed stable curve, $X^{\bp}$ admits a (DEG). We maintain the notation introduced in \ref{degdata}. Furthermore, we write $Q_{\overline \eta}$ (resp. $Q_{s}$) for the effective divisor on $\mcX_{\overline \eta}$ (resp. $\mcX_{s}$) induced by $D$ and $\alpha_{\overline \eta} \in \text{Rev}^{\rm adm}_{Q_{\overline \eta}}(\mcX^{\bp}_{\overline \eta})$ for the element induced by $\alpha$. Then we have $\gamma_{(\alpha, D)}=\gamma_{(\alpha_{\overline \eta}, Q_{\overline \eta})}.$ Write $\Pi_{\mcX_{\overline \eta}^{\bp}}$ and $\Pi_{\mcX_{s}^{\bp}}$ for the admissible fundamental groups of $\mcX^{\bp}_{\overline \eta}$ and $\mcX_{s}^{\bp}$, respectively. Then we have a specialization surjective homomorphism $$sp_{R}: \Pi_{\mcX_{\overline \eta}^{\bp}} \migisurj \Pi_{\mcX_{s}^{\bp}}.$$

We suppose that $X^{\bp}$ satisfies (DEG)-(iii)-(c) (\ref{degdata}). Moreover, we suppose that $C_{1}\neq \emptyset$ and $n_{X}\geq 5$ (see Example \ref{examdeg} (c)). \\

{\bf Step 1:} We define certain smooth pointed stable curves associated to irreducible components of $\mcX^{\bp}_{s}$. \\

We write $y_{v, v+1}, \ z_{v, v+1}$, $v \in \{2, \dots, n_{X}-2\}$, for the inverse image of $P_{v} \cap P_{v+1}$ of the natural closed immersion $P_{v} \migiinje \mcX_{s}$ and the inverse image of  $P_{v} \cap P_{v+1}$ of the natural closed immersion $P_{v+1} \migiinje \mcX_{s}$, respectively. We define
$$P_{2}^{\bp}=(P_{2}, D_{P_{2}}\defeq \{e_{s, 1}, e_{s, 2}, y_{2, 3}\} \cup (C_{1}\cap P_{2})),$$ $$P_{n_{X}-1}^{\bp}=(P_{n_{X}-1}, D_{P_{n_{X}-1}} \defeq \{z_{n_{X}-2, n_{X}-1}, e_{s, n_{X}-1}, e_{s, n_{X}}\}),$$ $$P_{v}^{\bp}=(P_{v}, D_{P_{v}}\defeq \{z_{v-1, v}, e_{s, v}, y_{v, v+1}\}), \ v \in \{3, \dots, n_{X}-2\},$$ to be smooth pointed stable curves of types $(0, 4)$, $(0, 3)$, and $(0,3)$ over $k_{R}$, respectively. Moreover, we define $$C_{1}^{\bp}= (C_{1}, D_{C_{1}}\defeq (C_{1} \cap P_{2})\cup ((\bigcup_{T \in \msT}T) \cap C_{1})),$$ $$T^{\bp} =(T, D_{T}\defeq T \cap C_{1}), \ T \in \msT,$$ to be a smooth pointed stable curve of type $(g_{X}-\#(X^{\rm sing}), 1+\#(X^{\rm sing}))$ and a singular pointed stable curve of type $(1, 1)$ over $k_{R}$, respectively. Note that since $C_{1}$ is generic, we have $\sigma_{C_{1}}=g_{X}-\#(X^{\rm sing})$ (i.e., $C_{1}^{\bp}$ is ordinary (\ref{def-1-3})). \\

{\bf Step 2:} We construct a Galois admissible covering of $\mcX_{s}^{\bp}$ by using specialization isomorphisms of maximal prime-to-$p$ quotients of admissible fundamental groups. \\

Let $f^{\bp}_{\overline \eta}\defeq f^{\bp}\times_{k}\overline K_{R}: \mcY_{\overline \eta}^{\bp} =(\mcY_{\overline \eta}, D_{\mcY_{\overline \eta}})\defeq Y^{\bp}\times_{k} \overline K_{R} \migi \mcX_{\overline \eta}^{\bp}$ be the Galois admissible covering over $\overline K_{R}$ with Galois group $\mbZ/n\mbZ$ induced by $f^{\bp}$, and $\Pi_{\mcY_{\overline \eta}^{\bp}} \subseteq \Pi_{\mcX_{\overline \eta}^{\bp}} \cong \Pi_{X^{\bp}}$ the admissible fundamental group of $\mcY_{\overline \eta}^{\bp}$. By the specialization theorem of  maximal prime-to-$p$ quotients of admissible fundamental groups (\cite[Th\'eor\`eme 2.2 (c)]{V}), we have  $$sp^{p'}_{R}: \Pi_{\mcX^{\bp}_{\overline \eta}}^{p'} \isom \Pi_{\mcX^{\bp}_{s}}^{p'}.$$ Then we obtain a normal open subgroup $\Pi_{\mcY_{s}^{\bp}}^{p'}\defeq sp^{p'}_{R}(\Pi_{\mcY_{\overline \eta}^{\bp}}^{p'}) \subseteq \Pi_{\mcX^{\bp}_{s}}^{p'}$. Write $\Pi_{\mcY_{s}^{\bp}} \subseteq \Pi_{\mcX_{s}^{\bp}}$ for the inverse image of $\Pi_{\mcY_{s}^{\bp}}^{p'}$ of the natural surjection $\Pi_{\mcX_{s}^{\bp}} \migisurj \Pi_{\mcX_{s}^{\bp}}^{p'}$. Then $\Pi_{\mcY_{s}^{\bp}}$ determines a Galois admissible covering $$f^{\bp}_{s}: \mcY_{s}^{\bp}= (\mcY_{s}, D_{\mcY_{s}}) \migi \mcX_{s}^{\bp}$$ over $k_{R}$ with Galois group $\mbZ/n\mbZ$. We denote by $\alpha_{s} \in \text{Rev}_{Q_{s}}^{\rm adm}(\mcX_{s}^{\bp})$ the element induced by the composition homomorphism $\Pi_{\mcX^{\bp}_{s}}^{p', \rm ab} \overset{(sp_{R}^{p', \rm ab})^{-1}}\isom \Pi_{\mcX^{\bp}_{\overline \eta}}^{p', \rm ab} \overset{\alpha_{\overline \eta}}\migi \mbZ/n\mbZ$. \\

{\bf Step 3:} We compute the generalized Hasse-Witt invariant $\gamma_{(\alpha_{s}, D_{\alpha_{s}})}$ by applying Theorem \ref{lem-3-1}, Lemma \ref{lem-3-5}, and Lemma \ref{lem-3-3}. \\

The structure of $\Pi_{\mcX_{s}^{\bp}}^{p'}$ (\ref{stradm}) implies that $f_{s}$ is \'etale over $(\bigcup_{T \in \msT}T) \cap C_{1}$. Then we obtain that $f_{s}$ is \'etale over $C_{1} \cap P_{2}$. Thus, $f_{s}$ is \'etale over $D_{C_{1}}$. Let $Y_{v}\defeq f^{-1}_{s}(P_{v})$, $v \in \{2, \dots, n_{X}-1\}$. We put $$Y_{v}^{\bp}\defeq (Y_{v}, D_{Y_{v}}\defeq f^{-1}_{s}(D_{P_{v}})), \ v\in \{2, \dots, n_{X}-1\}.$$ Then $f^{\bp}_{s}$ induces a Galois multi-admissible covering $$f_{v}^{\bp}: Y^{\bp}_{v} \migi P_{v}^{\bp},\ v\in \{2, \dots, n_{X}-1\},$$ over $k_{R}$ with Galois group $\mbZ/n\mbZ$. We maintain the notation introduced in Lemma \ref{lem-3-5}. Then we see that the ramification divisor on $P_{v}$, $v\in \{2, \dots, n_{X}-1\}$, induced by $f_{v}^{\bp}$ and $\alpha_{s}$ is as follows: $$Q_{2}\defeq d_{x_{1}}e_{s, 1}+d_{x_{2}}e_{s,2}+a_{2,3}y_{2,3},$$ $$Q_{n_{X}-1}\defeq b_{n_{X}-2, n_{X}-1}z_{n_{X}-2, n_{X}-1}+d_{x_{n_{X}-1}}e_{s, n_{X}-1}+d_{x_{n_{X}}}e_{s, n_{X}},$$ $$Q_{v}\defeq b_{v-1, v}z_{v-1, v}+d_{x_{v}}e_{s, v}+a_{v, v+1}y_{v, v+1}, \ v\in\{3, \dots, n_{X}-2\}.$$ Since $f_{s}$ is \'etale over $C_{1} \cap P_{2}$, we see that $f_{v}^{\bp}, \ v\in \{2, \dots, n_{X}-1\},$ induces a pair $([\mcL_{v}], Q_{v}) \in \widetilde \msP_{P_{v}^{\bp}, n}$. Moreover, the $k_{R}[\mu_{n}]$-module $H^{1}_{\text{\'et}}(Y_{v}, \mbF_{p})\otimes k_{R}$ admits the following canonical decomposition $$H^{1}_{\text{\'et}}(Y_{v}, \mbF_{p})\otimes k_{R}=\bigoplus_{j \in \mbZ/n\mbZ} M_{Y_v}(j),$$ where $\zeta \in \mu_{n}$ acts on $M_{Y_v}(j)$ as the $\zeta^{j}$-multiplication. On the other hand, lemma \ref{lem-3-5} implies $\text{deg}(Q_{v}^{(i)})=\text{deg}(Q_{v})=2n$, $i\in \{0, \dots, t-1\}$. Then by applying Lemma \ref{lem-3-3}, we have  $$\gamma_{([\mcL_{v}], Q_{v})}=\text{dim}_{k_{R}}(M_{Y_v}(1))=\text{dim}_{k_{R}}(H^{1}(P_{v}, \mcL_{v}))=1.$$

Let $Z_{1}\defeq f^{-1}_{s}(C_{1})$. Then $f^{\bp}_{s}$ induces a Galois \'etale covering (not necessarily connected) $$f^{\bp}_{C_{1}}: Z^{\bp}_{1}=(Z_{1}, D_{Z_{1}}\defeq f^{-1}_{s}(D_{C_{1}})) \migi C_{1}^{\bp}$$ over $k_{R}$ with Galois group $\mbZ/n\mbZ$. Write $\alpha_{C_{1}} \in \text{Rev}_{0}^{\rm adm}(C_{1}^{\bp})$ for the element induced by $f_{C_{1}}^{\bp}$  and $\alpha_{s}$. The $k_{R}[\mu_{n}]$-module $H^{1}_{\text{\'et}}(Z_{1}, \mbF_{p})\otimes k_{R}$ admits the following canonical decomposition $$H^{1}_{\text{\'et}}(Z_{1}, \mbF_{p})\otimes k_{R}=\bigoplus_{j \in \mbZ/n\mbZ} M_{Z_1}(j),$$ where $\zeta \in \mu_{n}$ acts on $M_{Z_{1}}(j)$ as the $\zeta^{j}$-multiplication. Then Proposition \ref{maineasy} (i) implies \begin{eqnarray*}
\gamma_{(\alpha_{C_{1}}, 0)}= \left\{ \begin{array}{ll}
g_{X}-\#(X^{\rm sing}), & \text{if} \ \alpha_{C_{1}}  =0,
\\
g_{X}-\#(X^{\rm sing})-1, & \text{if} \ \alpha_{C_{1}}  \neq 0.
\end{array} \right.
\end{eqnarray*}

Let $V_{T} \defeq f^{-1}_{s}(T)$, $T\in \msT$, and $\widetilde T$ the smooth compactification of $U_{T} \defeq T \setminus T^{\rm sing}$. Then $f^{\bp}_{s}$ induces a Galois multi-admissible covering $$f_{T}^{\bp}: V_{T}^{\bp}=(V_{T}, D_{V_{T}}\defeq f_{s}^{-1}(D_{T})) \migi T^{\bp}$$ over $k_{R}$ with Galois group $\mbZ/n\mbZ$. Then $f_{T}^{\bp}$ induces a Galois multi-admissible covering $$f_{\widetilde T}^{\bp}: V^{\bp}_{\widetilde T}=(V_{\widetilde T}, D_{V_{\widetilde T}}) \migi \widetilde T^{\bp}$$ over $k_{R}$.  Write $\alpha_{\widetilde T} \in \text{Rev}_{0}(\widetilde T^{\bp})$ for the element induced by $f_{\widetilde T}^{\bp}$ and $\alpha_{s}$. By using the proposition of the case where $X^{\bp}$ is a singular curve of type $(g_{X}, n_{X})=(1, 1)$, we obtain that $\gamma_{(\alpha_{\widetilde T}, 0)}=0$. Thus, Theorem \ref{lem-3-1} implies that $$\gamma_{(\alpha_{s}, Q_{s})}=g_{X}+n_{X}-2.$$ \\

{\bf Step 4:} We compute $\gamma_{(\alpha_{\overline \eta}, Q_{\overline \eta})}$ by using specialization surjective homomorphisms of admissible fundamental groups. \\

The $k_{R}[\mu_{n}]$-modules $H^{1}_{\rm \text{\'et}}(\mcY_{\overline \eta}, \mbF_{p}) \otimes k_{R}$ and $H^{1}_{\rm \text{\'et}}(\mcY_{s}, \mbF_{p}) \otimes k_{R}$ admit the following canonical decompositions $$H^{1}_{\rm \text{\'et}}(\mcY_{\overline \eta}, \mbF_{p}) \otimes k_{R}=\bigoplus_{j \in \mbZ/n\mbZ}M_{\mcY_{\overline \eta}}(j),$$ $$H^{1}_{\rm \text{\'et}}(\mcY_{s}, \mbF_{p}) \otimes k_{R}=\bigoplus_{j \in \mbZ/n\mbZ}M_{\mcY_{s}}(j),$$ respectively. Moreover, we have an injection as $k_{R}[\mu_{n}]$-modules $$H^{1}_{\rm \text{\'et}}(\mcY_{s}, \mbF_{p}) \otimes k_{R} \migiinje H^{1}_{\rm \text{\'et}}(\mcY_{\overline \eta}, \mbF_{p}) \otimes k_{R}$$ induced by the specialization map $\Pi_{\mcY^{\bp}_{\overline \eta}} \migisurj \Pi_{\mcY_{s}^{\bp}}$. Thus, we have $$g_{X}+n_{X}-2 = \gamma_{(\alpha_{s}, Q_{s})} =\text{dim}_{k_{R}}(M_{\mcY_{s}}(1))$$ $$\leq \gamma_{(\alpha_{\overline \eta}, Q_{\overline \eta})}=\text{dim}_{k_{R}}(M_{\mcY_{\overline \eta}}(1)) \leq g_{X}+n_{X}-2.$$ This means $\gamma_{(\alpha_{\overline \eta}, Q_{\overline \eta})} = g_{X}+n_{X}-2.$ We complete the proof of the proposition when $X^{\bp}$ satisfies (DEG)-(iii)-(c), $C_{1} \neq \emptyset$, and $n_{X}\geq 5$.

By applying similar arguments to the arguments given in the proof above, one can prove the proposition when $X^{\bp}$ satisfies (DEG)-(iii)-(c) and either $C_{1} = \emptyset$ or $n_{X}\leq 4$ holds.  Moreover, similar arguments to the arguments given in the proof above imply that the proposition holds when $X^{\bp}$ satisfies either (DEG)-(iii)-(a) or (DEG)-(iii)-(b). We complete the proof of the proposition.
\end{proof}

In the remainder of this section, under certain assumptions, we generalize Proposition \ref{pro-3-6} to the case where $X^{\bp}$ is not necessary to be irreducible. 

 either (DEG)-(v) or (DEG)-(vi). We complete the proof of the proposition.

\subsubsection{} By applying Proposition \ref{pro-3-6}, we solve Problem \ref{pbl-3-2} under certain assumptions as follows:

\begin{corollary}\label{coro-3-10}
Let $X^{\bp}=(X, D_{X})$ be a smooth component-generic pointed stable curve over $k$ and $D\in (\mbZ/n\mbZ)^{\sim}[D_{X}]^{0}$. Suppose that $s(D)=n_{X}-1$ if $n_{X}\neq 0$, and that ${\rm deg}(D^{(i)}) \geq {\rm deg}(D), \ i \in \{0, 1, \dots, t-1\}.$ Then the Raynaud-Tamagawa theta divisor $\Theta_{\mcE_{D}}$ associated to $\mcE_{D}$ exists.
\end{corollary}

\begin{proof}
Since $X^{\bp}$ is smooth over $k$, the corollary follows immediately from Proposition \ref{pro-3-6} and Remark \ref{new-rem-RTd}.
\end{proof}


\subsection{Minimal quasi-trees}\label{quasitree}
In this subsection, we introduce the so-called {\it minimal quasi-trees} which play an important role in the remainder of the present paper. 

\subsubsection{\bf Settings} Let $W^{\bp}=(W, D_{W})$ be a pointed stable curve of type $(g_{W}, n_{W})$ over $k$ and $\Gamma_{W^{\bp}}$ the dual semi-graph of $W^{\bp}$.

\subsubsection{}\label{4.4-2}
Let $Z^{\bp}$ be a pointed sable curve over $k$,  $\Gamma_{Z^{\bp}}$ the dual semi-graph of $Z^{\bp}$ such that $\Gamma_{Z^{\bp}} \setminus e^{\rm lp}(\Gamma_{Z^{\bp}})$ is a {\it tree} (see \ref{graph} for $e^{\rm lp}(\Gamma_{Z^{\bp}})$), and that $E_{Z} \subseteq e^{\rm op}(\Gamma_{Z^{\bp}})$ is a subset of open edges.

We put $$V_{Z}\defeq \{v\in v(\Gamma_{Z^{\bp}}) \ | \ \#(e^{\Gamma_{Z^{\bp}}}(v) \cap (e^{\rm cl}(\Gamma_{Z^{\bp}}) \setminus e^{\rm lp}(\Gamma_{Z^{\bp}})))=1, \ Z_{v} \cap E_{Z} =\emptyset \} \subseteq v(\Gamma_{Z^{\bp}}),$$ where $Z_{v}$ denotes the irreducible component corresponding to $v$. We call $V_{Z}$ {\it the set of terminal vertices avoiding to $E_{Z}$}. 

For instance, in Example \ref{rem-def-3-7} (b) below, we have $Z^{\bp}\defeq W_{1}^{\bp}$, $E_{Z}\defeq e^{\rm op}(\Gamma_{Z^{\bp}})=e^{\rm op}(\Gamma_{W_{1}^{\bp}})$, and $V_{Z}=\{v_{3}\}$.

\subsubsection{Constructions of minimal quasi-trees associated to $D_{W}$}\label{defquasitree}

Next, we define a minimal quasi-tree associated to $D_{W}$. Let $E \subseteq e^{\rm cl}(\Gamma_{W^{\bp}}) \setminus e^{\rm lp}(\Gamma_{W^{\bp}})$ be a (possibly empty) subset of closed edges such that $\Gamma_{W^{\bp}} \setminus (E\cup e^{\rm lp}(\Gamma_{W^{\bp}}))$ is a connected tree. Note that it is easy to see that $E$ exists. 

Write $N$ for the set of nodes of $W$ corresponding to the closed edges which are contained in $E$, and write $\text{norm}_{N}: W_{N} \migi W$ for the normalization morphism of the underlying curve $W$ over $N$. We define a pointed stable curve over $k$ to be $$W_{1}^{\bp}=(W_{1}, D_{W_{1}})\defeq W_{N}^{\bp} =(W_{N}, D_{W_{N}}\defeq \text{norm}_{N}^{-1}(D_{W} \cup N)).$$ Note that the above construction implies $D_{W} \subseteq D_{W_{1}}$ (i.e., $e^{\rm op}(\Gamma_{W^{\bp}}) \subseteq e^{\rm op}(\Gamma_{W_{1}^{\bp}})$). Write $\Gamma_{W_{1}^{\bp}}$ for the dual semi-graph of $W_{1}^{\bp}$. Then the construction of $W_{1}^{\bp}$ implies that $\Gamma_{W_{1}^{\bp}} \setminus e^{\rm lp}(\Gamma_{W_{1}^{\bp}})$ is a {\it tree}. \\


{\it Suppose that $n_{W} \neq 0$}. Then for all $i\in \mbN$, if $W_{i}^{\bp}$ has already been defined, we may define $W_{i+1}^{\bp}$ as follows: Let $V_{i} \subseteq v(\Gamma_{W_{i}^{\bp}})$ be the subset of terminal vertices avoiding to $e^{\rm op}(\Gamma_{W^{\bp}})\subseteq e^{\rm op}(\Gamma_{W_{i}^{\bp}})$. If $V_{i}=\emptyset$, then we put $W^{\bp}_{i+1} \defeq W^{\bp}_{i}$. If $V_{i}\neq \emptyset$, we write $W_{i+1}$ for the topological closure of $$W_{i} \setminus (\bigcup_{v\in V_{i}}W_{v})$$ in $W_{i}$. Note that the definition of $V_{i}$ implies that $W_{i+1}$ is connected, and that $D_{W}$ is contained in $W_{i+1}$. Then we define a pointed stable curve over $k$ to be $$W_{i+1}^{\bp} \defeq (W_{i+1}, D_{W_{i+1}} \defeq (D_{W_{i}} \cap W_{i+1}) \cup ((\bigcup_{v\in V_{i}}W_{v}) \cap W_{i+1}))).$$ Note that we have $e^{\rm op}(\Gamma_{W^{\bp}}) \subseteq e^{\rm op}(\Gamma_{W_{i+1}^{\bp}})$.

Let $i_{0}$ be the minimal natural number such that $V_{i_{0}}=\emptyset$. Note that the above construction implies that $W_{i_{0}}^{\bp}=W_{j}^{\bp}$ for all $j\geq i_{0}$. We put $$W^{\bp}_{\Gamma}=(W_{\Gamma}, D_{W_{\Gamma}})\defeq W_{i_{0}}^{\bp}, \ \Gamma\defeq \Gamma_{W_{i_{0}}^{\bp}}.$$  Then we have that $D_{W} \subseteq D_{W_{\Gamma}}$, and that $\Gamma\setminus e^{\rm lp}(\Gamma)$ is a {\it tree}. Moreover, we see $v(\Gamma) \subseteq v(\Gamma_{W^{\bp}})$ and $e(\Gamma) \subseteq e(\Gamma_{W^{\bp}})$, where $e(-)$ denotes the set of edges of $(-)$ (\ref{graph}). Note that an open edge of $\Gamma$ is {\it not} an open edge of $\Gamma_{W^{\bp}}$ in general, and that a closed edge of $\Gamma$ is a closed edge of $\Gamma_{W^{\bp}}$. \\

Moreover, the above constructions imply the following:

\begin{lemma}\label{lem-3-8}
Let $W^{\bp}=(W, D_{W})$ be a pointed stable curve of type $(g_{W}, n_{W})$ over $k$. Suppose that $n_{W} \neq 0$. We denote by $\msQ_{W^{\bp}}$ the set of $\Gamma$ constructed as above. Then we have $\msQ_{W^{\bp}}\neq \emptyset.$
\end{lemma}

Note that  the construction of $\Gamma$ {\it depends} on the choice of $E$ (i.e., a subset of closed edges of $e^{\rm cl}(\Gamma_{W^{\bp}}) \setminus e^{\rm lp}(\Gamma_{W^{\bp}})$ such that $\Gamma_{W^{\bp}} \setminus (E \cup e^{\rm lp}(\Gamma_{W^{\bp}}))$ is a tree). 

\subsubsection*{}
\begin{definition}\label{def-3-7}
We maintain the notation introduced above. We shall say that a semi-graph $$\Gamma_{D_{W}}$$ is {\it a minimal quasi-tree associated to $D_{W}$} if $\Gamma_{D_{W}}=\Gamma$ for some $\Gamma$ constructed above when $n_{W}\neq 0$, and $\Gamma_{D_{W}}=\emptyset$ when $n_{W}=0$.
\end{definition}

\subsubsection{}
We maintain the notation introduced in \ref{defquasitree}. Let us give some examples to explain the above constructions.

\begin{example}\label{rem-def-3-7}
(a) Let $W^{\bp}$ be a pointed stable curve over $k$ such that the following conditions hold: (i) The set of irreducible components of $W$ is $\{W_{v_{1}}, W_{v_{2}}, W_{v_{3}}\}$; (ii) $D_{W}=\{w_{b_{1}}, w_{b_{2}}\}$; (iii) The set of nodes is $\{w_{c}, w_{a_{1}}, w_{a_{2}}, w_{a_{3}}\}$; (iv) $W_{v_1}$ is a singular curve with the unique node $w_{c}$; (v) $w_{b_{1}} \in W_{v_{1}}$ and $w_{b_{2}} \in W_{v_{2}}$; (vi) $w_{a_{1}}, w_{a_{2}} \in W_{v_{1}}\cap W_{v_{2}}$; (vii) $w_{a_{3}} \in W_{v_{2}} \cap W_{v_{3}}$. We use the notation ``$\bp$" and ``$\circ$" to denote a node and a marked point, respectively. Then $W^{\bp}$ is as follows:

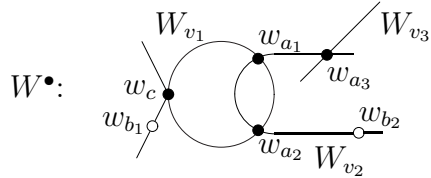
\begin{picture}(300,100)

\put(170,60){\circle{40}}
\put(150,60){\circle*{4}}
\put(150,60){\line(-5,10){10}}
\put(150,60){\line(-5,-10){5}}
\put(144,48){\circle{4}}
\put(143,46){\line(-5,-10){5}}

\put(220,60){\oval(90,30)[l]}
\put(184,73.5){\circle*{4}}
\put(184,46.5){\circle*{4}}
\put(222,45.5){\circle{4}}
\put(224,45.5){\line(5,0){7}}

\put(200,65){\line(5,5){30}}
\put(210,75){\circle*{4}}

\put(145,83){$W_{v_{1}}$}
\put(230,83){$W_{v_{3}}$}
\put(205,32){$W_{v_{2}}$}

\put(133,60){$w_{c}$}
\put(125,48){$w_{b_{1}}$}
\put(222,50){$w_{b_{2}}$}
\put(184,79){$w_{a_{1}}$}
\put(184,38){$w_{a_{2}}$}
\put(210,65){$w_{a_{3}}$}

\put(90,59){$W^{\bp}$:}
\end{picture}

The dual semi-graph $\Gamma_{W^{\bp}}$  of $W^{\bp}$ such that the following conditions hold:
(i) $v(\Gamma_{W^{\bp}}) \defeq \{v_{1}, v_{2}, v_{3}\}$;
(ii) $e^{\rm cl}(\Gamma_{W^{\bp}}) \setminus e^{\rm lp}(\Gamma_{W^{\bp}})\defeq \{a_{1}, a_{2}, a_{3}\}$ such that $a_{1}$ and $a_{2}$ abut to $v_{1}$ and $v_{2}$, respectively, and that $a_{3}$ abuts to $v_{2}$ and $v_{3}$; (iii) $e^{\rm lp}(\Gamma_{W^{\bp}})\defeq \{c\}$ and $c$ abuts to $v_{1}$; (iv) $e^{\rm op}(\Gamma_{W^{\bp}}) \defeq \{b_{1}, b_{2}\}$ such that $b_{1}$ and $b_{2}$ abut to $v_{1}$ and $v_{2}$, respectively. We use the notation $``\bp"$ and`` $\circ$ with a line segment" to denote a vertex and an open edge, respectively. Then $\Gamma_{W^{\bp}}$ is as follows:

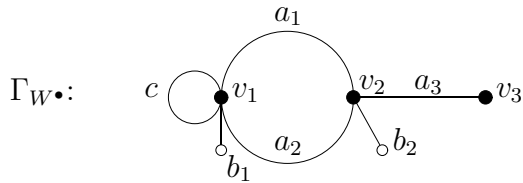
\begin{picture}(300,100)
\put(170,60){\circle*{5} $v_{1}$}
\put(195,60){\circle{50}}
\put(190,89){$a_{1}$}
\put(190,40){$a_{2}$}
\put(160,60){\circle{20}}
\put(141,60){$c$}
\put(170,60){\line(0,-9){18}}
\put(170,40){\circle{4}}
\put(172,30){$b_{1}$}
\put(220,60){\circle*{5}}
\put(222,63){$v_{2}$}
\put(220,60){\line(10,0){50}}
\put(243, 63){$a_{3}$}
\put(270,60){\circle*{5} $v_{3}$}
\put(220,60){\line(5,-9){10}}
\put(231,40){\circle{4} $b_{2}$}
\put(90,59){$\Gamma_{W^{\bp}}$:}
\end{picture}

(b) Let $E\defeq \{a_{1}\}$. Then we see that the dual semi-graph $\Gamma_{W_{1}^{\bp}}$ of $W^{\bp}_{1}$ is as follows:

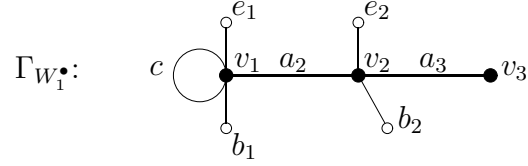
\begin{picture}(300,100)
\put(170,60){\circle*{5}}
\put(173,63){$v_{1}$}
\put(170,60){\line(1,0){50}}
\put(190,63){$a_{2}$}
\put(160,60){\circle{20}}
\put(141,60){$c$}
\put(170,60){\line(0,-9){18}}
\put(170,40){\circle{4}}
\put(172,30){$b_{1}$}
\put(170,78){\line(0,-9){18}}
\put(170,80){\circle{4}}
\put(172,83){$e_{1}$}
\put(220,60){\circle*{5}}
\put(222,63){$v_{2}$}
\put(220,78){\line(0,-9){18}}
\put(220,80){\circle{4}}
\put(222,83){$e_{2}$}
\put(220,60){\line(10,0){50}}
\put(243, 63){$a_{3}$}
\put(270,60){\circle*{5} $v_{3}$}
\put(220,60){\line(5,-9){10}}
\put(231,40){\circle{4} $b_{2}$}
\put(90,59){$\Gamma_{W^{\bp}_{1}}$:}
\end{picture}

\noindent
Note that the set of terminal vertices avoiding to $e^{\rm op}(\Gamma_{W^{\bp}})$ of $W_{1}^{\bp}$ is $\{v_{3}\}$. 

(c) We obtain a minimal quasi-tree $\Gamma_{D_{W}} \defeq \Gamma$ associated to $D_{W}$ is as follows:

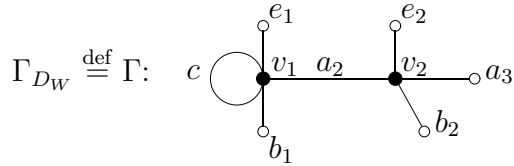
\begin{picture}(300,100)
\put(170,60){\circle*{5}}
\put(173,63){$v_{1}$}
\put(170,60){\line(1,0){50}}
\put(190,63){$a_{2}$}
\put(160,60){\circle{20}}
\put(141,60){$c$}
\put(170,60){\line(0,-9){18}}
\put(170,40){\circle{4}}
\put(172,30){$b_{1}$}
\put(170,78){\line(0,-9){18}}
\put(170,80){\circle{4}}
\put(172,83){$e_{1}$}
\put(220,60){\circle*{5}}
\put(222,63){$v_{2}$}
\put(220,78){\line(0,-9){18}}
\put(220,80){\circle{4}}
\put(222,83){$e_{2}$}
\put(220,60){\line(10,0){28}}
\put(250,60){\circle{4} $a_{3}$}
\put(220,60){\line(5,-9){10}}
\put(231,40){\circle{4} $b_{2}$}
\put(75,59){$\Gamma_{D_W}\defeq \Gamma$:}
\end{picture}

\noindent
On the other hand, $W_{\Gamma}^{\bp}$ is as follows:

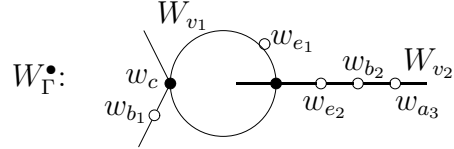
\begin{picture}(300,100)

\put(170,60){\circle{40}}
\put(150,60){\circle*{4}}
\put(150,60){\line(-5,10){10}}
\put(150,60){\line(-5,-10){5}}
\put(144,48){\circle{4}}
\put(143,46){\line(-5,-10){5}}
\put(190,60){\circle*{4}}

\put(185.7,75){\circle{4}}
\put(188, 75){$w_{e_{1}}$}

\put(175,60){\line(5,0){30}}
\put(207,60){\circle{4}}

\put(209,60){\line(5,0){10}}
\put(221,60){\circle{4}}

\put(223,60){\line(5,0){10}}
\put(235,60){\circle{4}}
\put(215, 65){$w_{b_{2}}$}
\put(200, 50){$w_{e_{2}}$}
\put(235, 50){$w_{a_{3}}$}

\put(237,60){\line(5,0){10}}

\put(145,83){$W_{v_{1}}$}
\put(238,65){$W_{v_{2}}$}

\put(133,60){$w_{c}$}
\put(125,48){$w_{b_{1}}$}

\put(90,59){$W_{\Gamma}^{\bp}$:}
\end{picture}

(d) Next, we give an example $\Gamma'\subseteq \Gamma_{W^{\bp}}$, which is a tree containing all open edges of $\Gamma_{W^{\bp}}$, and which is {\it not} a minimal quasi-tree associated to $D_{W}$

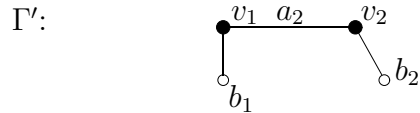
\begin{picture}(300,100)
\put(170,60){\circle*{5}}
\put(173,63){$v_{1}$}
\put(170,60){\line(1,0){50}}
\put(190,63){$a_{2}$}

\put(170,60){\line(0,-9){18}}
\put(170,40){\circle{4}}
\put(172,30){$b_{1}$}

\put(220,60){\circle*{5}}
\put(222,63){$v_{2}$}

\put(220,60){\line(5,-9){10}}
\put(231,40){\circle{4} $b_{2}$}
\put(90,59){$\Gamma'$:}
\end{picture}

\end{example}

\subsubsection{}\label{imgraph}
We maintain the notation introduced in \ref{defquasitree}. Suppose that $n_{W}\neq 0$, and that $\Gamma=\Gamma_{D_{W}}$ is a minimal quasi-tree associated to $D_{W}$. 

The construction of $W_{\Gamma}$ implies that there is a natural morphism $$f_{\Gamma}: W_{\Gamma} \migi W$$ over $k$.  We denote by $$\phi_{\Gamma}: \Gamma \migi \Gamma_{W^{\bp}}$$ the map of dual semi-graphs induced by $f_{\Gamma}$. Let $$D_{E_{1}}, D_{E_{2}} \subseteq D_{W_\Gamma} \setminus D_{W}$$ be the subsets of marked points of $W^{\bp}_{\Gamma}$ such that $D_{W_{\Gamma}}=D_{W} \cup D_{E_{1}} \cup D_{E_{2}}$, that $f_{\Gamma}(D_{E_{1}}) \cap N =\emptyset$, and that $N_{2}\defeq f_{\Gamma}(D_{E_{2}}) \subseteq N$. Namely, the following conditions hold: (a) $D_{W_{\Gamma}}=D_{W} \cup D_{E_{1}} \cup D_{E_{2}}$; (b) $f_{\Gamma}(D_{W})=D_{W}$; (c) $f_{\Gamma}(D_{E_{1}}) \subseteq W^{\rm sing} \setminus N$; (d) $N_{2} \subseteq N$; (e) $f_{\Gamma}(W^{\rm sing}_{\Gamma}) \subseteq W^{\rm sing}$.

We define a pointed stable curve over $k$ to be (for instance, see Example \ref{graphim1} below)$$W_{\Gamma^{\rm im}}^{\bp}=(W_{\Gamma^{\rm im}} \defeq f_{\Gamma}(W_{\Gamma}), D_{W_{\Gamma^{\rm im}}} \defeq f_{\Gamma}(D_{W}\cup D_{E_{1}})).$$ Then we see that the dual semi-graph of $W^{\bp}_{\Gamma^{\rm im}}$ coincides with the image $\text{Im}(\phi_{\Gamma})$. We call $$\Gamma^{\rm im} \defeq  \Gamma_{W^{\bp}_{\Gamma^{\rm im}}}$$ {\it the image of the map $\phi_{\Gamma}$} (for instance, see Example \ref{graphim1} below). Moreover, we denote by $$\text{norm}_{\Gamma}: W_{\Gamma} \migi W_{\Gamma^{\rm im}}$$ the natural morphism induced by $f_{\Gamma}$ which coincides with the normalization morphism of $W_{\Gamma^{\rm im}}$ over the set of nodes $N_{2}$. Then we have that $D_{W_{\Gamma}}=\text{norm}^{-1}_{\Gamma}(D_{W_{\Gamma^{\rm im}}} \cup N_{2}).$

\subsubsection{} 

We maintain the notation introduced in Example \ref{rem-def-3-7} and give an example of $\Gamma^{\rm im}$.

\begin{example}\label{graphim1}
We see that the set of open edges of $W_{\Gamma}^{\bp}$ corresponding to $D_{E_{1}}$ is $\{a_{3}\}$ (Example \ref{rem-def-3-7} (c)),  and that the set of open edges of $W_{\Gamma}^{\bp}$ corresponding to $D_{E_{2}}$ is $\{e_{1}, e_{2}\}$ (Example \ref{rem-def-3-7} (c)). Then the image $\Gamma^{\rm im}$ of the map $\phi_{\Gamma}: \Gamma \migi \Gamma_{W}^{\bp}$ is as follows:

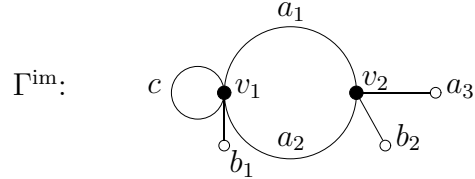
\begin{picture}(300,100)
\put(170,60){\circle*{5} $v_{1}$}
\put(195,60){\circle{50}}
\put(190,89){$a_{1}$}
\put(190,40){$a_{2}$}
\put(160,60){\circle{20}}
\put(141,60){$c$}
\put(170,60){\line(0,-9){18}}
\put(170,40){\circle{4}}
\put(172,30){$b_{1}$}
\put(220,60){\circle*{5}}
\put(222,63){$v_{2}$}
\put(220,60){\line(10,0){28}}
\put(250,60){\circle{4} $a_{3}$}
\put(220,60){\line(5,-9){10}}
\put(231,40){\circle{4} $b_{2}$}
\put(90,59){$\Gamma^{\rm im}$:}
\end{picture}

\noindent
On the other hand, $W_{\Gamma^{\rm im}}^{\bp}$ is as follows:

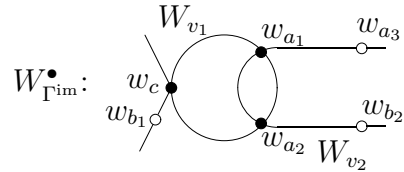
\begin{picture}(300,100)

\put(170,60){\circle{40}}
\put(150,60){\circle*{4}}
\put(150,60){\line(-5,10){10}}
\put(150,60){\line(-5,-10){5}}
\put(144,48){\circle{4}}
\put(143,46){\line(-5,-10){5}}

\put(220,60){\oval(90,30)[l]}
\put(184,73.5){\circle*{4}}
\put(184,46.5){\circle*{4}}
\put(222,45.5){\circle{4}}
\put(224,45.5){\line(5,0){7}}

\put(222,75){\circle{4}}
\put(224,75){\line(5,0){7}}


\put(145,83){$W_{v_{1}}$}
\put(205,32){$W_{v_{2}}$}

\put(133,60){$w_{c}$}
\put(125,48){$w_{b_{1}}$}
\put(222,50){$w_{b_{2}}$}
\put(184,79){$w_{a_{1}}$}
\put(184,38){$w_{a_{2}}$}
\put(220,81){$w_{a_{3}}$}

\put(90,59){$W_{\Gamma^{\rm im}}^{\bp}$:}
\end{picture}

\end{example}

\subsection{The first main theorem}

We are going to prove the first main theorem of the present paper. 

\subsubsection{\bf Settings} We maintain the notation introduced in \ref{defcurves} and the assumption introduced in \ref{assum-2.3.3} (i.e.,  $n\defeq p^{t}-1$). Moreover, we assume that  $X^{\bp}$ is a component-generic pointed stable curve over $k$.

\subsubsection{} We have the following key proposition.

\begin{proposition}\label{prop-3-9} Let $D\in (\mbZ/n\mbZ)^{\sim}[D_{X}]^{0}$ be an effective divisor on $X$ such that $s(D)=n_{X}-1$ if $n_{X}\neq 0$, and that ${\rm deg}(D^{(i)}) \geq {\rm deg}(D), \ i \in \{0, 1, \dots, t-1\}.$ Then the following statements hold:

(i) Suppose that one of the following conditions is satisfied:
\begin{itemize}
\item $n_{X}=0$ and $\Gamma_{X^{\bp}} \setminus e^{\rm lp}(\Gamma_{X^{\bp}})$ is a tree.

\item $n_{X}\neq 0$ and $\Gamma_{X^{\bp}}=\Gamma_{D_{W}}$ is a minimal quasi-tree associated to $D_{X}$ (see Definition \ref{def-3-7}). 
\end{itemize}
Then $\gamma_{(\alpha, D)}$ attains  maximum for all $\alpha \in {\rm Rev}_{D}^{\rm adm}(X^{\bp}) \setminus \{0\}$. Namely, the following holds:
\begin{eqnarray*}
\gamma_{(\alpha, D)}=\gamma^{\rm max}_{X^{\bp}}=\left\{ \begin{array}{ll}
g_{X}-1, & \text{if} \ n_{X} =0,
\\
g_{X}+n_{X}-2, & \text{if} \ n_{X} \neq 0.
\end{array} \right.
\end{eqnarray*}

(ii) There exists an element $\beta \in {\rm Rev}^{\rm adm}_{D}(X^{\bp}) \setminus \{0\}$ such that $\gamma_{(\beta, D)}$ attains maximum.  Namely, the following holds:
\begin{eqnarray*}
\gamma_{(\beta, D)}=\gamma^{\rm max}_{X^{\bp}}=\left\{ \begin{array}{ll}
g_{X}-1, & \text{if} \ n_{X} =0,
\\
g_{X}+n_{X}-2, & \text{if} \ n_{X} \neq 0.
\end{array} \right.
\end{eqnarray*}
\end{proposition}

\begin{proof}
(i) Let $f^{\bp}: Y^{\bp}=(Y, D_{Y}) \migi X^{\bp}$ be a Galois multi-admissible covering over $k$ with Galois group $\mbZ/n\mbZ$ induced by $\alpha$. To verify (i), we may assume that $Y^{\bp}$ is connected.

Suppose that $n_{X}=0$. Since $\Gamma_{X^{\bp}} \setminus e^{\rm lp}(\Gamma_{X^{\bp}})$ is a tree. We see immediately that $f$ is \'etale. Then (i) follows from Theorem \ref{lem-3-1} and Proposition \ref{pro-3-6}.

Suppose that $n_{X}\neq 0$. Let $v \in v(\Gamma_{X^{\bp}})$ and $\pi_{0}(\overline {X\setminus X_{v}})$ the set of connected components of $\overline {X\setminus X_{v}}$, where $\overline {X\setminus X_{v}}$ denotes the topological closure of ${X\setminus X_{v}}$ in $X$. We put $$D_{v}\defeq (D_{X} \cap X_{v}) \cup (\bigcup_{C\in \pi_{0}(\overline {X\setminus X_{v}})}(C\cap X_{v})).$$  Let $X_{v}^{\bp}= (X_{v}, D_{X_{v}}\defeq D_{v})$, $v \in v(\Gamma_{X^{\bp}}),$ be a pointed stable curve of type $(g_{X_v}, n_{X_v})$ over $k$. Then $f^{\bp}$ induces a Galois multi-admissible covering $$f_{v}^{\bp}: Y_{v}^{\bp} \migi X_{v}^{\bp}, \ v \in v(\Gamma_{X^{\bp}}),$$ over $k$ with Galois group $\mbZ/n\mbZ$.

Note that since $\Gamma_{X^{\bp}}$ is a minimal quasi-tree associated to $D_{X}$, we have that $C\cap X_{v}=\{x_{C}\}$ is a closed point of $X$ for all $C \in \pi_{0}(\overline {X\setminus X_{v}})$. Moreover, since $\Gamma_{X^{\bp}} \setminus e^{\rm lp}(\Gamma_{X^{\bp}})$ is a tree, the ramification divisor $$Q_{v}\defeq \sum_{x\in D_{X} \cap X_{v}} \text{ord}_{x}(Q_{v})x + \sum_{C\in \pi_{0}(\overline {X\setminus X_{v}})}\text{ord}_{x_{C}}(Q_{v})x_{C} \in (\mbZ/n\mbZ)^{\sim}[D_{v}]^{0}, \ v\in v(\Gamma_{X^{\bp}}),$$ on $X_{v}$ induced by $f_{v}$ satisfies the following: $$\text{ord}_{x}(Q_{v})\defeq \text{ord}_{x}(D), \ x \in D_{X} \cap X_{v},$$ $$\text{ord}_{x_{C}}(Q_{v})\defeq [\sum_{c \in D_{X} \cap C}\text{ord}_{c}(D)], \ C\in \pi_{0}(\overline {X\setminus X_{v}}),$$ where $[(-)]$ denotes the integer which is equal to the image of $(-)$ in $\mbZ/n\mbZ$ when we identify $\{0, \dots, n-1\}$ with $\mbZ/n\mbZ$ naturally. By applying similar arguments to the arguments given in the proof of Lemma \ref{lem-3-5}, we have $$\text{deg}(Q_{v})=(\#(D_{v})-1)n \ \text{and} \ \text{deg}(Q_{v}^{(i)})= \text{deg}(Q_{v}), \ i\in \{0, \dots, t-1\}.$$

Let $\Pi_{X_{v}^{\bp}}$, $v\in v(\Gamma_{X^{\bp}})$, be the admissible fundamental group of $X^{\bp}_{v}$. Write $\alpha_{v} \in \text{Rev}^{\rm adm}_{Q_{v}}(X_{v}^{\bp})$ for the composition of the natural homomorphisms $\Pi_{X_{v}^{\bp}} \migi \Pi_{X^{\bp}}^{\rm ab} \overset{\alpha}\migi \mbZ/n\mbZ$. Note that $\alpha_{v}\neq 0$ for all $v\in v(\Gamma_{X^{\bp}})$.  
Then Proposition \ref{pro-3-6} implies 
\begin{eqnarray*}
\gamma_{(\alpha_{v}, Q_{v})}=\left\{ \begin{array}{ll}
g_{X_{v}}-1, & \text{if} \ \text{Supp}(Q_{v})=\emptyset,
\\
g_{X_{v}}+s(Q_{v})-2, &  \text{if} \ \text{Supp}(Q_{v})\neq \emptyset.
\end{array} \right.
\end{eqnarray*}
Thus, Theorem \ref{lem-3-1} implies that 
\begin{eqnarray*}
\gamma_{(\alpha, D)}=\gamma^{\rm max}_{X^{\bp}}=\left\{ \begin{array}{ll}
g_{X}-1, & \text{if} \ n_{X} =0,
\\
g_{X}+n_{X}-2, & \text{if} \ n_{X} \neq 0.
\end{array} \right.
\end{eqnarray*}
This completes the proof of (i).

(ii) Suppose that $n_{X}\leq 1$. (ii) follows from Proposition \ref{maineasy} (ii). Then to verify (ii), we may assume that $n_{X}\geq 2$.

Let $\Gamma\defeq \Gamma_{D_{X}}$ be a minimal quasi-tree associated to $D_{X}$, $\Gamma^{\rm im}$ the image of the natural morphism $\phi_{\Gamma}: \Gamma \migi \Gamma_{X^{\bp}}$, and $$X^{\bp}_{\Gamma}=(X_{\Gamma}, D_{X_{\Gamma}}), \ X_{\Gamma^{\rm im}}^{\bp}=(X_{\Gamma^{\rm im}}, D_{X_{\Gamma^{\rm im}}})$$ the pointed stable curves over $k$ associated to $\Gamma, \ \Gamma^{\rm im}$, respectively (\ref{defquasitree}, \ref{imgraph}). Note that $D$ is also an effective divisor on $X_{\Gamma^{\rm im}}$. 

Write $D_{\Gamma}$ for the pulling back divisor $\text{norm}_{\Gamma}^{*}(D)$ (see \ref{imgraph}  for the definition of $\text{norm}_{\Gamma}$). Let $\alpha_{\Gamma} \in \text{Rev}^{\rm adm}_{D_{\Gamma}}(X^{\bp}_{\Gamma})$ be an arbitrary element such that $\alpha_{\Gamma}\neq 0$. Then similar arguments to the arguments given in (i) imply 
$\gamma_{(\alpha_{\Gamma}, D_{\Gamma})}=g_{X_{\Gamma}}+n_{X}-2$, where $g_{X_{\Gamma}}$ denotes the genus of $X_{\Gamma}$. We denote by $$g_{\Gamma}^{\bp}: Z_{\Gamma}^{\bp} \migi X^{\bp}_{\Gamma}$$ the Galois multi-admissible covering over $k$ with Galois group $\mbZ/n\mbZ$ induced by $\alpha_{\Gamma}$.
By gluing $Z_{\Gamma}^{\bp}$ along $g_{\Gamma}^{-1}(D_{X_\Gamma}\setminus \text{norm}_{\Gamma}^{-1}(D_{X_{\Gamma^{\rm im}}}))$ in a way that is compatible with the gluing of $X_{\Gamma}^{\bp}$ that gives rise to $X_{\Gamma^{\rm im}}^{\bp}$, we obtain a pointed stable curve $Z_{\Gamma^{\rm im}}^{\bp}$ over $k$. Moreover, $g_{\Gamma}^{\bp}$ induces a Galois multi-admissible covering
$$g_{\Gamma^{\rm im}}^{\bp}: Z_{\Gamma^{\rm im}}^{\bp} \migi X^{\bp}_{\Gamma^{\rm im}}$$ over $k$ with Galois group $\mbZ/n\mbZ$. Let $\Pi_{X^{\bp}_{\Gamma}}$, $\Pi_{X^{\bp}_{\Gamma^{\rm im}}}$ be the admissible fundamental groups of $X^{\bp}_{\Gamma}$, $X^{\bp}_{\Gamma^{\rm im}}$, respectively. Write $\alpha_{\Gamma^{\rm im}}$ for an element of $\text{Hom}(\Pi^{\rm ab}_{X^{\bp}_{\Gamma^{\rm im}}}, \mbZ/n\mbZ)$ induced by $g^{\bp}_{\Gamma^{\rm im}}$ such that the composition of the natural homomorphisms $\Pi_{X^{\bp}_{\Gamma}}^{\rm ab} \migi \Pi_{X^{\bp}_{\Gamma^{\rm im}}}^{\rm ab} \overset{\alpha_{\Gamma^{\rm im}}}\migi \mbZ/n\mbZ$ is equal to $\alpha_{\Gamma}$. We put $D_{\Gamma^{\rm im}} \defeq D_{\alpha_{\Gamma^{\rm im}}}$. Then Theorem \ref{lem-3-1} implies that $$\gamma_{(\alpha_{\Gamma^{\rm im}}, D_{\Gamma^{\rm im}})}=g_{X_{\Gamma^{\rm im}}}+n_{X}-2,$$ where $g_{X_{\Gamma^{\rm im}}}$ denotes the genus of $X_{\Gamma^{\rm im}}$.

On the other hand, we write $\pi_{0}(\overline {X \setminus X_{\Gamma^{\rm im}}})$ for the set of connected components of $\overline {X \setminus X_{\Gamma^{\rm im}}}$, where $\overline {X \setminus X_{\Gamma^{\rm im}}}$ denotes the topological closure of $X \setminus X_{\Gamma^{\rm im}}$ in $X$. We define the following pointed stable curve $$C^{\bp}=(C, D_{C}\defeq C \cap X_{\Gamma^{\rm im}}), \ C \in \pi_{0}(\overline {X \setminus X_{\Gamma^{\rm im}}}),$$ over $k$. Note that since $X^{\bp}$ is component-generic, we have that $C^{\bp}$ is also component-generic. Then the $p$-rank $\sigma_C$ is equal to the genus of $C^{\bp}$.

Let $C \in \pi_{0}(\overline {X \setminus X_{\Gamma^{\rm im}}})$. We put $$Z^{\bp}_{C}\defeq \bigsqcup_{i \in \mbZ/n\mbZ}C_{i}^{\bp},$$ where $C_{i}^{\bp}$ is a copy of $C^{\bp}$, and $\sqcup$ means a disjoint union. Then we obtain a Galois multi-admissible covering $$g_{C}^{\bp}: Z^{\bp}_{C}\migi C^{\bp}$$ over $k$ with Galois group $\mbZ/n\mbZ$, where the restriction morphism $g_{C}^{\bp}|_{C^{\bp}_{i}}$ is an identity, and the Galois action is $j(C_{i})=C_{i+j}$ for all $i, j \in \mbZ/n\mbZ$. By gluing $Z_{\Gamma^{\rm im}}^{\bp}$ and $\{Z_{C}^{\bp}\}_{C \in \pi_{0}(\overline {X \setminus X_{\Gamma^{\rm im}}})}$ along $g^{-1}_{\Gamma^{\rm im}}(X_{\Gamma^{\rm im}} \cap (\bigcup_{C \in \pi_{0}(\overline {X \setminus X_{\Gamma^{\rm im}}})}C))$ and $\{g^{-1}_{C} (X_{\Gamma^{\rm im}} \cap C)\}_{C \in \pi_{0}(\overline {X \setminus X_{\Gamma^{\rm im}}})}$ \tch{in a way} that is compatible with the gluing of $\{X^{\bp}_{\Gamma^{\rm im}}\} \cup \{C^{\bp}\}_{C \in \pi_{0}(\overline {X \setminus X_{\Gamma^{\rm im}}})}$ that gives rise to $X^{\bp}$, we obtain a Galois multi-admissible covering $$g^{\bp}: Z^{\bp} \migi X^{\bp}$$ over $k$ with Galois group $\mbZ/n\mbZ$. 

Let $\Pi_{X^{\bp}}$ be the admissible fundamental group of $X^{\bp}$. Moreover, we write $\beta \in \text{Rev}^{\rm adm}_{D}(X^{\bp})$ for an element induced by $g^{\bp}$ such that the composition of the natural homomorphisms $\Pi_{X^{\bp}_{\Gamma^{\rm im}}}^{\rm ab} \migi \Pi_{X^{\bp}}^{\rm ab} \overset{\beta}\migi \mbZ/n\mbZ$ is equal to $\alpha_{\Gamma^{\rm im}}$. By applying Theorem \ref{lem-3-1}, we see $$\gamma_{(\beta, D)}=\gamma_{X^{\bp}}^{\rm max}=g_{X}+n_{X}-2.$$ We complete the proof of (ii).
\end{proof}

\subsubsection{}
Now, the main result of the present section is as follows.

\begin{theorem}\label{main-them-1}
Let $X^{\bp}=(X, D_{X})$ be a  component-generic pointed stable curve over $k$. Let $m \in \mbN$ be an arbitrary positive natural number prime to $p$ and $D \in (\mbZ/m\mbZ)^{\sim}[D_{X}]^{0}$ (\ref{2.2.4}). Let $t \in \mbN$ be a positive natural number such that $p^{t}=1$ in $(\mbZ/m\mbZ)^{\times}$. Write $n$ for $p^{t}-1$, $m'$ for $n/m$, and $D'$ for the divisor $m'D \in (\mbZ/n\mbZ)^{\sim}[D_{X}]^{0}$ when we identify $\mbZ/m\mbZ$ with the unique subgroup of $\mbZ/n\mbZ$ of order $m$. Then the following statements hold:


(i) Suppose that one of the following conditions is satisfied:
\begin{itemize}
\item $n_{X}=0$ and $\Gamma_{X^{\bp}} \setminus e^{\rm lp}(\Gamma_{X^{\bp}})$ is a tree.

\item $n_{X}\neq 0$ and $\Gamma_{X^{\bp}}=\Gamma_{D_{W}}$ is a minimal quasi-tree associated to $D_{X}$ (\ref{defquasitree}). 
\end{itemize}
Then the following are equivalent: 
\begin{quote}
(a) $\gamma_{(\alpha, D)}$ attains  maximum for all $\alpha \in {\rm Rev}_{D}^{\rm adm}(X^{\bp}) \setminus \{0\}$ (Definition \ref{def-2-4} (i) and Definition \ref{def-3-2} (ii)). Namely, the following holds (see Definition \ref{def-3-2} (i) for $\gamma^{\rm max}_{X^{\bp}}$):
\begin{eqnarray*}
\gamma_{(\alpha, D)}=\gamma^{\rm max}_{X^{\bp}}=\left\{ \begin{array}{ll}
g_{X}-1, & \text{if} \ n_{X} =0,
\\
g_{X}+n_{X}-2, & \text{if} \ n_{X} \neq 0.
\end{array} \right.
\end{eqnarray*}
(b) We have (see \ref{2.2.4} and Definition \ref{def-2-4} for $(D')^{(i)}$)
\begin{eqnarray*}
s(D)=\left\{ \begin{array}{ll}
0, & \text{if} \ n_{X} =0,
\\
n_{X}-1, & \text{if} \ n_{X} \neq 0
\end{array} \right.
\end{eqnarray*}
and ${\rm deg}((D')^{(i)}) \geq {\rm deg}(D'), \ i \in \{0, 1, \dots, t-1\}$.
\end{quote}


(ii) The following are equivalent: 
\begin{quote}
(a) There exists an element $\beta \in {\rm Rev}^{\rm adm}_{D}(X^{\bp}) \setminus \{0\}$ such that $\gamma_{(\beta, D)}$ attains maximum.  Namely, the following holds:
\begin{eqnarray*}
\gamma_{(\beta, D)}=\gamma^{\rm max}_{X^{\bp}}=\left\{ \begin{array}{ll}
g_{X}-1, & \text{if} \ n_{X} =0,
\\
g_{X}+n_{X}-2, & \text{if} \ n_{X} \neq 0.
\end{array} \right.
\end{eqnarray*}
(b) We have
\begin{eqnarray*}
s(D)=\left\{ \begin{array}{ll}
0, & \text{if} \ n_{X} =0,
\\
n_{X}-1, & \text{if} \ n_{X} \neq 0,
\end{array} \right.
\end{eqnarray*}
and ${\rm deg}((D')^{(i)}) \geq {\rm deg}(D'), \ i \in \{0, 1, \dots, t-1\}$.
\end{quote}
\end{theorem}

\begin{proof}
(i) Write $\alpha' \in \text{Rev}_{D'}^{\rm adm}(X^{\bp})$ for the element induced by $\alpha$. Then we have $\gamma_{(\alpha, D)}=\gamma_{(\alpha', D')}$. The ``only if" part of (i) follows from Lemma \ref{lem-3-4}. Moreover, the ``if" part of (i) follows immediately from Proposition \ref{prop-3-9} (i). 

(ii) The ``only if" part of (ii) follows from Lemma \ref{lem-3-4}. We prove the ``if" part of (ii). Let $\Gamma\defeq \Gamma_{D_{X}}$ be a minimal quasi-tree associated to $D_{X}$, $\Gamma^{\rm im}$ the image of the natural morphism $\phi_{\Gamma}: \Gamma \migi \Gamma_{X^{\bp}}$, $X^{\bp}_{\Gamma^{\rm im}}=(X_{\Gamma^{\rm im}}, D_{X_{\Gamma^{\rm im}}})$ the pointed stable curve over $k$ associated to $\Gamma^{\rm im}$ (\ref{imgraph}), and $\Pi_{X^{\bp}_{\Gamma^{\rm im}}}$ the admissible fundamental group of $X^{\bp}_{\Gamma^{\rm im}}$. Note that $D$ is also an effective divisor on $X_{\Gamma^{\rm im}}$. Let $\beta_{\Gamma^{\rm im}}$ be an arbitrary element of $\text{Rev}_{D}^{\rm adm}(X^{\bp}_{\Gamma^{\rm im}})\setminus \{0\}$. Write $\beta_{\Gamma^{\rm im}}' \in \text{Rev}_{D'}^{\rm adm}(X^{\bp}_{\Gamma^{\rm im}})$ for the element induced by $\beta_{\Gamma^{\rm im}}$. 

Let $\Pi_{X^{\bp}}$ be the admissible fundamental group of $X^{\bp}$. Proposition \ref{prop-3-9} (ii) implies that there exists $\beta' \in \text{Rev}_{D'}^{\rm adm}(X^{\bp})$ such that the composition of the natural homomorphisms $\Pi_{X_{\Gamma^{\rm im}}^{\bp}}^{\rm ab} \migi \Pi_{X^{\bp}} \overset{\beta'}\migi \mbZ/n\mbZ$ is equal to $\beta_{\Gamma^{\rm im}}'$, and that $\gamma_{(\beta', D')}$ attains maximum. Moreover, the construction of $\beta'$ given in the proof of Proposition \ref{prop-3-9} (ii) (i.e., the Galois multi-admissible covering of $\widetilde X^{\bp}_{v}$, $v\in v(\Gamma_{X^{\bp}}) \setminus v(\Gamma^{\rm im})$, induced by $\beta'$ is trivial) implies that $\beta'$ can be induced by an element of $\beta\in \text{Rev}_{D}^{\rm adm}(X^{\bp})$. Then $\gamma_{(\beta, D)}=\gamma_{(\beta', D')}$ attains maximum. This completes the proof of (ii).
\end{proof}


\subsection{$(m, n_{X})$-ordinary curves} 

\subsubsection{} We maintain the notation introduced in \ref{defcurves}. Moreover, let $\overline \mcM_{g_{X}, n_{X}}$ be the moduli stack parameterizing pointed stable curves of type $(g_{X}, n_{X})$ (\ref{comgeneric}) and $\mcM_{g_{X}, n_{X}} \subseteq \overline \mcM_{g_{X}, n_{X}}$ the open substack parameterizing smooth pointed stable curves of type $(g_{X}, n_{X})$. We denote by $\overline M_{g_{X}, n_{X}}$, $M_{g_{X}, n_{X}}$ the coarse moduli spaces of $\overline \mcM_{g_{X}, n_{X}}$, $\mcM_{g_{X}, n_{X}}$, respectively.

\subsubsection{}

Let $m \in \mbN$ be an arbitrary positive natural number prime to $p$, and let $D \in (\mbZ/m\mbZ)^{\sim}[D_{X}]^{0}$ be an effective divisor on $X$ such that  $\text{deg}(D)=(n_{X}-1)m$ if $n_{X}\neq 0$.

\begin{definition}
We shall say that $X^{\bp}$ is {\it $(m, n_{X})$-ordinary} if, for {\it all} $\alpha \in {\rm Rev}_{D}^{\rm adm}(X^{\bp})\setminus \{0\}$, we have 
\begin{eqnarray*}
\gamma_{(\alpha, D)}=\gamma^{\rm max}_{X^{\bp}}=\left\{ \begin{array}{ll}
g_{X}-1, & \text{if} \ n_{X} =0,
\\
g_{X}+n_{X}-2, & \text{if} \ n_{X} \neq 0.
\end{array} \right.
\end{eqnarray*}
Note that, if $n_{X}=0$ and $X^{\bp}$ is non-singular, then the definition of $(m, n_{X})$-ordinary coincides with the definition of {\it $m$-ordinary} introduced by Nakajima (\cite[\S4]{N}).
\end{definition}

On the other hand, let $t \in \mbN$ be a positive natural number such that $p^{t}=1$ in $(\mbZ/m\mbZ)^{\times}$. Write $n$ for $p^{t}-1$, $m'$ for $n/m$, and $D'$ for the divisor $m'D \in (\mbZ/n\mbZ)^{\sim}[D_{X}]^{0}$ when we identify $\mbZ/m\mbZ$ with the unique subgroup of $\mbZ/n\mbZ$ of order $m$. If $n_{X}\neq 0$ and $X^{\bp}$ is $(m, n_{X})$-ordinary, then Lemma \ref{lem-3-4} (ii) implies that $${\rm deg}((D')^{(i)}) = {\rm deg}(D')=(n_{X}-1)n, \ i \in \{0, 1, \dots, t-1\}.$$


\subsubsection{} 
We denote by $$\overline U_{(m, n_{X})} \defeq \{ q\in  \overline M_{g_{X}, n_{X}}\ | \ \text{curves corresponding to} \ q \ \text{is} \ (m, n_{X})\text{-ordinary}\} \subseteq \overline M_{g_{X}, n_{X}}.$$ Moreover, we put $U_{(m, n_{X})}\defeq \overline U_{(m, n_{X})} \cap M_{g_{X}, n_{X}}$ (i.e., the set of points corresponding to {\it smooth} $(m, n_{X})$-ordinary curves). Then we have the following result.

\begin{proposition}\label{coro-3-13}
(i) The set $U_{(m, n_{X})}$ is a non-empty open subset of $M_{g_{X}, n_{X}}$. 

(ii) Suppose that $n_{X} \leq 1$. Then we have $$M_{g_{X}, n_{X}}^{\rm cl} \cap (\bigcap_{m \in \mbN \ {\text{\rm s.t.}} \ (m, p)=1}U_{(m, n_{X})} )=\emptyset,$$ where $M_{g_{X}, n_{X}}^{\rm cl}$ denotes the set of closed points of $M_{g_{X}, n_{X}}$.

(iii) Let $q \in M_{g_{X}, n_{X}}$ be an arbitrary point. We denote by $\Pi_{q}$ the admissible fundamental group of a smooth pointed stable curve corresponding to a geometric point over $q$. Note that the isomorphism class of $\Pi_{q}$ (as a profinite group) does not depend on the choices of geometric points over $q$.  Suppose that $n_{X} \leq 1$. Let $U \subseteq M_{g_{X}, n_{X}}$ be an arbitrary non-empty open subset. Then there exist closed points $q_{1}, q_{2} \in U \cap M_{g_{X}, n_{X}}^{\rm cl}$ such that $\Pi_{q_{1}} \not\cong \Pi_{q_{2}}$.
\end{proposition}

\begin{proof}
(i) By applying similar arguments to the arguments given in the proof of \cite[Theorem 2]{N}, (i) follows  from Theorem \ref{main-them-1} (i).

(ii) Suppose that $k=\overline \mbF_{p}$, and that $X^{\bp}$ is a smooth pointed stable curve of type $(g_{X}, n_{X})$ over $k$. To verify (ii), we only need to prove that, if $n_{X}\leq 1$, $X^{\bp}$ is not $(m, n_{X})$-ordinary for some positive natural number $m \in \mbN$ prime to $p$. 

Since $n_{X}\leq 1$, we have $(\mbZ/m\mbZ)^{\sim}[D_{X}]^{0}=\{0\}$. We denote by $\Theta_{X}$ the Raynaud-Tamagawa divisor associated to $\mcB^{1}_{0}$ (\ref{thetabundle},  \ref{divisortheta}) and by $\Theta'$ an arbitrary irreducible component of $\Theta_{X}$. Write $J_{X}^{1}$ for the pulling back of the Jacobian $J_{X}$ of $X$ by the absolute Frobenius morphism $F_{k}$ of $k$. If $X^{\bp}$ is $(m, n_{X})$-ordinary for all prime-to-$p$ positive natural numbers, then a property of Raynaud-Tamagawa theta divisors concerning new-ordinary coverings (e.g. \cite[\S 1.2 and \S 1.3]{PS}) implies $$J^{1}_{X}\{p'\} \cap \Theta_{X}(k) \subseteq \{0_{J^{1}_{X}}\},$$ where $0_{J^{1}_{X}}$ denotes the zero point of $J^{1}_{X}$, and $J^{1}_{X}\{p'\}$ denotes the set of prime-to-$p$ torsion points of $J^{1}_{X}(k)$. Since $\text{dim}(\Theta')>0$, we have $$\Theta'\{p'\}\defeq J^{1}_{X}\{p'\} \cap \Theta'(k) \subseteq J^{1}_{X}\{p'\} \cap \Theta_{X}(k)$$ is not dense in $\Theta'$. 

On the other hand, since $\Theta'$ is defined over $k=\overline \mbF_{p}$, by applying a result of Anderson-Indik (\cite[\S5]{T3}), we have that $\Theta'$ is a subvariety of a translate of a proper sub-abelian \tch{variety} of $J^{1}_{X}$. But this \tch{contradicts} a result of Raynaud (\cite[Proposition 1.2.1]{R2}) which says that there exists an irreducible component $\Theta'$ of $\Theta_{X}$ such that $\Theta'$ is not contained in a translate of a proper sub-abelian variety of $J^{1}_{X}$. This completes the proof of (ii).

(iii)  Let $q$ be an arbitrary closed point of $U$ and $q^{\rm gen}$ the generic point of $M_{g_{X}, n_{X}}$. Suppose that (iii) does not hold. Thus there exists a closed point $q \in U^{\rm cl}$ such that $\Pi_{q^{\rm gen}} \cong \Pi_{q}$. Then there exist a discrete valuation ring $R$ and a morphism $c_{R}: \spec R \migi M_{g_{X}, n_{X}}$ such that $c_{R}(\eta_{R})=q^{\rm gen}$ and $c_{R}(s_{R})=q$, where $\eta_{R}$ is the generic point of $\spec R$ and $s_{R}$ is the closed point of $\spec R$. By replacing $R$ by a finite extension of $R$, we have a smooth pointed stable curve $\mcX^{\bp}$ of type $(g_{X}, n_{X})$ over $\spec R$ determined by $c_{R}$. Moreover, we obtain a specialization surjective homomorphism $$sp_{R}: \Pi_{q^{\rm gen}} \migisurj \Pi_{q}.$$ Since $\Pi_{q^{\rm gen}}$ and $\Pi_{q}$ are topologically finitely generated, $sp_{R}$ is an isomorphism.

On the other hand, (ii) implies that there exist a positive integer $m$ prime to $p$ and a connected Galois \'etale covering $\mcY^{\bp} \migi \mcX^{\bp}$ over $R$ with Galois group \tch{$\mbZ/m\mbZ$} such that the geometric generic fiber of $\mcY$ is ordinary and the geometric special fiber of $\mcY$ is not ordinary. Namely, $sp_{R}$ is not an isomorphism.  We complete the proof of (iii).
\end{proof}

\begin{remarkA}
Proposition \ref{coro-3-13} (iii) gives an answer of a question of D. Harbater (\cite[4.2]{H}) which was first solved by Pop and Sa$\rm \ddot{\i}$di (\cite[Corollary]{PS}). In \cite{PS}, Pop and Sa$\rm \ddot{\i}$di proved a result which says that specialization homomorphisms of geometric \'etale fundamental groups of smooth projective curves in positive characteristic is not an isomorphism under certain assumptions. Then together with a result of C-L. Chai and F. Oort, and a result of J-P. Serre, they obtained Proposition \ref{coro-3-13} (iii).  
\end{remarkA}

\subsubsection{}
If $n_{X}\leq 1$, then Proposition \ref{coro-3-13} (i) was proved by Nakajima (\cite[Theorem 2]{N}). Then Proposition \ref{coro-3-13} (i) is a generalized version of Nakajima's result to the case of admissible coverings of smooth pointed stable curves. Moreover, Nakajima (\cite[\S4 Remark]{N}) asked whether or not $$\bigcap_{m \in \mbN \ {\text{\rm s.t.}} \ (m, p)=1}U_{(m, n_{X})}$$ is a non-empty open subset of $M_{g_{X}, n_{X}}$. Then Proposition \ref{coro-3-13} (ii) gives a {\it negative} answer of Nakajima's question. Furthermore, we may ask the following question:
\begin{problem}
(i) Suppose that $X^{\bp}$ is a component-generic pointed stable curve over $k$. Can we find a necessary and sufficient condition that $X^{\bp}$ is $(m, n_{X})$-ordinary for all $m\in \mbN$ prime to $p$?

(ii) Does $$M_{g_{X}, n_{X}}^{\rm cl} \cap (\bigcap_{m \in \mbN \ {\text{\rm s.t.}} \ (m, p)=1}U_{(m, n_{X})})=\emptyset$$ hold for an arbitrary non-negative integer $n_{X}$? Moreover, does $$\overline M_{g_{X}, n_{X}}^{\rm cl} \cap (\bigcap_{m \in \mbN \ {\text{\rm s.t.}} \ (m, p)=1}\overline U_{(m, n_{X})})=\emptyset$$ hold for an arbitrary non-negative integer $n_{X}$?
\end{problem}

\section{Maximum generalized Hasse-Witt invariants for arbitrary curves}\label{sec-arb}

In the present section, we discuss the maximum generalized Hasse-Witt invariants of cyclic admissible coverings for arbitrary pointed stable curves. The main result of this section is Theorem \ref{main-them-2}.

\subsection{Idea} We briefly  explain the idea of our proof of Theorem \ref{main-them-2}.

\subsubsection{\bf Settings} We maintain the notation introduced in \ref{defcurves}.

\subsubsection{An easy case}
Firstly, let us prove an easy case (i.e., $X^{\bp}$ is irreducible) of the main result of the present section.

\begin{proposition}\label{pro-4-2}
Suppose that $X^{\bp}$ is irreducible. Then there exist a positive natural number $n\defeq p^{t}-1 \in \mbN$, an effective divisor $D \in (\mbZ/n\mbZ)^{\sim}[D_{X}]^{0}$ (\ref{2.2.4}) on $X$ of degree $(n_{X}-1)n$ if $n_{X}\neq 0$ (resp. degree $0$ if $n_{X}=0$), and an element $\alpha \in {\rm Rev}^{\rm adm}_{D}(X^{\bp})\setminus \{0\}$ (Definition \ref{def-2-4} (i)) such that $\gamma_{(\alpha, D)}$ attains maximum (Definition \ref{def-3-2} (ii)). Namely, the following holds:  
\begin{eqnarray*}
\gamma_{(\alpha, D)}=\gamma^{\rm max}_{X^{\bp}}=\left\{ \begin{array}{ll}
g_{X}-1, & \text{if} \ n_{X} =0,
\\
g_{X}+n_{X}-2, & \text{if} \ n_{X} \neq 0.
\end{array} \right.
\end{eqnarray*}
\end{proposition}

\begin{proof}
Since $X^{\bp}$ is irreducible, we write  $\widetilde X^{\bp}$ for the smooth pointed stable curve of type $(g_{\widetilde X}, n_{\widetilde X})$ over $k$ associated to the unique vertex of the dual semi-graph $\Gamma_{X^{\bp}}$ of $X^{\bp}$ (\ref{smoothpointed}). Note that we have $$g_{\widetilde X}=g_{X}-\#(X^{\rm sing}), \ n_{\widetilde X}=n_{X}+2\#(X^{\rm sing}).$$ Moreover,  $D_{X} \subseteq D_{\widetilde X}$ implies $(\mbZ/n\mbZ)^{\sim}[D_{X}]^{0} \subseteq (\mbZ/n\mbZ)^{\sim}[D_{\widetilde X}]^{0}$. 

By applying Theorem \ref{lem-3-1}, to verify the proposition, it is sufficient to prove that there exist a positive natural number $n\defeq p^{t}-1 \in \mbN$, an effective divisor $\widetilde D \in (\mbZ/n\mbZ)^{\sim}[D_{X}]^{0}$ of degree $(n_{X}-1)n$ if $n_{X}\neq 0$ (resp. degree $0$ if $n_{X}=0$) on $\widetilde X$, and an element $\widetilde \alpha \in {\rm Rev}^{\rm adm}_{\widetilde D}(\widetilde X^{\bp}) \setminus \{0\}$ such that the following holds:
\begin{eqnarray*}
\gamma_{(\widetilde\alpha, \widetilde D)}=\left\{ \begin{array}{ll}
g_{\widetilde X}-1, & \text{if} \ n_{X} =0,
\\
g_{\widetilde X}+n_{X}-2, & \text{if} \ n_{X} \neq 0.
\end{array} \right.
\end{eqnarray*}


Suppose that $n_{X}\leq 1$. Then the proposition follows immediately from Proposition \ref{maineasy} (ii). 

Suppose that $n_{X} \geq 2$. Let $D_{X}\defeq \{x_{1}, \dots, x_{n_{X}}\}$ and $n_{i}\defeq p^{t_{i}}-1$, $i \in \{1, \dots, n_{X}-1\}$, such that the following conditions are satisfied:
\begin{itemize}
\item $n_{i} > \text{max}\{C(g_{X})+1, \#(e(\Gamma_{X^{\bp}}))\}$.

\item $0<a_{i, 1}, a_{i, 2} < n_{i}$ and $a_{i, 1}+a_{i, 2}=n_{i}$ for all $i \in \{1, \dots, n_{X}-1\}$.
\end{itemize}
Furthermore, we put $$D_{i} \defeq a_{i, 1}x_{i}+a_{i, 2}x_{i+1}+\sum_{x\in D_{X} \setminus \{x_{i}, x_{i+1}\}}n_{i}x, \ i \in \{1, \dots, n_{X}-1\},$$ which is an effective divisor on $\widetilde X$ of degree $(n_{X}-1)n_{i}$. Moreover, we put $$\widetilde D\defeq \sum_{i=1}^{n_{X}-1}p^{\sum_{j=0}^{i-1}t_{j}}D_{i},$$$$n\defeq p^{\sum_{i=0}^{n_{X}-1}t_{j}}-1=\sum_{i=1}^{n_{X}-1}p^{\sum_{j=0}^{i-1}t_{j}}(p^{t_i}-1),$$ where $t_{0}\defeq 0$. We see that $\widetilde D$ is an effective divisor on $\widetilde X$ of degree $(n_{X}-1)n$ such that $\widetilde D \in (\mbZ/n\mbZ)^{\sim}[D_{X}]^{0}$.

Let $\mcL_{\widetilde D}$ be a line bundle on $\widetilde X$ such that $\mcL_{\widetilde D}^{\otimes n} \cong \mcO_{X}(-\widetilde D)$, and $\mcL_{\widetilde D, t}$ the pulling back of $\mcL_{\widetilde D}$ by the natural morphism $\widetilde X_{t} \migi \widetilde X$ defined in \ref{thetabundle}.  Then by applying \cite[Corollary 2.6, Lemma 2.12 (ii), and Corollary 2.13]{T2}, we see that the Raynaud-Tamagawa theta divisor associated to $\mcB^{t}_{\widetilde D}\otimes \mcL_{\widetilde D, t}$ exists (Definition \ref{deftheta}). Moreover, Proposition \ref{prop-theta} implies that there exists a line bundle $\widetilde \mcI$ of degree $0$ on $\widetilde X$ such that $[\widetilde \mcI]\neq [\mcO_{\widetilde X}]$, that $[\widetilde \mcI^{\otimes n}]=[\mcO_{\widetilde X}]$, and that (see \ref{ghwline} for $\gamma_{([\mcL_{\widetilde D}\otimes \widetilde \mcI], \widetilde D)}$)
\begin{eqnarray*}
\gamma_{([\mcL_{\widetilde D}\otimes \widetilde \mcI], \widetilde D)}=\left\{ \begin{array}{ll}
g_{\widetilde X}-1, & \text{if} \ n_{X} =0,
\\
g_{\widetilde X}+n_{X}-2, & \text{if} \ n_{X} \neq 0.
\end{array} \right.
\end{eqnarray*}
Let $\widetilde \alpha \in {\rm Rev}^{\rm adm}_{\widetilde D}(\widetilde X^{\bp})$ be the element corresponding to the pair $([\mcL_{\widetilde D}\otimes \widetilde \mcI], \widetilde D) \in \widetilde \msP_{\widetilde X^{\bp}, n}$ (\ref{line}). Then we have $\gamma_{(\widetilde \alpha, \widetilde D)}=\gamma_{([\mcL_{\widetilde D}\otimes \widetilde \mcI], \widetilde D)}$. This completes the proof of the proposition.
\end{proof}

\begin{remarkA}
We maintain the notation introduced in the proof of Proposition \ref{pro-4-2}. By choosing a suitable $a_{i, 1}$ (or $a_{i, 2}$) for each $i \in \{1, \dots, n_{X}-1\}$, we obtain that the Galois multi-admissible covering induced by $\alpha$ is connected.
\end{remarkA}

\subsubsection{Strategy of the proof of Theorem \ref{main-them-2}}
In the remainder of this section, we will \tch{generalize} Proposition \ref{pro-4-2} to the case where $X^{\bp}$ is an arbitrary pointed stable curve over $k$ (i.e., Theorem \ref{main-them-2} below). 

Let $X^{\bp}$ be an arbitrary pointed stable curve over $k$. For simplicity, we assume that every irreducible component of $X^{\bp}$ is non-singular. Moreover, by Proposition \ref{maineasy} (ii), we assume $n_{X} \geq 2$. We maintain the notation introduced in the statement of Theorem \ref{lem-3-1}. To verify Theorem \ref{main-them-2}, by applying Theorem \ref{lem-3-1}, it is sufficient to construct a prime-to-$p$ cyclic Galois multi-admissible covering $\widetilde f^{\bp}_{v}: \widetilde Y_{v}^{\bp} \migi \widetilde X_{v}^{\bp}$ with Galois group $\mbZ/n\mbZ$ for every $v \in v(\Gamma_{X^{\bp}})$ such that the following conditions are satisfied: 
\begin{itemize}
\item $\gamma_{(\widetilde \alpha_{v}, D_{\widetilde \alpha_{v}})}$ satisfies the conditions mentioned in the statement of Theorem \ref{lem-3-1}.

\item $\{\widetilde f_{v}^{\bp}\}_{v\in v(\Gamma_{X^{\bp}})}$ can be glued together. Then we obtain a prime-to-$p$ cyclic Galois multi-admissible covering $f^{\bp}: Y^{\bp} \migi X^{\bp}$ with Galois group $\mbZ/n\mbZ$.

\item The ramification divisor $D \in (\mbZ/n\mbZ)^{\sim}[D_{X}]^{0}$ associated to $f^{\bp}$ such that $s(D)=n_{X}-1$ (\ref{2.2.4}).
\end{itemize}
If $X^{\bp}$ is {\it not} irreducible, the constructions of  desired Galois multi-admissible coverings of $X^{\bp}$ are very difficult in general.

\subsubsection*{The main difficulty} We {\it cannot determine the ramifications over nodes} of a Galois admissible covering in a unique way when the ramifications over $D_{X}$ are fixed. Then we have the following: 

(i) Let $f^{\bp}$ be a Galois admissible covering whose ramification divisor satisfies the third condition mentioned above. Since the Raynaud-Tamagawa theta divisor concerning $D_{\widetilde \alpha_{v}}$, $v\in v(\Gamma_{X^{\bp}})$, {\it does not exist} in general, $\gamma_{(\widetilde \alpha_{v}, D_{\widetilde \alpha_{v}})}$ does not satisfy the first condition mentioned above.

(ii) Even though we can construct $\widetilde f^{\bp}_{v}$ for every $v \in v(\Gamma_{X^{\bp}})$ (by applying Proposition \ref{pro-4-2}) such that $\gamma_{(\widetilde \alpha_{v}, D_{\widetilde \alpha_{v}})}$ satisfies the conditions mentioned in the statement of Theorem \ref{lem-3-1}, $\{\widetilde f_{v}^{\bp}\}_{v\in v(\Gamma_{X^{\bp}})}$ {\it cannot be glued together (as admissible coverings) in general.} \\


To overcome this difficulty, we observe that, when $\Gamma_{X^{\bp}}$ is a {\it minimal quasi-tree}, the ramifications over nodes of a Galois admissible covering can be uniquely determined  if the ramifications over $D_{X}$ are fixed. Then for every $v\in v(\Gamma_{X^{\bp}})$, we may construct certain effective divisors on $\widetilde X_{v}$ for every marked point of $D_{X} \cap \widetilde X_{v}$ and every node of $X^{\rm sing} \cap \widetilde X_{v}$. Furthermore, by similar arguments to the arguments given in the proof of Proposition \ref{pro-4-2}, we may construct an effective divisor on $\widetilde X_{v}$ such that the Raynaud-Tamagawa theta divisor concerning this effective divisor exists. Then we can obtain Galois multi-admissible coverings for all $v\in v(\Gamma_{X^{\bp}})$ whose first generalized Hasse-Witt invariants attain maximum. On the other hand, since the ramifications over nodes of a Galois admissible covering can be uniquely determined by the ramifications of marked points when $\Gamma_{X^{\bp}}$ is a minimal quasi-tree, then we also obtain an effective divisor on $X$ whose restriction on $\widetilde X_v$, $v\in v(\Gamma_{X^{\bp}})$, is the effective divisor constructed above. Then we may construct a desired Galois multi-admissible covering $f^{\bp}$ (e.g. see Example \ref{examkeylemma} below).  

In the general case (i.e., $\Gamma_{X^{\bp}}$ is not a tree), we take a {\it minimal quasi-tree $\Gamma\defeq \Gamma_{D_{X}}$ associated to $D_{X}$} (\ref{defquasitree}) and the pointed stable curve $X^{\bp}_{\Gamma}$ (\ref{imgraph}) associated to $\Gamma$. Then we may construct a desired Galois multi-admissible covering for $X^{\bp}_{\Gamma}$ (see Lemma \ref{main-lem}). In fact, this is the motivation of the definition of minimal quasi-trees. Moreover, Raynaud's theorem (i.e., \cite[Th\'eor\`eme 4.1.1]{R1} or Theorem \ref{them-theta}) implies that we may construct a Galois (\'etale) multi-admissible covering for each irreducible component corresponding to $w \in v(\Gamma_{X^{\bp}})\setminus v(\Gamma)$ whose first generalized Hasse-Witt invariant (\ref{ghw}) attains maximum. Then by gluing the Galois multi-admissible coverings together, we obtain a desired Galois multi-admissible covering of $X^{\bp}$.

\subsection{A key lemma}

\subsubsection{\bf Settings} We maintain the notation introduced in \ref{defcurves}.

\subsubsection{}
In order to convince the reader to follow the constructions given in the proof of Lemma \ref{main-lem} below, we give an example for constructing effective divisors $D_{\widetilde \alpha_{v}} \in (\mbZ/n\mbZ)^{\sim}[D_{\widetilde X_{v}}]^{0}$, $v\in v(\Gamma_{X^{\bp}})$, and $D\in (\mbZ/n\mbZ)^{\sim}[D_{X}]^{0}$.

\begin{example}\label{examkeylemma}
Let $X^{\bp}$ be a pointed stable curve of type $(g_{X}, 3)$ over $k$. Suppose that $X$ has two non-singular irreducible components $X_{1}$ and $X_{2}$, that $X_{1} \cap X_{2}\defeq \{x^{-}\}$, that $D_{X} \cap X_{1}\defeq \{x_{1, 1}, x_{1, 2}\}$, and that $D_{X} \cap X_{2}\defeq \{x_{2,1}\}$. Let $n_{0}\defeq p^{t_{0}}-1>>0$, $D_{X_{1}}\defeq \{x_{1,1}, x_{1, 2}, x^{-}\}$, $D_{X_{2}} \defeq \{x_{2, 1}, x^{-}\}$, $X^{\bp}_{1}=(X_{1}, D_{X_{1}})$ a pointed stable curve of type $(g_{X_{1}}, 3)$, and $X^{\bp}_{2}=(X_{2}, D_{X_{2}})$ a pointed  stable curve of type $(g_{X_{2}}, 2)$. Then we have the following:

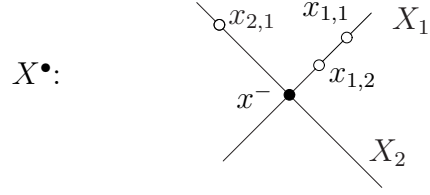
\begin{picture}(300,100)
\put(170,30){\line(5, 5){35}}
\put(225,30){$X_{2}$}
\put(170,80){\line(5, -5){60} $X_{1}$}
\put(168.5,81.5){\circle{4}}
\put(172,82){$x_{2,1}$}
\put(160,90){\line(5,-5){7.2}}
\put(195, 55){\circle*{4}}
\put(175, 50){$x^{-}$}

\put(206.2, 66.2){\circle{4}}
\put(207.5,67.5){\line(1,1){8}}

\put(210, 60){$x_{1,2}$}

\put(216.7, 76.7){\circle{4}}
\put(218,78){\line(1,1){8}}

\put(200, 85){$x_{1,1}$}

\put(168.5,81.5){\circle{4}}

\put(90,59){$X^{\bp}$:}
\end{picture}

{\bf Step 1:} We construct effective divisors of degree $2n$ (resp. $n$) on $X_{1}$ (resp. $X_{2}$) such that the Raynaud-Tamagawa theta divisors concerning the divisors exist. \\

We put $$Q_{1, 1} \defeq a_{1}x_{1,1}+a_{2}x_{1,2}+n_{0}x^{-}\in \text{Div}(X_{1}),$$ $$Q_{1,2}\defeq n_{0}x_{2, 1} \in \text{Div}(X_{2}),$$ $$Q_{1} \defeq a_{1}x_{1,1}+a_{2}x_{1,2}+n_{0}x_{2, 1} \in \text{Div}(X),$$ where $0<a_{1}, a_{2}<n_{0}$ such that $a_{1}+a_{2}=n_{0}$. Moreover, we put $$Q_{2, 1} \defeq n_{0}x_{1,1}+b_{1}x_{1,2}+b_{2}x^{-}\in \text{Div}(X_{1}),$$ $$Q_{2,2}\defeq b_{1}x^{-}+b_{2}x_{2,1}\in \text{Div}(X_{2}),$$ $$Q_{2} \defeq n_{0}x_{1,1}+b_{1}x_{1,2}+b_{2}x_{2,1}\in \text{Div}(X),$$ where $0<b_{1}, b_{2}<n_{0}$ such that $b_{1}+b_{2}=n_{0}$. 

Since $n_{0}>>0$, by applying \cite[Corollary 2.6, Lemma 2.12 (ii), and Corollary 2.13]{T2}, the Raynaud-Tamagawa theta divisors concerning $Q_{1}$, $Q_{2}$, $Q_{1, 1}$, $Q_{1, 2}$, $Q_{2,1}$, $Q_{2, 2}$ exist, respectively.  

Note that since $\Gamma_{X^{\bp}}$ is a {\it tree}, $Q_{1,1}$ and $Q_{1, 2}$ (resp. $Q_{2,1}$ and $Q_{2, 2}$) can be  completely determined by $Q_{1}$ (resp. $Q_{2}$). On the other hand, at present, we {\it cannot} construct Galois multi-admissible coverings whose ramification divisors are the above divisors since $Q_{1}$, $Q_{1, 1}$, $Q_{1, 2}$ (resp. $Q_{2}$, $Q_{2, 1}$, $Q_{2, 2}$) {\it are not contained in} $(\mbZ/n\mbZ)^{\sim}[D_{X_{1}}]^{0}$ (resp. $(\mbZ/n\mbZ)^{\sim}[D_{X_{2}}]^{0})$. \\

{\bf Step 2:} By using the effective divisors constructed in Step 1, we construct a global effective divisor $D\in (\mbZ/n\mbZ)^{\sim}[D_{X}]^{0}$ on $X$ and $D_{i} \in (\mbZ/n\mbZ)^{\sim}[D_{X_{i}}]^{0}$, $i \in \{1, 2\}$, on $X_{i}$ such that $D_{i}$ is completely determined by $D$. \\

Write $v_{1}$ and $v_{2}$ for the vertices of $v(\Gamma_{X^{\bp}})$ corresponding to $X_{1}$ and $X_{2}$, respectively. We define the following sets of effective divisors $$\text{Div}_{v_{i}}^{\rm irr\text{-}mp} \defeq \{Q_{1, i}\}, \ \text{Div}_{v_{i}}^{\rm irr\text{-}nd}\defeq \{Q_{2, i}\},$$ $$\text{Div}_{X}^{\rm irr}\defeq \bigsqcup_{i=1, 2}(\text{Div}_{v_{i}}^{\rm irr\text{-}mp} \sqcup \text{Div}_{v_{i}}^{\rm irr\text{-}nd}),\ \text{Div}_{X}\defeq \{Q_{1}, Q_{2}\},$$ where $\sqcup$ means disjoint union.  Let $n\defeq p^{t^{0}}-1+ p^{t_{0}}(p^{t_{0}}-1)=p^{2t_{0}}-1$. We obtain $$D_{1} \defeq Q_{1,1}+p^{t_{0}}Q_{2,1} \in (\mbZ/n\mbZ)^{\sim}[D_{X_{1}}]^{0},$$ $$D_{2} \defeq Q_{1,2}+p^{t_{0}}Q_{2,2} \in (\mbZ/n\mbZ)^{\sim}[D_{X_{2}}]^{0},$$ $$D\defeq Q_{1}+p^{t_{0}}Q_{2} \in (\mbZ/n\mbZ)^{\sim}[D_{X}]^{0}.$$ Then we have $$D|_{\{x_{1,1}, x_{1, 2}\}}=D_{1}|_{\{x_{1,1}, x_{1, 2}\}}, \ D|_{\{x_{2,1}\}}=D_{2}|_{\{x_{2,1}\}}.$$ \\

{\bf Step 3:} We construct a desired Galois multi-admissible covering of $X^{\bp}$. \\

By applying \cite[Corollary 2.6, Lemma 2.12 (ii), and Corollary 2.13]{T2}, the Raynaud-Tamagawa theta divisors concerning $D_{1}$ and $D_{2}$ exist. By Proposition \ref{prop-theta}, if $t_{0}>>0$, then there exists $\alpha_{v_{i}} \in \text{Rev}_{D_{i}}^{\rm adm}(X_{i}^{\bp}) \setminus \{0\}$, $i\in \{1, 2\}$, such that $\gamma_{(\alpha_{v_{1}}, D_{1})}=g_{X_{1}}+1$ and $\gamma_{(\alpha_{v_{1}}, D_{1})}=g_{X_{2}}$. Moreover, since $\Gamma_{X^{\bp}}$ is a tree, we have that $$D_{1}|_{x^{-}}=[\sum_{x\in D_{X}\setminus \{x_{1,1}, x_{1, 2}\}}\text{ord}_{x}(D)]x^{-}=\text{ord}_{x_{2,1}}(D)x^{-},$$ $$D_{2}|_{x^{-}}=\text{ord}_{x^{-}}(D_{2})x^{-}=(n-\text{ord}_{x_{2,1}}(D_{2}))x^{-}=[\sum_{x\in D_{X}\setminus \{x_{2,1}\}}\text{ord}_{x}(D)]x^{-}$$$$=(n-\text{ord}_{x_{2,1}}(D))x^{-},$$ where $[(-)]$ denotes the image of $(-)$ in $\mbZ/n\mbZ$. Namely, the Galois multi-admissible coverings induced by $\alpha_{v_{1}}$ and $\alpha_{v_{2}}$ {\it can be glued} (as admissible coverings). Then we obtain a  desired Galois multi-admissible covering of $X^{\bp}$.
\end{example}

\subsubsection{}\label{path}
Let $\mathbf{G}$ be a connected semi-graph and $v \in v(\mathbf{G})$ an arbitrary vertex. Moreover, we suppose that $\mathbf{G}$ is a {\it tree}. For each $v' \in v(\mathbf{G})$, there exists a path $\delta_{v, v'}$ connecting $v$ and $v'$ in $\mathbf{G}$. We define $$\text{leng}(\delta_{v, v'})\defeq \#\{\delta_{v, v'} \cap v(\mathbf{G})\}-1$$ to be the {\it length} of the path $\delta_{v, v'}$. Moreover, since $\mathbf{G}$ is a tree, there exists a unique path connecting $v$ and $v'$ whose length is equal to $\text{min}\{\text{leng}(\delta_{v,v'})\}_{\delta_{v,v'}}$. We shall write $$\delta(\mathbf{G}, v, v')$$ for this unique path connecting $v$ and $v'$ in $\mathbf{G}$, and call $\delta(\mathbf{G}, v, v')$  the {\it minimal path} connecting $v$ and $v'$ in $\mathbf{G}$.

\subsubsection{}
Now, we are going to prove the key lemma of the present section.

\begin{lemma}\label{main-lem}
Let $\Gamma\defeq \Gamma_{D_{X}}$ be a minimal quasi-tree associated to $D_{X}$, $$X_{\Gamma}^{\bp}=(X_{\Gamma}, D_{X_{\Gamma}})$$ the pointed stable curve of type $(g_{X_{\Gamma}}, n_{X_{\Gamma}})$ associated to $\Gamma$ (\ref{defquasitree}), and $\Pi_{X^{\bp}_{\Gamma}}$ the admissible fundamental group of $X^{\bp}_{\Gamma}$. Suppose that $n_{X}\geq 2$. Then there exist a positive natural number $n\defeq p^{t}-1 \in \mbN$, an \tch{effective} divisor $D_{\Gamma}\in (\mbZ/n\mbZ)^{\sim}[D_{X}]^{0} \subseteq (\mbZ/n\mbZ)^{\sim}[D_{X_{\Gamma}}]^{0}$ on $X_{\Gamma}$, and an element $\alpha_{\Gamma} \in {\rm Rev}^{\rm adm}_{D_{\Gamma}}(X_{\Gamma}^{\bp}) \setminus \{0\}$ such that the following holds:
$$\gamma_{(\alpha_{\Gamma}, D_{\Gamma})}
=g_{X_{\Gamma}}+n_{X}-2.$$
\end{lemma}

\begin{proof}
Since $\Gamma$ is a minimal quasi-tree associated to $D_{X}$, we obtain that $\Gamma' \defeq \Gamma \setminus e^{\rm lp}(\Gamma)$ is a tree. Then we have $v(\Gamma)=v(\Gamma')$. Note that $D_{X} \subseteq D_{X_{\Gamma}}$. Let $v \in v(\Gamma)$ be an arbitrary vertex and $n_{0}=p^{t_{0}}-1 \in \mbN$ a positive natural number satisfying $$n_{0}>\text{max}\{C(g_{X})+1, \#(e(\Gamma_{X^{\bp}}))\}.$$ \\

{\bf Step 1 (mp):} Let $v \in v(\Gamma)$. We construct a family of effective divisors $\text{Div}_{v}^{\rm irr\text{-}mp}$ on the irreducible component $X^{\bp}_{v}$ associated to the set of marked points $D_{X} \cap X_{v}$, and construct a family of effective divisors $\text{Div}_{v}^{\rm mp}$ on $X_{\Gamma}$, where ``mp" means ``marked point". \\ 

We put $$D'_{v}\defeq D_{X} \cap X_{v}, \  m_{v}\defeq \#(D'_{v}), \ \text{and} \ D'_{v} \defeq \{x_{v, 1},\dots, x_{v, m_{v}}\} \ \text{if} \ m_{v}\neq 0.$$ Note that $D_{v}'\neq D_{X_{\Gamma}} \cap X_{v}$ in general. Moreover, we put $$D_{v}\defeq D'_{v} \cup (X_{v} \cap (\overline {X_{\Gamma} \setminus X_{v}})),$$ where $\overline {X_{\Gamma} \setminus X_{v}}$ denotes the topological closure of $X_{\Gamma} \setminus X_{v}$ in $X_{\Gamma}$. Note that since $n_{X}>0$, we have $\#(D_{v})>0$. Let $w \in v(\Gamma) \setminus \{v\}$ be an arbitrary vertex distinct from $v$. Since $\Gamma'$ is a tree, there exists a unique node $$x_{v, w}^{-} \in D_{w} \subseteq X_{w}$$ such that the closed edge of $\Gamma'$ corresponding to $x_{v, w}^{-}$ is contained in the minimal path $\delta(\Gamma', v, w)$ (i.e., the minimal path connecting $v$ and $w$ in $\Gamma'$ defined in \ref{path}). On the other hand, we define a set of nodes to be $$\text{Node}_{v, w}^{+}\defeq \{X_{w} \cap X_{w'}, \ w' \in v(\Gamma) \ | \ \text{leng}(\delta(\Gamma', v, w'))=\text{leng}(\delta(\Gamma', v, w))+1\}.$$ Note that $\text{Node}_{v, w}^{+}$ may possibly be an empty set, and that $D_{w}=\{x^{-}_{v, w}\} \cup \text{Node}^{+}_{v, w} \cup D'_{w}.$

We define two sets of effective divisors $$\text{Div}_{v}^{\text{irr-mp}}, \ \text{Div}_{v}^{\text{mp}}$$ associated to $v$ on $X_{v}$ and $X_{\Gamma}$, as follows. Let $i \in \{1, \dots, m_{v}-1\}$ and $0 <a_{v, i, 1}, a_{v, i, 2}<n_{0}$ such that $a_{v, i, 1}+a_{v, i, 2}=n_{0}$. Suppose that $m_{v}\leq 1$. Then we put $$\text{Div}_{v}^{\text{irr-mp}}\defeq \emptyset, \ \text{Div}_{v}^{\text{mp}}\defeq \emptyset.$$ Suppose that $m_{v}\geq 2$. We define $$Q_{v, v, i}\defeq a_{v, i, 1}x_{v, i}+a_{v, i, 2}x_{v, i+1}+\sum_{x' \in D'_{v} \setminus \{x_{v,i}, x_{v, i+1}\}}n_{0}x'+\sum_{x \in D_{v} \setminus D'_{v}}n_{0}x,  \ i \in \{1, \dots, m_{v}-1\},$$ to be an effective divisor on $X_{v}$ whose support is $D_{v}$, and whose degree is equal to $(\#(D_{v})-1)n_{0}$. We define $$Q_{v, w, i}\defeq \sum_{x \in D_{w}\setminus \{x_{v, w}^{-}\}}n_{0}x, \ w\in v(\Gamma) \setminus \{v\},$$ to be an effective divisor on $X_{w}$ whose support is $D_{w}\setminus \{x_{v, w}^{-}\}$, and whose degree is equal to $(\#(D_{w})-1)n_{0}$. 

Moreover, we define $$Q^{v}_{i} \defeq a_{v, i, 1}x_{v, i}+a_{v, i, 2}x_{v, i+1}+\sum_{x \in D_{X} \setminus \{x_{v, i}, x_{v, i+1}\}}n_{0}x,$$ to be an effective divisor on $X_{\Gamma}$ whose support is $D_{X}$, and whose degree is $(n_{X}-1)n_{0}$. Note that $Q_{i}^{v}|_{D_{u}'}=Q_{v,u,i}|_{D_{u}'}$ for all $u\in v(\Gamma)$, and that  $Q_{v, w, i}=Q_{v, w, i'}$ for all $i, i' \in \{1, \dots, m_{v}-1\}$. 

We put
$$\text{Div}_{v, i}^{\text{irr-mp}}\defeq \bigsqcup_{u\in v(\Gamma)}\{Q_{v, u, i}\}, \ \text{Div}_{v}^{\text{irr-mp}}\defeq \bigsqcup_{i=1}^{m_{v}-1}\text{Div}_{v, i}^{\text{irr-mp}}, \ \text{Div}_{v}^{\text{mp}}\defeq \bigsqcup_{i=1}^{m_{v}-1} \{Q_{i}^{v}\},$$ where $\sqcup$ means disjoint union. \\

{\bf Step 1 (nd):} Let $v \in v(\Gamma)$. We  construct a family of effective divisors $\text{Div}_{v}^{\rm irr\text{-}nd}$ on the irreducible component $X^{\bp}_{v}$ associated to the set of nodes $D_{v} \setminus D'_{v}=X_{v} \cap (\overline {X_{\Gamma} \setminus X_{v}})$, and construct a family of effective divisors $\text{Div}_{v}^{\rm nd}$ on $X_{\Gamma}$, where ``nd" means ``node". \\ 

We define two families of effective divisors $$\text{Div}_{v}^{\text{irr-nd}}, \ \text{Div}_{v}^{\text{nd}}$$ associated to $v$ on $X_{v}$ and $X_{\Gamma}$, respectively, as follows. Let $z \in D_{X} \setminus D'_{v}$ and $0 <b_{v, z, 1}, b_{v, z, 2}<n_{0}$ such that $b_{v, z, 1}+b_{v, z, 2}=n_{0}$. Suppose that $m_{v}=0$. Then we put $$\text{Div}_{v}^{\text{irr-nd}}\defeq \emptyset, \ \text{Div}_{v}^{\text{nd}}\defeq \emptyset.$$ Suppose that $m_{v}\neq 0$. Let $w_{z}$ be the vertex such that the irreducible component $X_{w_{z}}$ corresponding to $w_{z}$ contains $z$ (i.e., $z\in D'_{w_{z}}\defeq D_{X} \cap X_{w_{z}}$). Note that $w_{z} \neq v$. Moreover, let  $\delta(\Gamma', v, w_{z})$ be the minimal path connecting $v$ and $w_{z}$ in $\Gamma'$ and $w \in v(\Gamma)$ an arbitrary vertex distinct from $w_{z}$ such that $w \subseteq \delta(\Gamma', v, w_{z})$. Since $\Gamma'$ is a tree, we have  $\#(\text{Node}_{v, w}^{+} \cap \delta(\Gamma', v, w_{z}))=1$. Then we put $$x_{v, w, w_{z}}^{+}\defeq \text{Node}_{v, w}^{+} \cap \delta(\Gamma', v, w_{z}) \in D_{w} \subseteq X_{w}.$$ 

For $v$ and $w_{z}$, we define $$Q_{v, v, z}\defeq b_{v, z, 1}x_{v, m_{v}}+b_{v, z, 2}x_{v, v, w_{z}}^{+}+\sum_{x \in D_{v} \setminus \{x_{v, m_{v}}, x_{v,v, w_{z}}^{+}\}}n_{0}x,$$$$Q_{v, w_z, z}\defeq b_{v, z, 1}x_{v, w_{z}}^{-}+b_{v, z, 2}z+\sum_{x \in D_{w_{z}} \setminus \{x_{v,w_{z}}, z\}}n_{0}x$$ to be effective divisors on $X_{v}$ and $X_{w_{z}}$ whose supports are $D_{v}$ and $D_{w_{z}}$, and whose degrees are equal to $(\#(D_{v})-1)n_{0}$ and $(\#(D_{w_{z}})-1)n_{0}$, respectively. 

Let $w \in v(\Gamma)\setminus \{v, w_{z}\}$ be an arbitrary vertex such that $w \subseteq \delta(\Gamma', v, w_{z})$. Then we define $$Q_{v, w, z}\defeq b_{v, z, 1}x_{v, w}^{-}+b_{v, z, 2}x_{v, w, w_{z}}^{+}+\sum_{x \in D_{w} \setminus \{x_{v, w}^{-},   x_{v, w, w_{z}}^{+}\}}n_{0}x$$ to be an effective divisor on $X_{w}$ whose support is $D_{w}$, and whose degree is equal to $(\#(D_{w})-1)n_{0}$. 

Let $w' \in v(\Gamma)$ be an arbitrary vertex such that $w' \not\subseteq \delta(\Gamma', v, w_{z})$. Then we define $$Q_{v, w', z}\defeq \sum_{x\in D_{w'}\setminus \{x_{v, w'}^{-}\}}n_{0}x$$ to be an effective divisor on $X_{w'}$ whose support is $D_{w'}\setminus \{x_{v, w'}^{-}\}$, and whose degree is equal to $(\#(D_{w'})-1)n_{0}$. Note that, if $w'' \not\subseteq \delta(\Gamma', v, w_{z}) \cup \delta(\Gamma', v, w_{z'})$ for $z, z' \in D_{X}\setminus D_{v}'$, we have $Q_{v, w'', z}=Q_{v, w'', z'}$.

Moreover, we define $$Q_{z}^{v}\defeq b_{v,z,1}x_{v, m_{v}}+b_{v,z,2}z+\sum_{x\in D_{X}\setminus \{x_{v, m_{v}}, z\}}n_{0}x$$ to be an effective divisor on $X_{\Gamma}$ whose support is $D_{X}$, and whose degree is equal to $(n_{X}-1)n_{0}$. Note that $Q_{z}^{v}|_{D_{u}'}=Q_{v,u,z}|_{D_{u}'}$ for all $u \in v(\Gamma)$.

We put $$\text{Div}_{v, z}^{\text{irr-nd}}\defeq \bigsqcup_{u \in v(\Gamma)}\{Q_{v,u,z}\}, \ \text{Div}_{v}^{\text{irr-nd}}\defeq \bigsqcup_{z \in D_{X} \setminus D'_{v}}\text{Div}_{v, z}^{\text{irr-nd}}, \ \text{Div}_{v}^{\text{nd}}\defeq  \bigsqcup_{z \in D_{X} \setminus D'_{v}}\{Q_{z}^{v}\}.$$ \\

{\bf Step 2:} We construct an effective divisor $P_{u} \in (\mbZ/n\mbZ)^{\sim}[D_{u}]^{0}$, $v\in v(\Gamma)$, on $X_{u}$ and an effective divisor $P_{\Gamma} \in (\mbZ/n\mbZ)^{\sim}[D_{X_\Gamma}]^{0}$ on $X_{\Gamma}$ for some $n$. \\ 

We put $$\text{Div}^{\text{irr}}_{X}\defeq \bigsqcup_{v \in v(\Gamma)}(\text{Div}_{v}^{\text{irr-mp}} \sqcup \text{Div}_{v}^{\text{irr-nd}}),$$ $$\text{Div}_{X}\defeq \bigsqcup_{v \in v(\Gamma)}(\text{Div}_{v}^{\text{mp}} \sqcup \text{Div}_{v}^{\text{nd}}).$$ We denote by $\text{Div}^{\text{irr}}_{X}(X_{u})\subseteq \text{Div}_{X}^{\text{irr}}$, $u \in v(\Gamma)$, the subset  whose elements are effective divisors on $X_{u}$. Note that the above constructions imply $d\defeq \#(\text{Div}_{X}^{\text{irr}}(X_{u_{1}}))=\#(\text{Div}_{X}^{\text{irr}}(X_{u_{2}}))=\#(\text{Div}_{X})$ for all $u_{1}, u_{2} \in v(\Gamma)$. Moreover, let $$o_{u}: \{1, \dots, d\} \isom \text{Div}_{X}^{\text{irr}}(X_{u})\defeq \{P_{u, 1}\defeq o_{u}(1), \dots, P_{u, d}\defeq o_{u}(d)\}, \ u \in v(\Gamma),$$ be a bijective (as sets) such that, for all $u_{1}, u_{2} \in v(\Gamma)$ and all $j \in \{1, \dots, d\}$, the following conditions are satisfied: 
\begin{itemize}
\item if $P_{u_{1}, j} \in \text{Div}_{v, i}^{\text{irr-mp}}$ for some $v \in v(\Gamma)$ and some $i \in \{1, \dots, m_{v}-1\}$, then $P_{u_{2}, j} \in \text{Div}_{v, i}^{\text{irr-mp}}$.

\item If $P_{u_{1}, j} \in \text{Div}_{v, z}^{\text{irr-nd}}$ for some $v \in v(\Gamma)$ and some $z \in D_{X} \setminus D_{v}'$, then $P_{u_{2}, j} \in \text{Div}_{v, z}^{\text{irr-nd}}$. 
\end{itemize}
Then, by the construction of $\text{Div}_{X}$, we obtain a bijection $$o: \{1, \dots, d\} \isom \text{Div}_{X}\defeq \{P_{1}\defeq o(1), \dots, P_{d}\defeq o(d)\}$$ induced by $o_{u}$, $u\in v(\Gamma)$.

Let $t\defeq dt_{0}$ and $n\defeq \sum_{j=1}^{d}p^{(j-1)t_{0}}(p^{t_{0}}-1)=p^{t}-1$. We define $$P_{u}\defeq \sum_{j=1}^{d}p^{(j-1)t_{0}}P_{u, j} \in (\mbZ/n\mbZ)^{\sim}[D_{u}]^{0}, \ u\in v(\Gamma),$$  $$P_{\Gamma}\defeq \sum_{j=1}^{d}p^{(j-1)t_{0}}P_{j} \in (\mbZ/n\mbZ)^{\sim}[D_{X}]^{0}$$ to be effective divisors of degrees $\text{deg}(P_{u})=(\#(D_{u})-1)n$ and  $\text{deg}(P_{\Gamma})=(n_{X}-1)n$ on $X_{u}$ and $X_{\Gamma}$, respectively. We see that the support of $P_{u}$, $u \in v(\Gamma)$, is $D_{u}$, and that the support of $P_{\Gamma}$ is $D_{X}$. 

Let $u \in v(\Gamma)$ and $x'\in D'_{u}$. Then the above constructions implies $P_{\Gamma}|_{x'}=P_{u}|_{x'}$. Moreover, let $x\in D_{u}\setminus D_{u}'$. Since $\Gamma \setminus \{e^{\rm lp}(\Gamma)\}$ is a tree, $X_{\Gamma} \setminus \{x\}$ has two connected components $C_{1}$, $C_{2}$. Let $C_{2}$ be the connected component such that $X_{u}\setminus \{x\}$ is not contained in $C_{2}$. We denote by $C_{x} \defeq \overline {\{C_{2}\}}$ the topological closure of $C_{2}$ in $X_{\Gamma}$ and $M_{x}\defeq C_{x} \cap D_{X}$. Then the above constructions imply $$P_{u}|_{x}=\text{ord}_{x}(P_{u})x = [\sum_{x' \in M_{x}}\text{ord}_{x'}(P_{\Gamma})]x,$$ where $[(-)]$ denotes the image of $(-)$ in $\mbZ/n\mbZ$.\\

{\bf Step 3:} We construct a desired Galois multi-admissible covering of $X^{\bp}$. \\

Let $u\in v(\Gamma)$ and $\widetilde X^{\bp}_{u}$ the smooth pointed stable curve of type $(g_{u}, n_{u})$ over $k$ (\ref{smoothpointed}). Write $\text{norm}_{u}: \widetilde X_{u} \migi X_{u}$ for the normalization morphism. Then we obtain the pulling back divisor $\text{norm}_{u}^{*}(P_{u})$ on $\widetilde X_{u}$. Note that since $\text{Supp}(P_{u})$ is contained in the smooth locus of $X_{u}$, we have $\text{norm}_{u}^{*}(P_{u})=P_{u}$. Let $\mcL_{P_{u}}$ be a line bundle on $\widetilde X_{u}$ such that $\mcL_{P_{u}}^{\otimes n} \cong \mcO_{\widetilde X_{u}}(-P_{u})$, and $\mcL_{P_{u}, t}$ the pulling back of $\mcL_{P_{u}}$ by the natural morphism $\widetilde X_{u, t} \migi \widetilde X$ defined in \ref{thetabundle}.  Then by applying \cite[Corollary 2.6, Lemma 2.12 (ii), and Corollary 2.13]{T2}, we see that the Raynaud-Tamagawa theta divisor associated to $\mcB^{t}_{P_{u}}\otimes \mcL_{P_{u}, t}$ exists (Definition \ref{deftheta}). Moreover, Proposition \ref{prop-theta} implies that there exists $\widetilde \alpha_{u} \in \text{Rev}^{\rm adm}_{ P_{u}}(\widetilde X_{u}^{\bp})$ such that 
$$\gamma_{(\widetilde \alpha_{u}, P_{u})}=g_{u}+\#(D_{u})-2.$$
Then Theorem \ref{lem-3-1} implies that the element $\alpha_{u} \in \text{Rev}^{\rm adm}_{P_{u}}(X_{u}^{\bp})$ induced by $\widetilde \alpha_{u}$ such that the following holds: $$\gamma_{(\alpha_{u}, P_{u})}=g_{X_{u}}+\#(D_{u})-2,$$
where $g_{X_{u}}$ denotes the genus of $X_{u}$. Write $$f_{u}^{\bp}: Y_{u}^{\bp} \migi X_{u}^{\bp}$$ for the Galois multi-admissible covering over $k$ with Galois group $\mbZ/n\mbZ$ induced by $\alpha_{u}$. 

Let $u' \in v(\Gamma) \setminus \{u\}$ such that $ X_{u}\cap X_{u'} \neq \emptyset$. We denote by $x_{u, u'} \defeq X_{u}\cap X_{u'}$ the unique node. Then the above constructions imply $$0 <\text{ord}_{x_{u, u'}}(P_{u}), \text{ord}_{x_{u, u'}}(P_{u'})<n, \ \text{ord}_{x_{u, u'}}(P_{u})+\text{ord}_{x_{u, u'}}(P_{u'})=n.$$ This means that we may glue (as admissible coverings) $\{Y_{u}^{\bp}\}_{u\in v(\Gamma)}$ along $\{f^{-1}_{u}(D_{u} \setminus D'_{u})\}_{u\in v(\Gamma)}$ \tch{in a way} that is compatible with the gluing of $\{X_{u}^{\bp}\}_{u\in v(\Gamma)}$ that gives rise to $X_{\Gamma}^{\bp}$. Then we obtain a Galois multi-admissible covering $$f_{\Gamma}^{\bp}: Y_{\Gamma}^{\bp} \migi X_{\Gamma}^{\bp}$$ over $k$ with Galois group $\mbZ/n\mbZ$. Note that the construction of $f_{\Gamma}^{\bp}$ implies that $f_{\Gamma}$ is \'etale over $D_{X_\Gamma} \setminus D_{X}$.

Let $\Pi_{X^{\bp}_{u}}$, $u\in v(\Gamma)$, be the admissible fundamental group of $X_{u}^{\bp}$. 
We denote by $\alpha_{\Gamma} \in \text{Hom}(\Pi^{\rm ab}_{X_{\Gamma}^{\bp}}, \mbZ/n\mbZ)$ an element induced by $f_{\Gamma}^{\bp}$ such that the composition of the natural homomorphisms $\Pi_{X_{u}^{\bp}}^{\rm ab} \migi \Pi_{X_{\Gamma}^{\bp}}^{\rm ab} \overset{\alpha_{\Gamma}} \migi \mbZ/n\mbZ$ is equal to $\alpha_{u}$ for all $u\in v(\Gamma)$. We put $D_{\Gamma}\defeq P_{\Gamma}$. Then we see $\alpha_{\Gamma} \in \text{Rev}^{\rm adm}_{D_{\Gamma}}(X_{\Gamma}^{\bp}) \setminus \{0\}$. Theorem \ref{lem-3-1} implies $$\gamma_{(\alpha_{\Gamma}, D_{\Gamma})}=g_{X_{\Gamma}}+s(P_{\Gamma})-1=g_{X_{\Gamma}}+n_{X}-2.$$ We complete the proof of the lemma.
\end{proof}

\subsection{The second main theorem}
Now, we prove the second main result of the present paper.

\begin{theorem}\label{main-them-2}
Let $X^{\bp}$ be an arbitrary pointed stable curve of type $(g_{X}, n_{X})$ over an algebraically closed field $k$ of characteristic $p>0$. Then there exist a positive natural number $n\defeq p^{t}-1 \in \mbN$, an effective divisor $D \in (\mbZ/n\mbZ)^{\sim}[D_{X}]^{0}$ (\ref{2.2.4}) on $X$ of degree $(n_{X}-1)n$ if $n_{X}\neq 0$ (resp. degree $0$ if $n_{X}=0$), and an element $\alpha \in {\rm Rev}^{\rm adm}_{D}(X^{\bp})\setminus \{0\}$ (Definition \ref{def-2-4} (i)) such that $\gamma_{(\alpha, D)}$ attains maximum (Definition \ref{def-3-2} (ii)). Namely, the following holds (see Definition \ref{def-3-2} (i) for $\gamma^{\rm max}_{X^{\bp}}$):  
\begin{eqnarray*}
\gamma_{(\alpha, D)}=\gamma^{\rm max}_{X^{\bp}}=\left\{ \begin{array}{ll}
g_{X}-1, & \text{if} \ n_{X} =0,
\\
g_{X}+n_{X}-2, & \text{if} \ n_{X} \neq 0.
\end{array} \right.
\end{eqnarray*}
\end{theorem}

\begin{proof}
Suppose that $n_{X} \leq 1$. Then the theorem follows from Proposition \ref{maineasy} (ii). To verify the theorem, we may assume $n_{X} \geq 2$.

 Let $\Gamma\defeq \Gamma_{D_{X}}$ be a minimal quasi-tree associated to $D_{X}$, $\Gamma^{\rm im}$ the image of the natural morphism $\phi_{\Gamma}: \Gamma \migi \Gamma_{X^{\bp}}$,  $X^{\bp}_{\Gamma}=(X_{\Gamma}, D_{X_{\Gamma}})$, $ X_{\Gamma^{\rm im}}^{\bp}=(X_{\Gamma^{\rm im}}, D_{X_{\Gamma^{\rm im}}})$ the pointed stable curves over $k$ of types $(g_{X_{\Gamma}}, n_{X_{\Gamma}})$, $(g_{X_{\Gamma^{\rm im}}}, n_{X_{\Gamma^{\rm im}}})$ associated to $\Gamma$, $\Gamma^{\rm im}$, respectively (see \ref{defquasitree}, \ref{imgraph} for the definitions of  $\Gamma^{\rm im}$, $\phi_{\Gamma}$, $X^{\bp}_{\Gamma}$, and $X^{\bp}_{\Gamma^{\rm im}}$), and $\Pi_{X^{\bp}_{\Gamma}},$ $\Pi_{X^{\bp}_{\Gamma^{\rm im}}}$ the admissible fundamental groups of $X^{\bp}_{\Gamma}$, $X^{\bp}_{\Gamma^{\rm im}}$, respectively. 

Lemma \ref{main-lem} implies that there exist a natural number $$n\defeq p^{t}-1 >\text{max}\{C(g_{X})+1, \#(e(\Gamma_{X^{\bp}}))\},$$  an effective divisor $D\defeq D_{\Gamma} \in  (\mbZ/n\mbZ)^{\sim}[D_{X}]^{0} \subseteq (\mbZ/n\mbZ)^{\sim}[D_{X_{\Gamma}}]^{0}$ on $X_{\Gamma}$ of degree $(n_{X}-1)n$, and an element $\alpha_{\Gamma}\in \text{Rev}_{D}^{\rm adm}(X^{\bp}_{\Gamma}) \setminus \{0\}$ such that the following holds: $$\gamma_{(\alpha_{\Gamma}, D)}=g_{X_\Gamma}+n_{X}-2.$$  We denote by $f_{\Gamma}^{\bp}: Z_{\Gamma}^{\bp} \migi X^{\bp}_{\Gamma}$ the Galois multi-admissible covering over $k$ with Galois group $\mbZ/n\mbZ$ induced by $\alpha_{\Gamma}$. Note that $f_{\Gamma}$ is \'etale over $D_{X_\Gamma}\setminus D_{X}$. By gluing $Z_{\Gamma}^{\bp}$ along $$f_{\Gamma}^{-1}(D_{X_\Gamma}\setminus (D_{X} \cup \{x_{e}\}_{e \in \phi^{-1}_{\Gamma}(e^{\rm op}(\Gamma^{\rm im}))}))$$ in a way that is compatible with the gluing of $X_{\Gamma}^{\bp}$ that gives rise to $X_{\Gamma^{\rm im}}^{\bp}$, we obtain a pointed stable curve $Z_{\Gamma^{\rm im}}^{\bp}$ over $k$.  Moreover, $f_{\Gamma}^{\bp}$ induces a Galois multi-admissible covering
$$f_{\Gamma^{\rm im}}^{\bp}: Z_{\Gamma^{\rm im}}^{\bp} \migi X^{\bp}_{\Gamma^{\rm im}}$$ over $k$ with Galois group $\mbZ/n\mbZ$. Write $\alpha_{\Gamma^{\rm im}}$ for an element of $\text{Hom}(\Pi^{\rm ab}_{X^{\bp}_{\Gamma^{\rm im}}}, \mbZ/n\mbZ)$ induced by $f^{\bp}_{\Gamma^{\rm im}}$ such that the composition of the natural homomorphisms $\Pi_{X^{\bp}_{\Gamma}}^{\rm ab} \migi \Pi^{\rm ab}_{X^{\bp}_{\Gamma^{\rm im}}} \overset{\alpha_{\Gamma^{\rm im}}}\migi \mbZ/n\mbZ$ is equal to $\alpha_{\Gamma}$. Note that we have $D_{\alpha_{\Gamma^{\rm im}}}=D$. Then Theorem \ref{lem-3-1} implies $$\gamma_{(\alpha_{\Gamma^{\rm im}}, D)}=g_{X_{\Gamma^{\rm im}}}+n_{X}-2.$$

On the other hand, we write $\pi_{0}(\overline {X \setminus X_{\Gamma^{\rm im}}})$ for the set of connected components of the topological closure $\overline {X \setminus X_{\Gamma^{\rm im}}}$ of $X \setminus X_{\Gamma^{\rm im}}$ in $X$. We define the following pointed stable curve $$E^{\bp}=(E, D_{E}\defeq E \cap X_{\Gamma^{\rm im}}), \ E \in \pi_{0}(\overline {X \setminus X_{\Gamma^{\rm im}}}),$$ over $k$. 
Proposition \ref{maineasy} (ii) implies that there exists a Galois {\it \'etale} covering $$f^{\bp}_{E}: Z^{\bp}_{E}=(Z_{E}, D_{Z_{E}}) \migi E^{\bp}$$ over $k$ with Galois group $\mbZ/n\mbZ$ such that the following holds:
\begin{eqnarray*}
\gamma_{(\alpha_{E}, 0)}=\left\{ \begin{array}{ll}
g_{E}, & \text{if} \ g_{E} =0,
\\
g_{X}-1, & \text{if} \ g_{E} \neq 0,
\end{array} \right.
\end{eqnarray*}
where $g_{E}$ denotes the genus of $E$, and $\alpha_{E} \in \text{Rev}^{\rm adm}_{0}(E^{\bp})$ is an element induced by $f^{\bp}_{E}$.




We may glue $Z_{\Gamma^{\rm im}}^{\bp}$ and $\{Z_{E}^{\bp}\}_{E \in \pi_{0}(\overline {X \setminus X_{\Gamma^{\rm im}}})}$ along $f^{-1}_{\Gamma}(X_{\Gamma^{\rm im}} \cap (\bigcup_{E \in \pi_{0}(\overline {X \setminus X_{\Gamma^{\rm im}}})}E))$ and $\{f^{-1}_{E} (X_{\Gamma^{\rm im}} \cap E)\}_{E \in \pi_{0}(\overline {X \setminus X_{\Gamma^{\rm im}}})}$ \tch{in a way} that is compatible with the gluing of $\{X^{\bp}_{\Gamma^{\rm im}}\} \cup \{E^{\bp}\}_{E \in \pi_{0}(\overline {X \setminus X_{\Gamma^{\rm im}}})}$ that gives rise to $X^{\bp}$, then we obtain a Galois multi-admissible covering $$f^{\bp}: Z^{\bp} \migi X^{\bp}$$ over $k$ with Galois group $\mbZ/n\mbZ$. 

Let $\Pi_{X^{\bp}}$, $\Pi_{E^{\bp}}$ be the admissible fundamental groups of $X^{\bp}$, $E^{\bp}$, $E \in \pi_{0}(\overline {X \setminus X_{\Gamma^{\rm im}}})$, respectively. We write $\alpha \in \text{Hom}(\Pi_{X^{\bp}}^{\rm ab}, \mbZ/n\mbZ)$ for an element induced by $f^{\bp}$ such that the compositions of the natural homomorphisms $ \Pi^{\rm ab}_{X^{\bp}_{\Gamma^{\rm im}}} \migi \Pi^{\rm ab}_{X^{\bp}} \overset{\alpha}\migi \mbZ/n\mbZ,$ $ \Pi^{\rm ab}_{E^{\bp}} \migi \Pi^{\rm ab}_{X^{\bp}} \overset{\alpha}\migi \mbZ/n\mbZ, \  E \in \pi_{0}(\overline {X \setminus X_{\Gamma^{\rm im}}}),$ are equal to $\alpha_{\Gamma^{\rm im}}$ and $\alpha_{E}$, respectively. We see  $ \alpha \in \text{Rev}^{\rm adm}_{D}(X^{\bp}) \setminus \{0\}$. By applying Theorem \ref{lem-3-1}, we obtain $$\gamma_{(\alpha, D)}=g_{X}+n_{X}-2.$$ This completes the proof of the theorem.
\end{proof}

\subsection{A stronger version when $n_{X}=3$}
In the remainder of the present section, we prove a slightly stronger version of Theorem \ref{main-them-2} when $n_{X}=3$ (i.e., we prove Theorem \ref{main-them-2} for a certain fixed effective divisor $D \in (\mbZ/n\mbZ)^{\sim}[D_{X}]^{0}$). The stronger version will be used in the proof of reconstructions of field structures associated to inertia subgroups (cf. Section \ref{sec-5-2}). The main result of this subsection is Theorem \ref{main-them-2-3}.

\subsubsection{\bf Settings} We maintain the notation introduced in \ref{defcurves}. Moreover, {\it we suppose that $n_{X}=3$}.

\subsubsection{}
We generalize Lemma \ref{main-lem} to the case of certain fixed ramification divisors when $n_{X}=3$. Let $D_{X}\defeq \{x_{1}, x_{2}, x_{3}\}$, and let $\Gamma\defeq \Gamma_{D_{X}}$ be a minimal quasi-tree associated to $D_{X}$, $$X_{\Gamma}^{\bp}=(X_{\Gamma}, D_{X_{\Gamma}})$$ the pointed stable curve of type $(g_{X_{\Gamma}}, n_{X_{\Gamma}})$ associated to $\Gamma$, and $\Pi_{X^{\bp}_{\Gamma}}$ the admissible fundamental group of $X^{\bp}_{\Gamma}$.  Let $D_{j}  \in \mbZ[D_{X}]\subseteq \mbZ[D_{X_{\Gamma}}]$, $j\in \{1, 2, 3\}$, be an effective divisor on $X_{\Gamma}$ such that the following conditions are satisfied:
\begin{itemize}
\item ${\rm deg}(D_{j})=2(p^{t_{j}}-1)$.

\item ${\rm ord}_{x}(D_{j}) \leq p^{t_{j}}-1$ for each $x \in D_{X}$.

\item $\#\{x\in D_{X} \ | \ {\rm ord}_{x}(D_{j})=p^{t_{j}}-1\}\geq 1$. 
\end{itemize}
Let $n\defeq p^{t}-1\defeq p^{t_{1}+t_{2}+t_{3}}-1$, and let $$D_{\Gamma} \defeq D_{1}+p^{t_{1}}D_{2}+p^{t_{1}+t_{2}}D_{3}\in \mbZ[D_{X}] \subseteq \mbZ[D_{X_{\Gamma}}]$$ be an effective divisor on $X_{\Gamma}$ with degree $2n$.
Then we have the following lemma.

\begin{lemma}\label{main-lem-3}
We maintain the notation introduced above. Suppose that $D_{\Gamma} \in (\mbZ/n\mbZ)^{\sim}[D_{X}]^{0} \subseteq (\mbZ/n\mbZ)^{\sim}[D_{X_{\Gamma}}]^{0}$ (i.e., ${\rm ord}_{x}(D_{\Gamma})<n$ for all $x \in D_{X}$), and that $$n>{\rm max}\{C(g_{X})+1, \#(e(\Gamma_{X^{\bp}}))\}.$$ Then there exists an element $\alpha_{\Gamma} \in {\rm Rev}^{\rm adm}_{D_{\Gamma}}(X_{\Gamma}^{\bp}) \setminus \{0\}$ such that the following holds:
$$\gamma_{(\alpha_{\Gamma}, D_{\Gamma})}
=g_{X_{\Gamma}}+1.$$
\end{lemma}

\begin{proof}
Since $\Gamma$ is a minimal quasi-tree associated to $D_{X}$, we obtain that $\Gamma'\defeq \Gamma\setminus e^{\rm lp}(\Gamma)$ is a tree. If $D_{X} \subseteq X_{v}$ for some $v \in v(\Gamma)$, then the lemma follows from Proposition \ref{pro-4-2}. Without loss of generality, we may assume that one of the following conditions holds:  
\begin{quote}
(i) Let $w_{1}$, $w_{2}\in v(\Gamma)$ be vertices distinct from each other such that $x_{1}$, $x_{2} \in X_{w_{1}}$, and $x_{3} \in X_{w_{2}}$ (see Example \ref{exm-5-6} (i) below). 

\noindent
(ii) Let $v_{1}$, $v_{2}$, $v_{3} \in v(\Gamma)$  be vertices distinct from each other such that $x_{1}\in X_{v_{1}}$, $x_{2} \in X_{v_{2}}$, and $x_{3} \in X_{v_{3}}$ (see Example \ref{exm-5-6} (ii) below). 
\end{quote}


We put $$D'_{v}\defeq D_{X} \cap X_{v}, \  D_{v}\defeq D'_{v} \cup (X_{v} \cap (\overline {X_{\Gamma} \setminus X_{v}})), \ v \in v(\Gamma)$$ where $\overline {X_{\Gamma} \setminus X_{v}}$ denotes the topological closure of $X_{\Gamma} \setminus X_{v}$ in $X_{\Gamma}$. Next, we construct an effective divisor $P_{v}  \in (\mbZ/n\mbZ)^{\sim}[D_{v}]^{0}$ on $X_{v}$ for all $v\in v(\Gamma)$.

Suppose that (i) holds. Let $y_{w_{1}, w_{2}} \in D_{w_{1}}$ be the unique node of $X_{\Gamma}$ such that the closed edge of $\Gamma'$ corresponding to $y_{w_{1}, w_{2}}$ is contained in the minimal path $\delta(\Gamma', w_{1}, w_{2})$ (\ref{path}) and $z_{w_{1}, w_{2}} \in D_{w_{2}}$ the unique node of $X_{\Gamma}$ such that the closed edge of $\Gamma'$ corresponding to $z_{w_{1}, w_{2}}$ is contained in the minimal path $\delta(\Gamma', w_{1}, w_{2})$. On the other hand, let $w \in v(\Gamma) \setminus \{w_{1}, w_{2}\}$ be an arbitrary vertex (possibly empty). Note that $n_{X}=3$ implies that $w$ is contained in $\delta(\Gamma', w_{1}, w_{2})$. Since $\Gamma'$ is a tree, there exist a unique node $x_{w_{1}, w}\in D_{w}$ and a unique node $x_{w_{2}, w} \in D_{w}$ such that the closed edges of $\Gamma'$ corresponding to $x_{w_{1}, w}$ and $x_{w_{2}, w}$ are contained in $\delta(\Gamma', w_{1}, w)$ and $\delta(\Gamma', w_{2}, w)$, respectively. 

Let $j \in \{1, 2, 3\}$. We put 
$$Q_{w_{1}, j}  \defeq \text{ord}_{x_{1}}(D_{j})x_{1}+\text{ord}_{x_{2}}(D_{j})x_{2}+\text{ord}_{x_{3}}(D_{j})y_{w_{1}, w_{2}},$$
$$Q_{w_{2}, j} \defeq [\text{deg}(D_{j})-\text{ord}_{x_{1}}(D_{j})-\text{ord}_{x_{2}}(D_{j})]z_{w_{1}, w_{2}}+\text{ord}_{x_{3}}(D_{j})x_{3},$$
$$Q_{w, j} \defeq [\text{deg}(D_{j})-\text{ord}_{x_{1}}(D_{j})-\text{ord}_{x_{2}}(D_{j})]x_{w_{1}, w}+\text{ord}_{x_{3}}(D_{j}) x_{w_{2}, w}, \ w \in v(\Gamma)\setminus \{w_{1}, w_{2}\},$$ where $[(-)]$ denotes the image of $(-)$ in $\mbZ/(p^{t_{j}}-1)\mbZ$.  Then $Q_{v, j}$, $v \in v(\Gamma)$, is an effective divisor on $X_{v}$ whose degree is equal to $(\#(D_{v})-1)(p^{t_{j}}-1)$. Moreover, we put $$P_{v} \defeq Q_{v, 1}+p^{t_{1}}Q_{v, 2}+p^{t_{1}+t_{2}}Q_{v, 3} \in (\mbZ/n\mbZ)^{\sim}[D_{v}]^{0}, \ v\in v(\Gamma).$$ Then $P_{v}$ is an effective divisor on $X_{v}$ whose degree is equal to $(\#(D_{v})-1)n$, and whose support is equal to $D_{v}$.

Suppose that (ii) holds. Then one of the following conditions is satisfied: 
\begin{quote}
(1) There exist $a, b, c \in \{1, 2, 3\}$ distinct from each other such that $\delta(\Gamma', v_{a}, v_{b}) \cap \delta(\Gamma', v_{b}, v_{c}) \cap e^{\rm cl}(\Gamma)=\emptyset$ (i.e., $\delta(\Gamma', v_{a}, v_{c})=\delta(\Gamma', v_{a}, v_{b}) \cup \delta(\Gamma', v_{b}, v_{c})$) (see Example \ref{exm-5-6} (ii)-(1) below). 

\noindent
(2) For all $a, b, c \in \{1, 2, 3\}$ distinct from each other, we have that  $\delta(\Gamma', v_{a}, v_{b}) \cap \delta(\Gamma', v_{b}, v_{c})\cap e^{\rm cl}(\Gamma)\neq \emptyset$ (see Example \ref{exm-5-6} (ii)-(2) below).
\end{quote}

Suppose that (1) holds. Without loss of generality, we may assume that $a=1$, $b=2$, and $c=3$. Note that $\delta(\Gamma', v_{1}, v_{3})=\delta(\Gamma', v_{1}, v_{2})\cup \delta(\Gamma', v_{2}, v_{3})$. Write $y_{v_{1}, v_{3}} \in D_{v_{1}}$ for the unique node of $X_{\Gamma}$ such that the closed edge of $\Gamma'$ corresponding to $y_{v_{1}, v_{3}}$ is contained in $\delta(\Gamma', v_{1}, v_{3})$ and $z_{v_{1}, v_{3}} \in D_{v_{3}}$ for the unique node of $X_{\Gamma}$ such that the closed edge of $\Gamma'$ corresponding to $z_{v_{1}, v_{3}}$ is contained in $\delta(\Gamma', v_{1}, v_{3})$. On the other hand, let $v \in v(\Gamma) \setminus\{v_{1}, v_{3}\}$ be an arbitrary vertex. Since $n_{X}=3$, we see that $v \in \delta(\Gamma', v_{1}, v_{3})$, and that either $v \in \delta(\Gamma', v_{1}, v_{2})$ or $v \in \delta(\Gamma', v_{2}, v_{3})$ holds. Since $\Gamma'$ is a tree, there exist a unique node $x_{v_{1}, v}\in D_{v}$ and a unique node $x_{v_{3}, v} \in D_{v}$ such that the closed edges of $\Gamma'$ corresponding to $x_{v_{1}, v}$ and $x_{v_{3}, v}$ are contained in $\delta(\Gamma', v_{1}, v)$ and $\delta(\Gamma', v, v_{3})$, respectively.  

Let $j\in \{1, 2, 3\}$. We put 
$$Q_{v_{1}, j} \defeq \text{ord}_{x_{1}}(D_{j})x_{1}+[\text{deg}(D_{j})-\text{ord}_{x_{2}}(D_{j})-\text{ord}_{x_{3}}(D_{j})]y_{v_{1}, v_{3}},$$ 
$$Q_{v_{2}, j}\defeq \text{ord}_{x_{1}}(D_{j})x_{v_{1}, v_{2}}+\text{ord}_{x_{2}}(D_{j})x_{2}+\text{ord}_{x_{3}}(D_{j})x_{v_{3}, v_{2}}$$
$$Q_{v_{3}, j}\defeq [\text{deg}(D_{j})-\text{ord}_{x_{1}}(D_{j})-\text{ord}_{x_{2}}(D_{j})]z_{v_{1}, v_{3}}+\text{ord}_{x_{3}}(D_{j})x_{3}$$
$$Q_{v, j} \defeq \text{ord}_{x_{1}}(D_{j})x_{v_{1}, v}+[\text{deg}(D_{j})-\text{ord}_{x_{2}}(D_{j})-\text{ord}_{x_{3}}(D_{j})]x_{v_{3}, v}$$ for all $v\in (v(\Gamma) \cap \delta(\Gamma', v_{1}, v_{2})) \setminus \{v_{1}, v_{2}\}$,
and 
$$Q_{v, j}\defeq [\text{deg}(D_{j})-\text{ord}_{x_{1}}(D_{j})-\text{ord}_{x_{2}}(D_{j})]x_{v_{1}, v}+\text{ord}_{x_{3}}(D_{j})x_{v_{3}, v}$$  for all $v\in (v(\Gamma) \cap \delta(\Gamma', v_{2}, v_{3})) \setminus \{v_{2}, v_{3}\}$. Then $Q_{v, j}$, $v \in v(\Gamma)$, is an effective divisor on $X_{v}$ whose degree is equal to $(\#(D_{v})-1)(p^{t_{j}}-1)$. Moreover, we put $$P_{v} \defeq Q_{v, 1}+p^{t_{1}}Q_{v, 2}+p^{t_{1}+t_{2}}Q_{v, 3} \in (\mbZ/n\mbZ)^{\sim}[D_{v}]^{0}, \ v\in v(\Gamma).$$ Then $P_{v}$ is an effective divisor on $X_{v}$ whose degree is equal to $(\#(D_{v})-1)n$, and whose support is equal to $D_{v}$.

Suppose that (2) holds. Then we have
$$\{v_{0}\}= v(\Gamma) \cap \delta(\Gamma', v_{1}, v_{2}) \cap \delta(\Gamma', v_{2}, v_{3}) \cap \delta(\Gamma', v_{3}, v_{1}).$$ Let $v\in v(\Gamma)$. Since $n_{X}=3$, we obtain that either $v\in \delta(\Gamma', v_{1}, v_{0})$ or $v\in\delta(\Gamma', v_{2}, v_{0})$ or $v\in \delta(\Gamma', v_{3}, v_{0})$ holds. Let $y_{v_{i}, v_{0}} \in D_{v_{i}}$, $i\in \{1, 2, 3\}$, be the unique node of $X_{\Gamma}$ such that the closed edge of $\Gamma'$ corresponding to $y_{v_{i}, v_{0}}$ is contained in $\delta(\Gamma', v_{i}, v_{0})$ and $z_{v_{i}, v_{0}}\in D_{v_{0}}$, $i\in \{1, 2, 3\}$, be the unique node of $X_{\Gamma}$ such that the closed edge of $\Gamma'$ corresponding to $z_{v_{i}, v_{0}}$ is contained in $\delta(\Gamma', v_{i}, v_{0})$. Moreover, let $v\in (v(\Gamma) \cap \delta(\Gamma', v_{i}, v_{0}))\setminus \{v_{i}, v_{0}\}$, $i\in \{1, 2, 3\}$. Since $\Gamma'$ is a tree, there exist a unique node $x_{v_{i}, v} \in D_{v}$ and a unique node $x_{v_{0}, v} \in D_{v}$ such that the closed edges of $\Gamma'$ corresponding to $x_{v_{i}, v}$ and $x_{v_{0}, v}$ are contained in $\delta(\Gamma', v_{i}, v)$ and $\delta(\Gamma', v_{0}, v)$, respectively.

Let $j\in \{1, 2, 3\}$. We put 
$$Q_{v_{1}, j} \defeq \text{ord}_{x_{1}}(D_{j})x_{1}+[\text{deg}(D_{j})-\text{ord}_{x_{2}}(D_{j})-\text{ord}_{x_{3}}(D_{j})]y_{v_{1}, v_{0}},$$
$$Q_{v_{2}, j} \defeq \text{ord}_{x_{2}}(D_{j})x_{2}+[\text{deg}(D_{j})-\text{ord}_{x_{1}}(D_{j})-\text{ord}_{x_{3}}(D_{j})]y_{v_{2}, v_{0}},$$
$$Q_{v_{3}, j} \defeq \text{ord}_{x_{3}}(D_{j})x_{3}+[\text{deg}(D_{j})-\text{ord}_{x_{1}}(D_{j})-\text{ord}_{x_{2}}(D_{j})]y_{v_{3}, v_{0}},$$
$$Q_{v_{0}, j} \defeq \text{ord}_{x_{1}}(D_{j})z_{v_{1}, v_{0}}+ \text{ord}_{x_{2}}(D_{j})z_{v_{2}, v_{0}}+ \text{ord}_{x_{3}}(D_{j})z_{v_{3}, v_{0}}$$
$$Q_{v, j} \defeq \text{ord}_{x_{1}}(D_{j})x_{v_{1}, v}+ [\text{deg}(D_{j})-\text{ord}_{x_{2}}(D_{j})-\text{ord}_{x_{3}}(D_{j})]x_{v_{0}, v}$$ for all $v\in (v(\Gamma) \cap \delta(\Gamma', v_{1}, v_{0}))\setminus \{v_{1}, v_{0}\}$,  
$$Q_{v, j} \defeq \text{ord}_{x_{2}}(D_{j})x_{v_{2}, v}+ [\text{deg}(D_{j})-\text{ord}_{x_{1}}(D_{j})-\text{ord}_{x_{3}}(D_{j})]x_{v_{0}, v}$$ for all $v\in (v(\Gamma) \cap \delta(\Gamma', v_{2}, v_{0}))\setminus \{v_{2}, v_{0}\}$,
and
$$Q_{v, j} \defeq \text{ord}_{x_{3}}(D_{j})x_{v_{3}, v}+ [\text{deg}(D_{j})-\text{ord}_{x_{1}}(D_{j})-\text{ord}_{x_{2}}(D_{j})]x_{v_{0}, v}$$ for all $v\in (v(\Gamma) \cap \delta(\Gamma', v_{3}, v_{0}))\setminus \{v_{3}, v_{0}\}$. Then $Q_{v, j}$, $v \in v(\Gamma)$, is an effective divisor on $X_{v}$ whose degree is equal to $(\#(D_{v})-1)(p^{t_{j}}-1)$. Moreover, we put $$P_{v} \defeq Q_{v, 1}+p^{t_{1}}Q_{v, 2}+p^{t_{1}+t_{2}}Q_{v, 3} \in (\mbZ/n\mbZ)^{\sim}[D_{v}]^{0}, \ v\in v(\Gamma).$$ Then $P_{v}$ is an effective divisor on $X_{v}$ whose degree is equal to $(\#(D_{v})-1)n$, and whose support is equal to $D_{v}$.

Let $v\in v(\Gamma)$, and let $\widetilde X^{\bp}_{v}$ be the smooth pointed stable curve of type $(g_{v}, n_{v})$ over $k$ associated to $v$ (\ref{smoothpointed}). Since $D_{X}$ is contained in the smooth locus of $X_{v}$, $P_{v}$ can be also regarded as an effective divisor on $\widetilde X_{v}$. By applying similar arguments to the arguments given in the proof of Proposition \ref{pro-4-2}, there exists $\widetilde \alpha_{v} \in \text{Rev}^{\rm adm}_{P_{v}}(\widetilde X_{v}^{\bp})$ such that 
$$\gamma_{(\widetilde \alpha_{v}, P_{v})}=g_{v}+\#(D_{v})-2.$$
Then Theorem \ref{lem-3-1} implies that the element $\alpha_{v} \in \text{Rev}^{\rm adm}_{P_{v}}(X_{v}^{\bp})$ induced by $\widetilde \alpha_{v}$ such that the following holds: $$\gamma_{(\alpha_{v}, P_{v})}=g_{X_{v}}+\#(D_{v})-2,$$
where $g_{X_{v}}$ denotes the genus of $X_{v}$. Write $f_{v}^{\bp}: Y_{v}^{\bp} \migi X_{v}^{\bp}$ for the Galois multi-admissible covering over $k$ with Galois group $\mbZ/n\mbZ$ induced by $\alpha_{v}$.

By applying similar arguments to the arguments given in Step 3 of the proof of Lemma \ref{main-lem}, we may glue $\{Y_{v}^{\bp}\}_{v\in v(\Gamma)}$ along $\{f^{-1}_{v}(D_{v} \setminus D'_{v})\}_{v\in v(\Gamma)}$ (as admissible coverings) in a way that is compatible with the gluing of $\{X_{v}^{\bp}\}_{v\in v(\Gamma)}$ that gives rise to $X_{\Gamma}^{\bp}$. Then we obtain a Galois multi-admissible covering $$f_{\Gamma}^{\bp}: Y_{\Gamma}^{\bp} \migi X_{\Gamma}^{\bp}$$ over $k$ with Galois group $\mbZ/n\mbZ$. Note that the construction of $f_{\Gamma}^{\bp}$ implies that $f_{\Gamma}$ is \'etale over $D_{X_\Gamma} \setminus D_{X}$.  Moreover, there exists an element $\alpha_{\Gamma} \in {\rm Rev}^{\rm adm}_{D_{\Gamma}}(X_{\Gamma}^{\bp}) \setminus \{0\}$ induced by $f^{\bp}_{\Gamma}$ such that the following holds: $$\gamma_{(\alpha_{\Gamma}, D_{\Gamma})}=g_{X_{\Gamma}}+1.$$ We complete the proof of the lemma.
\end{proof}


\begin{example}\label{exm-5-6}
We give some examples to explain the conditions considered in the proof of Lemma \ref{main-lem-3}. We maintain the notation introduced in the proof of Lemma \ref{main-lem-3}. Moreover, we suppose that $\Gamma_{X^{\bp}}=\Gamma=\Gamma'$ is a tree.

(i) Suppose that $v(\Gamma)=\{w_{1}, w_{2}\}$. Then we have

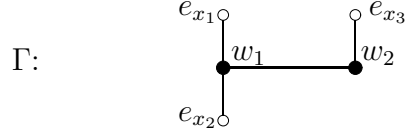
\begin{picture}(300,100)
\put(170,60){\circle*{5}}
\put(173,63){$w_{1}$}
\put(170,60){\line(1,0){50}}
\put(170,60){\line(0,-9){18}}
\put(170,40){\circle{4}}
\put(170,78){\line(0,-9){18}}
\put(170,80){\circle{4}}
\put(220,60){\circle*{5}}
\put(222,63){$w_{2}$}
\put(220,78){\line(0,-9){18}}
\put(220,80){\circle{4}}
\put(153, 80){$e_{x_{1}}$}
\put(153, 40){$e_{x_{2}}$}
\put(225, 80){$e_{x_{3}}$}

\put(90,59){$\Gamma$:}
\end{picture}

(ii)-(1) Suppose that $v(\Gamma)=\{v_{1}, v_{2}, v_{3}\}$. Then we have

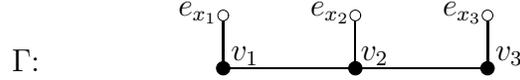
\begin{picture}(300,100)
\put(170,60){\circle*{5}}
\put(173,63){$v_{1}$}
\put(170,60){\line(1,0){50}}
\put(170,78){\line(0,-9){18}}
\put(170,80){\circle{4}}
\put(220,60){\circle*{5}}
\put(222,63){$v_{2}$}
\put(220,78){\line(0,-9){18}}
\put(220,80){\circle{4}}
\put(270,60){\circle*{5}}
\put(273,63){$v_{3}$}
\put(220,60){\line(1,0){50}}
\put(270,80){\circle{4}}
\put(270,78){\line(0,-9){18}}
\put(153, 80){$e_{x_{1}}$}
\put(203, 80){$e_{x_{2}}$}
\put(253, 80){$e_{x_{3}}$}

\put(90,59){$\Gamma$:}
\end{picture}

(ii)-(2) Suppose that $v(\Gamma)=\{v_{0}, v_{1}, v_{2}, v_{3}\}$. Then we have

\begin{picture}(300,100)
\put(170,60){\circle*{5}}
\put(173,63){$v_{1}$}
\put(170,60){\line(1,0){50}}
\put(170,78){\line(0,-9){18}}
\put(170,80){\circle{4}}
\put(220,60){\circle*{5}}
\put(222,63){$v_{0}$}
\put(220,60){\line(0,-1){30}}
\put(220,30){\circle*{5} $v_{3}$}
\put(220,10){\circle{4}}
\put(220,30){\line(0,-9){18}}
\put(270,60){\circle*{5}}
\put(273,63){$v_{2}$}
\put(220,60){\line(1,0){50}}
\put(270,80){\circle{4}}
\put(270,78){\line(0,-9){18}}
\put(153, 80){$e_{x_{1}}$}
\put(253, 80){$e_{x_{2}}$}
\put(203, 10){$e_{x_{3}}$}

\put(90,59){$\Gamma$:}
\end{picture}
\end{example}

\subsubsection{} We generalize Theorem \ref{main-them-2} to the case of certain fixed ramification divisors as follows:

\begin{theorem}\label{main-them-2-3}
Let $X^{\bp}$ be an arbitrary pointed stable curve of type $(g_{X}, 3)$ over an algebraically closed field $k$ of characteristic $p>0$. Let $D_{X}\defeq \{x_{1}, x_{2}, x_{3}\}$, and let $D_{j}  \in \mbZ[D_{X}]$, $j\in \{1, 2, 3\}$, be an effective divisor on $X$ such that the following conditions are satisfied:
\begin{itemize}
\item ${\rm deg}(D_{j})=2(p^{t_{j}}-1)$.

\item ${\rm ord}_{x}(D_{j}) \leq p^{t_{j}}-1$ for each $x \in D_{X}$.

\item $\#\{x\in D_{X} \ | \ {\rm ord}_{x}(D_{j})=p^{t_{j}}-1\}\geq 1$. 
\end{itemize}
Let $n=p^{t}-1\defeq p^{t_{1}+t_{2}+t_{3}}-1$, and let $D \defeq D_{1}+p^{t_{1}}D_{2}+p^{t_{1}+t_{2}}D_{3}\in \mbZ[D_{X}]$ be an effective divisor on $X$ with degree $2n$. Moreover, suppose that $D\in (\mbZ/n\mbZ)^{\sim}[D_{X}]^{0}$ (\ref{2.2.4}), and that $$n>{\rm max}\{C(g_{X})+1, \#(e(\Gamma_{X^{\bp}}))\}.$$ Then there exists an element $\alpha \in {\rm Rev}^{\rm adm}_{D}(X^{\bp})\setminus \{0\}$ (Definition \ref{def-2-4} (i)) such that the following holds:
$$\gamma_{(\alpha, D)}
=g_{X}+1.$$
\end{theorem}

\begin{proof}
By applying Lemma \ref{main-lem-3} and similar arguments to the arguments given in the proof of Theorem \ref{main-them-2}, we obtain the theorem.
\end{proof}

\section{Applications to anabelian geometry}\label{sec-appag} In this section, we give some applications of results obtained in previous sections. The main results of the present section are Theorem \ref{formula} and Theorem \ref{fieldstr}.

\subsection{An anabelian formula for topological types} 

In this subsection, by using Theorem \ref{main-them-2}, we prove a group-theoretical formula for the topological type of an arbitrary pointed stable curve over an algebraically closed field of characteristic $p>0$.

\subsubsection{\bf Settings} We maintain the notation introduced in \ref{defcurves}. Moreover, let $\Pi_{X^{\bp}}$ be the admissible fundamental group of $X^{\bp}$.

\subsubsection{}
Let $\Delta$ be an arbitrary profinite group, and let $m \in \mbN$ be a positive natural number and $\ell$ a prime number. We denote by $D_{\ell}(\Delta) \subseteq \Delta$ the topological closure of $[\Delta, \Delta]\Delta^{\ell}$, where $[\Delta, \Delta]$ denotes the commutator subgroup of $\Delta$. Moreover, we define the closed normal subgroup $D^{(m)}_{\ell}(\Delta)$ of $\Delta$ inductively by $D^{(1)}_{\ell}(\Delta)\defeq D_{\ell}(\Delta)$ and \tch{$D_{\ell}^{(i+1)}(\Delta)\defeq D_{\ell}(D^{(i)}_{\ell}(\Delta))$}, $i \in \{1, \dots, m-1\}$. Note that $\#(\Delta/D_{\ell}^{(m)}(\Delta))\leq \infty$ when $\Delta$ is topologically finitely generated. On the other hand, we denote by $F_{r, m}^{\ell}$ the finite group $\widehat F_{r}/D_{\ell}^{(m)}(\widehat F_{r})$, where $\widehat F_{r}$ denotes the free profinite group of rank $r$.

\subsubsection{}\label{5.1.2}
Let $\Pi$ be {\it an abstract profinite group} which is isomorphic to $\Pi_{X^{\bp}}$ as profinite groups. Moreover, we denote by $\pi_{A}(\Pi)$ the set of finite quotients of $\Pi$. We put $$b^{1}_{\Pi}\defeq \text{max}\{r \ | \ \text{there exists a prime number} \ \ell \ \text{such that} \ (\mbZ/\ell\mbZ)^{\oplus r} \in \pi_{A}(\Pi)\},$$
\begin{eqnarray*}
b^{2}_{\Pi} \defeq \left\{ \begin{array}{ll}
0, & F_{b_{\Pi}^{1}, 2}^{\ell} \in \pi_{A}(\Pi) \ \text{for some prime number} \ \ell,
\\
1, & \text{otherwise}.
\end{array} \right.
\end{eqnarray*}
Note that $b^{i}_{\Pi}, \ i\in \{1, 2\}$, is a group-theoretical invariant associated to $\Pi$ (i.e., depends only on the isomorphism class of $\Pi$). Firstly, we have the following lemma.

\begin{lemma}\label{lem-5-2}
(i) We have $$b^{1}_{\Pi}=2g_{X}+n_{X}-1+b^{2}_{\Pi},$$ 
\begin{eqnarray*}
b^{2}_{\Pi} = \left\{ \begin{array}{ll}
1, & \text{if} \ n_{X}=0,
\\
0, & \text{if} \ n_{X} \neq 0.
\end{array} \right.
\end{eqnarray*}

(ii) There exists a unique prime number $p_{\Pi}$ such that $(\mbZ/p_{\Pi}\mbZ)^{\oplus b^{1}_{\Pi}} \not\in \pi_{A}(\Pi)$. In particular, we have $p=p_{\Pi}$.
\end{lemma}

\begin{proof}
(i) We put $r_{\Pi}\defeq \text{dim}_{\mbF_{\ell}}(\Pi^{\rm ab}\otimes\mbF_{\ell})$, where $\ell$ is an arbitrary prime number $\mfP \setminus \{p\}$, and $\mfP$ denotes the set of prime numbers. Then the structures of maximal prime-to-$p$ quotients of admissible fundamental groups (\ref{stradm}) imply that (see \ref{def-1-3} for $\sigma_{X}$) $$\Pi^{\rm ab} \cong \mbZ_{p}^{\sigma_{X}}\times \prod_{\ell \in \mfP\setminus \{p\}} \mbZ_{\ell}^{r_{\Pi}}.$$ Since $X^{\bp}$ is a pointed stable curve, we have that $\sigma_{X} < r_{\Pi}.$ This implies that $b^{1}_{\Pi}=r_{\Pi}$. Moreover, we have 
\begin{eqnarray*}
b^{1}_{\Pi}= \left\{ \begin{array}{ll}
2g_{X}, & \text{if} \ n_{X}=0,
\\
2g_{X}+n_{X}-1, & \text{if} \ n_{X} \neq 0.
\end{array} \right.
\end{eqnarray*}

Suppose that $n_{X}>0$. Let $\ell_{1} \in \mfP \setminus \{p\}$. The specialization theorem of maximal \tch{pro-$\ell_{1}$} quotients of admissible fundamental groups (\cite[Th\'eor\`eme 2.2 (c)]{V}) implies that the maximal pro-$\ell_{1}$ quotient $\Pi^{\ell_{1}}$ of $\Pi$ is a free pro-$\ell_{1}$ profinite group of rank $b_{\Pi}^{1}$. Then we have $F^{\ell_{1}}_{b^{1}_{\Pi}, 2} \in \pi_{A}(\Pi).$ Thus, we obtain $b_{\Pi}^{2}=0$ if $n_{X}>0$. 

Conversely, we assume that $F^{\ell_{2}}_{b^{1}_{\Pi}, 2} \in \pi_{A}(\Pi)$ for some prime number $\ell_{2}$. Then we have $\ell_{2} \neq p$. The Schreier index formula (\cite[Chapter I \S3 3.4 Corollary]{S2}) implies the following natural exact sequence $$1 \migi (\mbZ/\ell_{2}\mbZ)^{\oplus \ell_{2}^{b_{\Pi}^{1}}(b_{\Pi}^{1}-1)+1} \migi F_{b_{\Pi}^{1}, 2}^{\ell_{2}} \migi(\mbZ/\ell_{2}\mbZ)^{\oplus b_{\Pi}^{1}} \migi 1.$$ Let $\phi: \Pi \migisurj F^{\ell_{2}}_{b^{1}_{\Pi}, 2}$ be a surjection. We denote by $X_{\ell_{2}}^{\bp}$ the pointed stable curve over $k$ corresponding to the kernel of the natural surjection $\Pi_{X^{\bp}} \isom \Pi \overset{\phi}\migisurj F^{\ell_{2}}_{b^{1}_{\Pi}, 2} \migisurj (\mbZ/\ell_{2}\mbZ)^{\oplus b_{\Pi}^{1}}$ and by $\Pi_{\ell_{2}} \subseteq \Pi$ the kernel of the surjection $\Pi \overset{\phi}\migisurj F^{\ell_{2}}_{b^{1}_{\Pi}, 2} \migisurj (\mbZ/\ell_{2}\mbZ)^{\oplus b_{\Pi}^{1}}$. Then we have $(\mbZ/\ell_{2}\mbZ)^{\oplus \ell_{2}^{b_{\Pi}^{1}}(b_{\Pi}^{1}-1)+1} \in \pi_{A}(\Pi_{\ell_{2}}).$ This implies $b^{1}_{\Pi_{\ell_{2}}} \geq \ell_{2}^{b_{\Pi}^{1}}(b_{\Pi}^{1}-1)+1$. If $n_{X}=0$, the Riemann-Hurwitz formula implies $g_{X_{\ell_{2}}}=\ell_{2}^{b_{\Pi}^{1}}(g_{X}-1)+1,$ where $g_{X_{\ell_{2}}}$ denotes the genus of $X_{\ell_{2}}^{\bp}$. Then we have $b_{\Pi_{\ell_{2}}}^{1}=2(\ell_{2}^{b_{\Pi}^{1}}(g_{X}-1)+1)=\ell_{2}^{b_{\Pi}^{1}}(b_{\Pi}^{1}-2)+2.$ On the other hand, $\ell_{2}^{b_{\Pi}^{1}}(b_{\Pi}^{1}-2)+2 < \ell_{2}^{b_{\Pi}^{1}}(b_{\Pi}^{1}-1)+1.$ This \tch{contradicts} the fact that $b^{1}_{\Pi_{\ell_{2}}} \geq \ell_{2}^{b_{\Pi}^{1}}(b_{\Pi}^{1}-1)+1$. Then we obtain $n_{X}>0$ if  $b_{\Pi}^{2}=0$. Moreover, we see  $$b^{1}_{\Pi}=2g_{X}+n_{X}-1+b^{2}_{\Pi}.$$ 

(ii) This follows immediately from the structure of $\Pi^{\rm ab}$. This completes the proof of the lemma.
\end{proof}

\subsubsection{}
Let $\overline \mbF_{p_{\Pi}}$ be an arbitrary algebraic closure of $\mbF_{p_{\Pi}}$. Let $\chi \in \text{Hom}(\Pi, \overline \mbF_{p_{\Pi}}^{\times})$. We denote by $\Pi_{\chi} \subseteq \Pi$ the kernel of $\chi$. The profinite group $\Pi_{\chi}$ admits a natural action of $\Pi$ via the conjugation action. We put
$$N_{\chi}\defeq\{\pi \in H_{\chi, p_{\Pi}}\defeq \text{Hom}(\Pi_{\chi}, \mbZ/p_{\Pi}\mbZ)\otimes_{\mbF_{p_{\Pi}}} \overline \mbF_{p_{\Pi}} \ | \ \tau\cdot\pi=\chi(\tau)\pi \ \text{for all} \ \tau \in \Pi\},$$ $$\gamma_{N_{\chi}} \defeq \text{dim}_{\overline \mbF_{p_{\Pi}}}(N_{\chi}),$$ where $(\tau\cdot\pi)(x)\defeq \pi(\tau^{-1}\cdot x)$ for all $x \in \Pi_{\chi}$.  We define a group-theoretical invariant associated to $\Pi$ as follows: $$\gamma_{\Pi}^{\rm max} \defeq \text{max}\{ \gamma_{N_\chi} \ | \ \chi \in \text{Hom}(\Pi, \overline \mbF_{p_{\Pi}}^{\times}) \ \text{and} \ \chi \neq 1\}.$$ 

Let $\mu_{m}\defeq \chi(\Pi) \subseteq \overline \mbF_{p_{\Pi}}^{\times}$ be the image of $\chi$ which is the group of $m$th roots of unity for some $m$ prime to $p_{\Pi}$, and $X^{\bp}_{\chi}=(X_{\chi}, D_{X_{\chi}}) \migi X^{\bp}$ the Galois multi-admissible covering over $k$ with Galois group $\mu_{m}$ induced by $\chi$. Then we have a natural $\Pi$-equivariant isomorphism $$H^{1}_{\et}(X_{\chi}, \mbF_{p_{\Pi}})\otimes_{\mbF_{p_{\Pi}}} \overline \mbF_{p_{\Pi}} \cong H_{\chi, p_{\Pi}}.$$ Moreover, since the actions of $\Pi$ on  $H^{1}_{\et}(X_{\chi}, \mbF_{p_{\Pi}})\otimes_{\mbF_{p_{\Pi}}} \overline \mbF_{p_{\Pi}}$ and $H_{\chi, p_{\Pi}}$ factor through $ \Pi/\Pi_{\chi} \cong \mu_{m}$, the isomorphism above is also a $\mu_{m}$-equivariant. Namely, $\gamma_{N_\chi}$ is a generalized Hasse-Witt invariant of $\mbZ/m\mbZ$-cyclic admissible coverings of $X^{\bp}$ (\ref{ghw}). Then we have the following result.

\begin{lemma}\label{lem-5-3}
We maintain the notation introduced above. Then we have $\gamma_{\Pi}^{\rm max}=\gamma_{X^{\bp}}^{\rm max}$ (see Definition \ref{def-3-2} (i) for $\gamma_{\Pi}^{\rm max}$). In particular, we have $$\gamma_{\Pi}^{\rm max}=g_{X}+n_{X}-2+b_{\Pi}^{2}.$$
\end{lemma}

\begin{proof}
The first part of the lemma follows from the above explanation concerning $\gamma_{N_\chi}$. The ``in particular" part of the lemma follows immediately from Theorem \ref{main-them-2} and Lemma \ref{lem-5-2} (i).
\end{proof}

\subsubsection{}
We have the following anabelian formula for $(g_{X}, n_{X})$.

\begin{theorem}\label{formula}
Let $X^{\bp}$ be an arbitrary pointed stable curve of type $(g_{X}, n_{X})$ over an algebraically closed field $k$ of characteristic $p>0$, $\Pi_{X^{\bp}}$ the admissible fundamental group of $X^{\bp}$, and $\Pi$ an abstract profinite group such that $\Pi \cong \Pi_{X^{\bp}}$ as profinite groups. Then we have  $$g_{X}=b^{1}_{\Pi}-\gamma_{\Pi}^{\rm max}-1,$$ $$n_{X}=2\gamma^{\rm max}_{\Pi}-b_{\Pi}^{1}-b_{\Pi}^{2}+3.$$ In particular, $g_{X}$ and $n_{X}$ are group-theoretical invariants associated to $\Pi$.
\end{theorem}

\begin{proof}
The theorem follows immediately from Lemma \ref{lem-5-2} and Lemma \ref{lem-5-3}. 
\end{proof}

\subsubsection{}\label{rem-5-4-1}
We maintain the notation introduced above. Moreover, suppose that $X^{\bp}$ is {\it smooth} over $k$. In the remainder of this subsection, we discuss a formula for $(g_{X}, n_{X})$ which was essentially obtained by Tamagawa. Let $n\defeq p^{t}-1$, and let $K_{n}$ be the kernel of the natural surjection $\Pi \migisurj \Pi^{\rm ab}\otimes\mbZ/n\mbZ$. In \cite{T2}, Tamagawa introduced $$\text{Avr}_{p}(\Pi)\defeq \lim_{t\migi \infty}\frac{\text{dim}_{\mbF_{p}}(K_{n}^{\rm ab}\otimes \mbF_{p})}{\#(\Pi^{\rm ab} \otimes \mbZ/n\mbZ)}$$ which is called the limit of $p$-averages associated to $\Pi$. Note that since $p=p_{\Pi}$ (Lemma \ref{lem-5-2} (ii)), we have that $\text{Avr}_{p}(\Pi)$ is a group-theoretical invariant associated to $\Pi$. Then the main theorem of \cite{T2} (i.e., Tamagawa's $p$-average theorem, see \cite[Theorem 0.5]{T2}) is the following:
\begin{eqnarray*}
\text{Avr}_{p}(\Pi)= \left\{ \begin{array}{ll}
g_{X}-1, & \text{if} \ n_{X}\leq 1,
\\
g_{X}, & \text{if} \ n_{X} \geq 2.
\end{array} \right.
\end{eqnarray*}

Let $\ell >>0$ be an arbitrary prime number distinct from $p_{\Pi}$. Write $\text{Nom}_{\ell}(\Pi)$ for the set of normal subgroups of $\Pi$ of index $\ell$. Suppose that $b_{\Pi}^{2}=0$ (i.e., $n_{X}\neq 0$). It is well-known that $$\text{Avr}_{p}(\Pi(\ell))-1=\ell(\text{Avr}_{p}(\Pi))$$ for every $\ell \in \mfP \setminus \{p_{\Pi}\}$ and every $\Pi(\ell) \in \text{Nom}_{\ell}(\Pi)$ if and only if $n_{X}=1$. We may define a group-theoretical invariant associated to $\Pi$ as follows:
\begin{eqnarray*}
c_{\Pi}\defeq \left\{ \begin{array}{ll}
1, & \text{if} \  b_{\Pi}^{2}=1,
\\
1, & \text{if} \ b_{\Pi}^{2}=0, \ \text{Avr}_{p}(\Pi(\ell))-1=\ell(\text{Avr}_{p}(\Pi)) \ \text{for}
\\
&\ \text{all} \ \ell \in \mfP \setminus \{p_{\Pi}\} \ \text{and all} \ \Pi(\ell) \in \text{Nom}_{\ell}(\Pi),
\\
0, & \text{otherwise}.
\end{array} \right.
\end{eqnarray*}
Then Tamagawa's $p$-average theorem implies immediately the following formula: $$g_{X}=\text{Avr}_{p}(\Pi)+c_{\Pi}, \ n_{X}=b_{\Pi}^{1}-2\text{Avr}_{p}(\Pi)-2c_{\Pi}-b^{2}_{\Pi}+1.$$ 
In particular, $g_{X}$ and $n_{X}$ are group-theoretical invariants associated to $\Pi$ when $X^{\bp}$ is smooth over $k$.

Before Tamagawa proved the above result, he also obtained an \'etale fundamental group version formula for $(g_{X}, n_{X})$ in a completely different way (by using wildly
ramified coverings) which is much simpler than the case of tame fundamental groups (see \cite[\S1]{T1}). Note that, for any smooth pointed stable curve over an algebraically
closed field of positive characteristic, since the tame fundamental group can be reconstructed
group-theoretically from the \'etale fundamental group (\cite[Corollary 1.10]{T1}), the tame
fundamental group version is stronger than the \'etale fundamental group version. On the other hand, tame fundamental groups are much better than \'etale fundamental groups if we study anabelian geometry in positive characteristic from the point of view of moduli spaces (e.g. \cite{T3}, \cite{Y2}).

\subsubsection{}\label{singav}
The approach to finding a group-theoretical formula for $(g_{X}, n_{X})$ by applying the limit of $p$-averages associated to $\Pi$ explained above {\it \tch{is difficult to be generalized}} to the case where $X^{\bp}$ is an arbitrary (possibly singular) pointed stable curve. The reason is as follows. In \cite{Y4}, the author generalized Tamagawa's $p$-average theorem to the case of pointed stable curves (\cite[Theorem 1.3 and Theorem 1.4]{Y4}). The generalized formula concerning the limit of $p$-averages associated to $\Pi$ is very complicated in general when $X^{\bp}$ is not smooth over $k$, and $\text{Avr}_{p}(\Pi)$ depends not only on the topological type $(g_{X}, n_{X})$ but also on the structure of dual semi-graph $\Gamma_{X^{\bp}}$.

\subsection{Field structures associated to inertia subgroups of marked points}\label{sec-5-2}

In this subsection, we prove that the field structures associated to inertia subgroups of marked points of arbitrary pointed stable curves can be reconstructed group-theoretically from admissible fundamental groups and inertia subgroups of marked points, and that a surjective open continuous homomorphism of admissible fundamental groups induces a field isomorphism of the fields associated to inertia subgroups of marked points.

\subsubsection{\bf Settings}
Let $i \in \{1, 2\}$, and let $X_{i}^{\bp}$ be a pointed stable curve of type $(g_{X}, n_{X})$ over an algebraically closed field $k_{i}$ of characteristic $p>0$, $\Gamma_{X_{i}^{\bp}}$ the dual semi-graph of $X_{i}^{\bp}$, and $\Pi_{X_{i}^{\bp}}$ the admissible fundamental group of $X_{i}^{\bp}$. Let $\widehat X_{i}^{\bp} =(\widehat X_{i}, D_{\widehat X_{i}})$ be the universal admissible covering of $X_{i}^{\bp}$ corresponding to $\Pi_{X_{i}^{\bp}}$ (\ref{universal}), $\Gamma_{\widehat X_{i}^{\bp}}$ the dual semi-graph of $\widehat X_{i}^{\bp}$,  $e_{i} \in e^{\rm op}(\Gamma_{X_{i}^{\bp}})$ an open edge, and $x_{e_{i}}$ the closed point of $X_{i}$ corresponding to $e_{i}$. Note that we have $\text{Aut}(\widehat X_{i}^{\bp} /X_{i}^{\bp})=\Pi_{X_{i}^{\bp}}$. 

Let $\widehat e_{i} \in e^{\rm op}(\Gamma_{\widehat X_{i}^{\bp}})$ be an open edge over $e_{i}$. We denote by $$I_{\widehat e_{i}} \cong \widehat \mbZ(1)^{p'} \subseteq \Pi_{X_{i}^{\bp}}$$ the stabilizer subgroup of $\widehat e_{i}$. 

\subsubsection{}
Write $\overline \mbF_{p, i}$ for the algebraic closure of $\mbF_{p}$ in $k_{i}$. We put $$\mbF_{\widehat e_{i}}\defeq (I_{\widehat e_{i}}\otimes_{\mbZ} (\mbQ/\mbZ)_{i}^{p'}) \sqcup \{*_{\widehat e_{i}}\},$$ where $\{*_{\widehat e_{i}}\}$ is an one-point set, and $(\mbQ/\mbZ)_{i}^{p'}$ denotes the prime-to-$p$ part of $\mbQ/\mbZ$ which can be canonically identified with $$\bigcup_{(p, m)=1}\mu_{m}(k_{i}).$$ Moreover, $\mbF_{\widehat e_{i}}$ can be identified with $\overline \mbF_{p, i}$ as sets, hence, admits {\it a  structure of field}, whose multiplicative group is $I_{\widehat e_{i}}\otimes_{\mbZ} (\mbQ/\mbZ)^{p'}_{i}$, and whose zero element is $*_{\widehat e_{i}}$. We have the following important result.

\begin{theorem}\label{fieldstr}
We maintain the notation introduced above. Let $\phi: \Pi_{X^{\bp}_{1}} \migisurj \Pi_{X_{2}^{\bp}}$ be an arbitrary surjective open continuous homomorphism of admissible fundamental groups. Suppose that $\phi(I_{\widehat e_{1}})=I_{\widehat e_{2}}$, and that $n_{X}=3$. Then there exists a group-theoretical algorithm whose input data are $\Pi_{X_{i}^{\bp}}$ and $I_{\widehat e_{i}}$, and whose output datum is $\mbF_{\widehat e_{i}}$ (as a field). Moreover, $\phi$ induces a field isomorphism $$\phi^{\rm fd}_{\widehat e_{1}, \widehat e_{2}}:\mbF_{\widehat e_{1}}\isom \mbF_{\widehat e_{2}},$$ where ``fd" means ``field".
\end{theorem}

\begin{proof}
Let $t \in \mbZ_{>0}$. We denote by $\mbF_{p^{t}, \widehat e_{i}}$ the unique subfield of $\mbF_{\widehat e_{i}}$ whose cardinality is equal to $p^{t}$. On the other hand, we fix a finite field $\mbF_{p^{t}}$ of cardinality $p^{t}$ and an algebraic closure $\overline \mbF_{p}$ of $\mbF_{p}$ containing $\mbF_{p^{t}}$. Note that the field structure of $\mbF_{p^{t}, \widehat e_{i}}$ is equivalent to a subset $$\text{Hom}_{\rm fd}(\mbF_{p^{t}, \widehat e_{i}}, \mbF_{p^{t}}) \subseteq \text{Hom}_{\rm gp}(\mbF_{p^{t}, \widehat e_{i}}^{\times}, \mbF_{p^{t}}^{\times}),$$ where ``gp" means ``group", and ``fd" means ``field". Then in order to prove the first part of the theorem, it is sufficient to prove that there exists a group-theoretical algorithm whose input data are $\Pi_{X_{i}^{\bp}}$ and $I_{\widehat e_{i}}$, and whose output datum is the subset $\text{Hom}_{\rm fd}(\mbF_{p^{t}, \widehat e_{i}}, \mbF_{p^{t}})$ for $t$ in a cofinal subset of $\mbN$ with respect to division. 

Let $n\defeq p^{t}-1$ and $\chi_{i} \in \text{Hom}_{\rm gp}(\Pi_{X_{i}^{\bp}}^{\rm ab}\otimes \mbZ/n\mbZ, \mbF^{\times}_{p^{t}}).$ Write $H_{\chi_{i}}$ for the kernel of $\Pi_{X_{i}^{\bp}} \migisurj \Pi_{X_{i}^{\bp}}^{\rm ab}\otimes \mbZ/n\mbZ \overset{\chi_{i}}\migi \mbF_{p^{t}}^{\times}$, $M_{\chi_{i}}$ for $H^{\rm ab}_{\chi_{i}} \otimes \mbF_{p}$, and $X^{\bp}_{H_{\chi_{i}}}=(X_{H_{\chi_{i}}}, D_{X_{H_{\chi_{i}}}})$ for the pointed stable  curve over $k_{i}$ induced by $H_{\chi_{i}}$. Note that $M_{\chi_{i}}$ admits a natural action of $\Pi_{X^{\bp}_{i}}$ via the conjugation action. Moreover, this action factors through $\Pi_{X_{i}^{\bp}}^{\rm ab}\otimes \mbZ/n\mbZ$. We put $$M_{\chi_{i}}[\chi_{i}]\defeq \{a \in M_{\chi_{i}} \otimes_{\mbF_{p}} \overline \mbF_{p} \ | \ \sigma\cdot a=\chi_{i}(\sigma)a \ \text{for all} \ \sigma \in \Pi_{X_{i}^{\bp}}^{\rm ab}\otimes \mbZ/n\mbZ \},$$
$$\gamma_{\chi_{i}}(M_{\chi_{i}})\defeq \text{dim}_{\overline \mbF_{p}}(M_{\chi_{i}}[\chi_{i}]).$$ Note that $\gamma_{\chi_{i}}(M_{\chi_{i}})$ is a generalized Hasse-Witt invariant of a cyclic multi-admissible covering of $X_{i}^{\bp}$ with Galois group $\mbZ/n\mbZ$, and that Lemma \ref{lem-2-4} implies that $\gamma_{\chi_{i}}(M_{\chi_{i}}) \leq g_{X}+1$ if $n_{X}=3$. Moreover, we define two maps $$\text{Res}_{i, t}: \text{Hom}_{\rm gp}(\Pi_{X_{i}^{\bp}}^{\rm ab}\otimes \mbZ/n\mbZ, \mbF^{\times}_{p^{t}}) \migi \text{Hom}_{\rm gp}(\mbF^{\times}_{p^{t}, \widehat e_{i}}, \mbF_{p^{t}}^{\times}),$$ 
 $$\Gamma_{i, t}: \text{Hom}_{\rm gp}(\Pi_{X_{i}^{\bp}}^{\rm ab}\otimes \mbZ/n\mbZ, \mbF^{\times}_{p^{t}}) \migi \mbZ_{\geq 0},  \ \chi_{i}\mapsto\gamma_{\chi_{i}}(M_{\chi_{i}}),$$
where the map $\text{Res}_{i, t}$ is the restriction with respect to the natural inclusion $$\mbF^{\times}_{p^{t}, \widehat e_{i}}=I_{\widehat e_{i}} \otimes\mbZ/n\mbZ \migiinje \Pi_{X_{i}^{\bp}}^{\rm ab}\otimes \mbZ/n\mbZ.$$

Let $m_{0}$ be the product of all prime numbers $\leq p-2$ if $p\neq 2, 3$ and $m_{0}=1$ if $p=2, 3$. Let $t_{0}$ be the order of $p$ in the multiplicative group $(\mbZ/m_{0}\mbZ)^{\times}$. We have the following claim (see \cite[Claim 5.4]{T2} for the case where $X^{\bp}_{i}$ is smooth over $k_{i}$): \\

{\bf Claim A:} There exists a constant $C(g_{X})$ which  depends only on $g_{X}$ such that, for each natural number $t > \text{max}\{\text{log}_{p}(C(g_{X})+1), \log_{p}(\#(e^{\rm cl}(\Gamma_{X_{i}^{\bp}})))\}$ divisible by $t_{0}$, we have $$\text{Hom}_{\rm fd}(\mbF_{p^{t}, \widehat e_{i}}, \mbF_{p^{t}})= \text{Hom}^{\rm surj}_{\rm gp}(\mbF_{p^{t}, \widehat e_{i}}^{\times}, \mbF_{p^{t}}^{\times}) \setminus \text{Res}_{i, t}(\Gamma_{i, t}^{-1}(\{g_{X}+1\})), \ i \in \{1, 2\},$$
where $\text{Hom}^{\rm surj}_{\rm gp}(-, -)$ denotes the set of surjections whose elements are contained in $\text{Hom}_{\rm gp}(-, -)$. \\ 

By applying similar arguments to the arguments given in the proof of \cite[Claim 5.4]{T2}, Claim A is equivalent to the following claim: \\

{\bf Claim B:} Let $m \in (\mbZ/n\mbZ)^{\sim}$ (\ref{2.2.4}). Then the following statements are equivalent:

(i) We have $m \in \{p^{r} \ | \ r=0, \dots, t-1\}$.

(ii) We have that $(m, n)=1$, and that, there does not exist $D\in (\mbZ/n\mbZ)^{\sim}[D_{X}]^{0}$ (\ref{2.2.4}) and $\alpha\in \text{Rev}_{D}^{\rm adm}(X^{\bp}_{i})$ (Definition \ref{def-2-4} (i)) such that $\text{ord}_{x_{e_{i}}}(D)=m$ and $\gamma_{(\alpha, D)}=g_{X}+1$. \\

Let us prove Claim B. Firstly, we prove (i) $\Rightarrow$ (ii). If $s(D)=1$ (\ref{2.2.4}), then Lemma \ref{lem-2-4} implies  $\gamma_{(\alpha, D)} \leq g_{X}$. Thus, we may assume that $s(D)=2$. We put $$D_{X_{i}}\defeq \{x_{i, 1}\defeq x_{e_{i}}, x_{i, 2}, x_{i, 3}\}.$$ By \cite[Proposition 2.21 (iv-a)]{T2}, either $\text{ord}_{x_{i, 1}}(D)=n$ or $\text{ord}_{x_{i, 3}}(D)=n$ holds. This is impossible as $D\in (\mbZ/n\mbZ)^{\sim}[D_{X}]^{0}$. Next, we prove (ii) $\Rightarrow$ (i). Suppose that $m \not\in \{p^{r} \ | \ r=0, \dots, t-1\}$, and that $(m, n)=1$. Since $t$ is divisible by $t_{0}$, $n$ is divisible by all prime numbers $\leq p-2$. Then the assumption $(m, n)=1$ implies that $m \not\in \{ap^{r} \ | \ a=0, \dots, p-2, r =0, \dots, t-1\}$. Then \cite[Proposition 2.21 (iv-c)]{T2} and Theorem \ref{main-them-2-3} imply that there exist $D\in (\mbZ/n\mbZ)^{\sim}[D_{X}]^{0}$ and $\alpha\in \text{Rev}_{D}^{\rm adm}(X^{\bp}_{i})$ such that $\text{ord}_{x_{e_{i}}}(D)=m$ and $\gamma_{(\alpha, D)}=g_{X}+1$. This completes the proof of the claim. Then we complete the proof of the first part of the theorem.

Next, we prove the ``moreover" part of the theorem. Let $\kappa_{2} \in \text{Hom}_{\rm gp}(\Pi_{X^{\bp}_{2}}^{\rm ab}\otimes \mbZ/n\mbZ, \mbF^{\times}_{p^{t}}).$ Then we obtain a character $$\kappa_{1} \in \text{Hom}_{\rm gp}(\Pi_{X^{\bp}_{1}}^{\rm ab}\otimes \mbZ/n\mbZ, \mbF^{\times}_{p^{r}})$$ induced by $\phi$. Moreover, the surjection $\phi|_{H_{\kappa_{1}}}$ induces a surjection $M_{\kappa_{1}}[\kappa_{1}] \migisurj M_{\kappa_{2}}[\kappa_{2}].$  Suppose that $\kappa_{2} \in \Gamma_{2, r}^{-1}(\{g_{X}+1\})$. The surjection $M_{\kappa_{1}}[\kappa_{1}] \migisurj M_{\kappa_{2}}[\kappa_{2}]$ implies that $\gamma_{\kappa_{1}}(M_{\kappa_{1}})=g_{X}+1$. Namely, we have $\kappa_{1} \in \Gamma_{1, r}^{-1}(\{g_{X}+1\})$. On the other hand, the isomorphism $\phi|_{I_{\widehat e_{1}}}: I_{\widehat e_{1}} \isom I_{\widehat e_{2}}$ induces an injection $$\text{Res}_{2, r}(\Gamma_{2, r}^{-1}(\{g_{X}+1\})) \migiinje \text{Res}_{1, r}(\Gamma_{1, r}^{-1}(\{g_{X}+1\})).$$ Since $\#(\text{Hom}_{\rm fd}(\mbF_{p^{r}, \widehat e_{1}}, \mbF_{p^{r}}))=\#(\text{Hom}_{\rm fd}(\mbF_{p^{r}, \widehat e_{2}}, \mbF_{p^{r}}))$, Claim A implies that $\phi|_{I_{\widetilde e_{1}}}$ induces a bijection $\text{Hom}_{\rm fd}(\mbF_{p^{r}, \widehat e_{2}}, \mbF_{p^{r}}) \isom \text{Hom}_{\rm fd}(\mbF_{p^{r}, \widehat e_{1}}, \mbF_{p^{r}}).$ Thus, $\phi|_{I_{\widehat e_{1}}}$ induces a bijection $$\text{Hom}_{\rm fd}(\mbF_{\widehat e_{2}}, \overline \mbF_{p}) \isom \text{Hom}_{\rm fd}(\mbF_{\widehat e_{1}}, \overline \mbF_{p}).$$ If we choose $\overline \mbF_{p}=\mbF_{\widehat e_{2}}$, then the image of $\text{id}_{\mbF_{\widehat e_{2}}}$ via the above bijection induces a field isomorphism $$\phi^{\rm fd}_{\widehat e_{1}, \widehat e_{2}}:\mbF_{\widehat e_{1}}\isom \mbF_{\widehat e_{2}}.$$  This completes the proof of the theorem.
\end{proof}

\begin{remarkA}\label{rem-fstr}
We maintain the notation introduced above. In fact, by applying Theorem \ref{main-them-2}, we can prove that $I_{\widehat e_{i}}$ can be reconstructed group-theoretically from $\Pi_{X_{i}^{\bp}}$, and that $\phi(I_{\widehat e_{1}})$ is always a stabilizer of an open edge of $\Gamma_{\widehat X^{\bp}_{2}}$ over an open edge of $\Gamma_{X^{\bp}_{2}}$ (\cite[Theorem 4.11]{Y5}). Moreover, by applying similar arguments to the arguments given in the proof of \cite[Proposition 6.1]{Y2}, to reconstruct the field structures, we may assume that $n_{X}=3$. Thus, Theorem \ref{fieldstr} implies that  the field structures associated to inertia subgroups of marked points can be reconstructed group-theoretically from surjective open continuous homomorphisms of admissible fundamental groups (\cite[Theorem 4.13]{Y5}).
\end{remarkA}

\markright{ }

\end{document}